\documentclass[a4paper,11pt]{amsart}

\usepackage{amssymb}
\usepackage{amsmath}
\usepackage{tikz}
\usepackage{enumitem}
\usepackage[all,cmtip]{xy}
\newtheorem{thm}{Theorem}[section] 
\newtheorem{lem}[thm]{Lemma} 
\newtheorem{cor}[thm]{Corollary} 
\newtheorem{prop}[thm]{Proposition}
 
\theoremstyle{definition}
\newtheorem*{ack}{Acknowledgments}
\newtheorem{defi}[thm]{Definition} 
\newtheorem{rem}[thm]{Remark}

\DeclareMathOperator{\Res}{Res}

\DeclareMathOperator{\Hom}{Hom}
\DeclareMathOperator{\Ind}{Ind}
\DeclareMathOperator{\Cur}{Cur}
\DeclareMathOperator{\Lie}{Lie}
\DeclareMathOperator{\Sing}{Sing}
\DeclareMathOperator{\Ker}{Ker}
\DeclareMathOperator{\CoKer}{CoKer}
\DeclareMathOperator{\Ima}{Im}
\DeclareMathOperator{\rk}{rk}
\DeclareMathOperator{\End}{End}
\DeclareMathOperator{\I}{I}

\DeclareMathOperator{\Id}{Id}

\DeclareMathOperator{\Gr}{Gr}
\DeclareMathOperator{\size}{size}
\DeclareMathOperator{\ch}{ch}
\DeclareMathOperator{\dime}{dim}
\DeclareMathOperator{\tr}{tr}

\newcommand{\C}{\mathbb{C}}
\newcommand{\Z}{\mathbb{Z}}

\DeclareMathOperator{\degr}{deg}
\newcommand{\inlinewedge}{\textrm{\raisebox{0.6mm}{\footnotesize $\bigwedge$}}}
\newcommand{\displaywedge}{\textrm{\raisebox{0.6mm}{\tiny $\bigwedge$}}}
\newcommand{\g}{\mathfrak {g}}

\newcommand*{\bigchi}{\mbox{\Large$\chi$}}
\allowdisplaybreaks

\title[]{Computation of the homology of the complexes of finite Verma modules for $K'_4$}
\author{Lucia Bagnoli}
\address{Lucia Bagnoli, Dipartimento di matematica, Universit\`a di Bologna, Piazza di Porta San Donato 5, 40126 Bologna, Italy}

\email{luciabagnoli93@gmail.com, lucia.bagnoli4@unibo.it}

\begin{document}
\maketitle

\begin{abstract}
We compute the homology of the complexes of finite Verma modules over the annihilation superalgebra $\mathcal A(K'_{4})$, associated with the conformal superalgebra $K'_{4}$, obtained in \cite{K4}. We use the computation of the homology in order to provide an explicit realization of all the irreducible quotients of finite Verma modules over $\mathcal A(K'_{4})$.
	\end{abstract}

	\section{Introduction} 
	Finite simple conformal superalgebras were completely classified in \cite{fattorikac} and consist of the following list: $\Cur \mathfrak{g}$, where $\mathfrak{g}$ is a simple finite$-$dimensional Lie superalgebra, $W_{n} (n\geq 0)$, $S_{n,b}$, $\tilde{S}_{n}$ $(n\geq 2, \, b \in \mathbb{C})$, $K_{n} (n\geq 0, \, n \neq 4)$, $K'_{4}$, $CK_{6} $. The finite irreducible modules over the conformal superalgebras $\Cur \mathfrak{g}$, $K_{0}$, $K_{1}$ were studied in \cite{chengkac}. Boyallian, Kac, Liberati and Rudakov classified all finite irreducible modules over the conformal superalgebras of type $W$ and $S$ in \cite{bklr}. The finite irreducible modules over the conformal superalgebras $S_{2,0}$, $K_{n}$, for $n =2,3, 4$ were studied in \cite{chenglam}. Boyallian, Kac and Liberati classified all finite irreducible modules over the conformal superalgebras of type $K_{n}$ in \cite{kac1}. A classification of all finite irreducible modules over the conformal superalgebra $CK_{6}$ was obtained in \cite{ck6} and \cite{zm} with different approaches. 
	Finally the classification of all finite irreducible modules over the conformal superalgebra $K'_{4}$ is obtained in \cite{K4}.\\
	In \cite{K4} the classification of all finite irreducible modules over the conformal superalgebra $K'_{4}$ is obtained by their correspondence with irreducible \textit{finite conformal} modules over the annihilation superalgebra $\g:=\mathcal{A}(K'_{4})$ associated with $K'_{4}$. In order to obtain this classification, the authors classify all highest weight singular vectors, i.e. highest weight vectors which are annihilated by $\mathcal A(K'_4)_{>0}$, of finite Verma modules, that are the induced modules $\Ind(F)=U(\g) \otimes _{U(\g_{\geq 0})} F$ such that $F$ is a finite$-$dimensional irreducible $\g_{\geq 0}$-module \cite{kacrudakov,chenglam}. 
	In \cite{K4} the authors show that for $\mathcal A(K'_{4})$ there are four families of singular vectors of degree 1, four families of singular vectors of degree 2 and two singular vectors of degree 3.
	Since the classification of singular vectors of finite Verma modules is equivalent to the classification of morphisms between such modules, in \cite{K4} it is shown that these morphisms can be arranged in an infinite number of bilateral complexes as in Figure \ref{figura}, which is similar to those obtained for the exceptional Lie superalgebras $E(1,6)$, $E(3,6)$, $E(3,8)$ and $E(5,10)$ (see \cite{kacrudakovE36,kacrudakov,kacrudakovE38,E36III, cantacaselliE510, cantacasellikacE510}).\\
	The aim of this work is to compute the homology of the complexes in Figure \ref{figura} and provide an explicit construction of all the irreducible quotients of finite Verma modules over $\mathcal A(K'_{4})$. We compute the homology through the theory of spectral sequences of bicomplexes, using an argument similar to the one used in \cite{kacrudakovE36} for the homology of the complexes of finite Verma modules over $E(3,6)$. We obtain in particular that the complexes in Figure \ref{figura} are exact in each point except for the origin of the first quadrant and the point of coordinates $(1,1)$ in the third quadrant, in which the homology space is isomorphic to the trivial representation. \\
	The computation of the homology allows us to show that all the irreducible quotients of finite Verma modules over $\mathcal A(K'_4)$ occur among cokernels, kernel and images of complexes in Figure \ref{figura}.
	As an application of this result, we compute the size of all the irreducible quotients of finite Verma modules, that is defined following \cite{kacrudakovE36}.\\
	The paper is organized as follows. In section 2 we recall some notions on conformal superalgebras. In section 3 we recall the definition of the conformal superalgebra $K'_4$ and the classification of singular vectors obtained in \cite{K4}. In section 4 we find an explicit expression for the morphisms represented in Figure \ref{figura}. In section 5 we recall the preliminaries on spectral sequences that we need. In section 6 we compute the homology of the complexes in Figure \ref{figura}. Finally in section 7 we compute the size of all the irreducible quotients of finite Verma modules.

\section{Preliminaries on conformal superalgebras}
We recall some notions on conformal superalgebras. For further details see \cite[Chapter 2]{kac1vertex}, \cite{dandrea}, \cite{bklr}, \cite{kac1}.\\
Let $\g$ be a Lie superalgebra; a formal distribution with coefficients in $\g$, or equivalently a $\g-$valued formal distribution, in the indeterminate $z$ is an expression of the following form:
\begin{align*}
a(z)=\sum_{n \in \Z}a_{n}z^{-n-1},  
\end{align*}
with $a_{n} \in \g$ for every $n \in \Z$. We denote the vector space of formal distributions with coefficients in $\g$ in the indeterminate $z$ by $\g[[z,z^{-1}]]$. We denote by $\Res(a(z))=a_{0}$ the coefficient of $z^{-1}$ of $a(z)$. The vector space $\g[[z,z^{-1}]]$ has a natural structure of $\C[\partial_{z}]-$module. We define for all $a(z) \in \g[[z,z^{-1}]]$ its derivative:
\begin{align*}
\partial_{z}a(z)=\sum_{n \in \Z}(-n-1)a_{n}z^{-n-2}.
\end{align*} 
A formal distribution with coefficients in $\g$ in the indeterminates $z$ and $w$ is an expression of the following form:
\begin{align*}
a(z,w)=\sum_{m,n \in \Z}a_{m,n}z^{-m-1}w^{-n-1},  
\end{align*}
with $a_{m,n} \in \g$ for every $m,n \in \Z$. We denote the vector space of formal distributions with coefficients in $\g$ in the indeterminates $z$ and $w$ by $\g[[z,z^{-1},w,w^{-1}]]$.
Given two formal distributions $a(z) \in \g[[z,z^{-1}]]$ and $b(w) \in \g[[w,w^{-1}]]$, we define the commutator $[a(z),b(w)]$:
\begin{align*}
[a(z),b(w)]=\bigg[\sum_{n \in \Z}a_{n}z^{-n-1} ,\sum_{m \in \Z}b_{m}w^{-m-1}  \bigg]=\sum_{m,n \in \Z}[a_{n},b_{m}] z^{-n-1}  w^{-m-1} .
\end{align*}
\begin{defi}
Two formal distributions $a(z),b(z) \in \g[[z,z^{-1}]]$ are called local if:
\begin{align*}
(z-w)^{N}[a(z),b(w)]=0 \ for \ some \ N \gg 0.
\end{align*}
\end{defi}
We call $\delta-$function the following formal distribution in the indeterminates $z$ and $w$:
\begin{align*}
\delta(z-w)=z^{-1}\sum_{n \in \Z} \left( \frac{w}{z} \right)^{n}.
\end{align*}
See Corollary $2.2$ in \cite{kac1vertex} for the following equivalent condition of locality.
\begin{prop}
Two formal distributions $a(z),b(z) \in \g[[z,z^{-1}]]$ are local if and only if $[a(z),b(w)]$ can be expressed as a finite sum of the form:
\begin{align*}
[a(z),b(w)]=\sum_{j}(a(w)_{(j)}b(w))\frac{\partial_{w}^{j}}{j!} \delta(z-w),
\end{align*}
where the coefficients $(a(w)_{(j)}b(w)):=\Res_{z}(z-w)^{j}[a(z),b(w)]$ are formal distributions in the indeterminate $w$.
\end{prop}
\begin{defi}[Formal Distribution Superalgebra]
Let $\g$ be a Lie superalgebra and $\mathcal{F}$ a family of mutually local $\g-$valued formal distributions in the indeterminate $z$. The pair $(\g,\mathcal{F})$ is called a \textit{formal distribution superalgebra} if the coefficients of all formal distributions in $\mathcal{F}$ span $\g$.
\end{defi}
We define the $\lambda-$bracket between two formal distributions $a(z),b(z) \in \g[[z,z^{-1}]] $ as the generating series of the $(a(z)_{(j)}b(z))$'s:
\begin{align}
\label{bracketformal}
[a(z)_{\lambda}b(z)]=\sum_{j \geq 0} \frac{\lambda^{j}}{j!}(a(z)_{(j)}b(z)).
\end{align}
\begin{defi}[Conformal superalgebra]
A \textit{conformal superalgebra} $R$ is a left $\Z_{2}-$graded $\C[\partial]-$module endowed with a $\C-$linear map, called $\lambda-$bracket, $R \otimes R \rightarrow \C[\lambda]\otimes R$, $a \otimes b \mapsto [a_{\lambda}b]$, that satisfies the following properties for all $a,b,c \in R$:
\begin{align*}
(i)&\,\, conformal \,\, sesquilinearity: &&[\partial a_{\lambda}b]=-\lambda [a_{\lambda}b], \quad  [a_{\lambda} \partial b]=(\lambda+\partial)[a_{\lambda}b]; \\
(ii)&\,\, skew-symmetry:             &&[a_{\lambda}b]=-(-1)^{p(a)p(b)}[b_{-\lambda-\partial}a] ;    \\
(iii)&\,\, Jacobi \,\, identity:           &&[a_{\lambda}[b_{\mu}c]]=[[a_{\lambda}b]_{\lambda+\mu}c]+(-1)^{p(a)p(b)}[b_{\mu}[a_{\lambda}c]];
\end{align*}
where $p(a)$ denotes the parity of the element $a \in R$ and $p(\partial a)=p(a)$ for all $a \in R$.
\end{defi}
We call $n-$products the coefficients $(a_{(n)}b)$ that appear in $[a_{\lambda}b]=\sum_{n\geq 0} \frac{\lambda^{n}}{n!}(a_{(n)}b)$  and give an equivalent definition of conformal superalgebra.
\begin{defi}[Conformal superalgebra]
\label{definizionesuperalgebraconforme}
A \textit{conformal superalgebra} $R$ is a left $\Z_{2}-$graded $\C[\partial]-$module endowed with a $\C-$bilinear product $(a_{(n)}b): R\otimes R\rightarrow R$, defined for every $n \geq 0$, that satisfies the following properties for all $a,b,c \in R$, $m,n \geq 0$:
\begin{itemize}
  \item[(i)] $p(\partial a)=p( a)$;
  \item[(ii)] $(a_{(n)}b)=0, \,\, for \,\, n \gg 0$;
	\item[(iii)] $({\partial a}_{(0)}b)=0$ and $({\partial a}_{(n+1)}b)=-(n+1) (a_{(n)}b)$;
	\item[(iv)] $(a_{(n)}b)=-(-1)^{p(a)p(b)}\sum_{j \geq 0}(-1)^{j+n} \frac{\partial^{j}}{j!}(b_{(n+j)}a)$;
	\item[(v)] $(a_{(m)}(b_{(n)}c))=\sum^{m}_{j=0}\binom{m}{j}((a_{(j)}b)_{(m+n-j)}c)+(-1)^{p(a)p(b)}(b_{(n)}(a_{(m)}c))$.
\end{itemize}
\end{defi}
Using (iii) and (iv) in Definition \ref{definizionesuperalgebraconforme} it is easy to show that for all $a,b \in R$, $n \geq 0$:
\begin{equation*}
(a_{(n)}\partial b)=\partial (a_{(n)} b)+n (a_{(n-1)}b).
\end{equation*}
Due to this relation and (iii) in Definition \ref{definizionesuperalgebraconforme}, the map $\partial: R\rightarrow R$, $a \mapsto \partial a$ is a derivation with respect to the $n-$products.
\begin{rem}
Let $(\g ,\mathcal{F})$ be a formal distribution superalgebra, endowed with $\lambda-$bracket (\ref{bracketformal}). The elements of $\mathcal{F}$ satisfy sesquilinearity, skew-symmetry and Jacobi identity with $\partial=\partial_{z}$; for a proof see Proposition 2.3 in \cite{kac1vertex}.
\end{rem}
We say that a conformal superalgebra $R$ is \textit{finite} if it is finitely generated as a $\C[\partial]-$module. 
An \textit{ideal} $I$ of $R$ is a $\C[\partial]-$submodule of $R$ such that $(a_{(n)}b)\in I$ for every $a \in R$, $b \in I$, $n \geq 0$. A conformal superalgebra $R$ is \textit{simple} if it has no non-trivial ideals and the $\lambda-$bracket is not identically zero. We denote by $R'$ the \textit{derived subalgebra} of $R$, i.e. the $\C-$span of all $n-$products.
\begin{defi}
A module $M$ over a conformal superalgebra $R$ is a left $\Z_{2}-$graded $\C[\partial]-$module endowed with $\C-$linear maps $R \rightarrow \End_{\C} M$, $a\mapsto a_{(n)}$, defined for every $n \geq 0$, that satisfy the following properties for all $a,b \in R$, $v \in M$, $m,n \geq 0$:
\begin{enumerate}
  \item[(i)] $a_{(n)}v=0 \,\, for \,\, n \gg 0$;
	\item[(ii)] $(\partial a)_{(n)}v=[\partial,a_{(n)}]v=-n a_{(n-1)}v $;
	\item[(iii)] $[a_{(m)},b_{(n)}]v=\sum_{j=0}^{m} \binom{m}{j}(a_{(j)}b)_{(m+n-j)}v$.
\end{enumerate}
\end{defi}
For an $R-$module $M$, we define for all $a \in R$ and $v\in M$:
\begin{align*}
a_{\lambda}v=\sum_{n \geq 0}\frac{\lambda^{n}}{n!}a_{(n)}v.
\end{align*}
A module $M$ is called \textit{finite} if it is a finitely generated $\C[\partial]-$module.\\
We can construct a conformal superalgebra starting from a formal distribution superalgebra $(\g,\mathcal{F})$. Let $\mathcal{\overline{F}}$  be the closure of $\mathcal{F}$ under all the $n-$products, $\partial_{z}$ and linear combinations. By  Dong's Lemma, $\mathcal{\overline{F}}$ is still a family of mutually local distributions (see \cite{kac1vertex}). It turns out that $\mathcal{\overline{F}}$ is a conformal superalgebra. We will refer to it as the conformal superalgebra associated with $(\g,\mathcal{F})$.\\
Let us recall the construction of the annihilation superalgebra associated with a conformal superalgebra $R$.
Let $\widetilde{R}=R[y,y^{-1}]$, set $p(y)=0$ and $\widetilde{\partial}=\partial+\partial_{y}$. We define the following $n-$products on $\widetilde{R}$, for all $a,b \in R$, $f,g \in \C[y,y^{-1}]$, $n\geq 0$:
\begin{align*}
(af_{(n)}bg)=\sum_{j \in \Z_{+}}(a_{(n+j)}b) \Big(\frac{\partial_{y}^{j}}{j!}f \Big)g .
\end{align*}
In particular if $f=y^{m}$ and $g=y^{k}$ we have for all $n \geq 0$:
\begin{align*}
({ay^{m}}_{ (n)}by^{k})=\sum_{j \in \Z_{+}}\binom{m}{j}(a_{(n+j)}b)y^{m+k-j}.
\end{align*}
We observe that $\widetilde{\partial}\widetilde{R}$ is a two sided ideal of $\widetilde{R}$ with respect to the $0-$product. The quotient $\Lie R:=\widetilde{R}/ \widetilde{\partial}\widetilde{R}$ has a structure of Lie superalgebra with the bracket induced by the $0-$product, i.e. for all $a,b \in R$, $f,g \in \C[y,y^{-1}]$: 
\begin{align}
\label{bracketannihilation}
[af,bg]=\sum_{j \in \Z_{+}}(a_{( j)}b)\Big(\frac{\partial_{y}^{j}}{j!}f \Big)g .
\end{align}
\begin{defi}
The annihilation superalgebra $\mathcal{A}(R)$ of a conformal superalgebra $R$ is the subalgebra of $\Lie R$ spanned by all elements $ay^{n}$ with $n\geq 0$ and $a\in R$. \\
The extended annihilation superalgebra $\mathcal{A}(R)^{e}$ of a conformal superalgebra $R$ is the Lie superalgebra $\C \partial \ltimes \mathcal{A}(R)$. The semidirect sum $\C \partial \ltimes \mathcal{A}(R)$ is the vector space $\C \partial \oplus \mathcal{A}(R)$ endowed with the structure of Lie superalgebra determined by the bracket:
\begin{align*}
[\partial,ay^{m}]=-\partial_{y}(ay^{m})=-m a y^{m-1},
\end{align*}
for all $a \in R$ and the fact that $\C \partial$, $\mathcal{A}(R)$ are Lie subalgebras.
\end{defi}
For all $a \in R$ we consider the following formal power series in $\mathcal{A}(R)[[\lambda]]$:
\begin{align}
\label{powerseries}
a_{\lambda}=\sum_{n \geq 0}\frac{\lambda^{n}}{n!}ay^{n}.
\end{align}
For all $a,b \in R$, we have: $[a_{\lambda},b_{\mu}]=[a_{\lambda}b]_{\lambda+\mu}$ and $(\partial a)_{\lambda}=-\lambda a_{\lambda}$ (for a proof see \cite{cantacasellikac}).  
\begin{prop}[\cite{chengkac}]
\label{propcorrispmoduli}
Let $R$ be a conformal superalgebra.
	If $M$ is an $R$-module then $M$ has a natural structure of $\mathcal A(R)^e$-module, where the action of $ay^n$ on $M$ is uniquely determined by $a_\lambda v=\sum_{n \geq 0}\frac{\lambda^{n}}{n!}ay^{n}.v$ for all $v\in M$. Viceversa if $M$ is a $\mathcal A(R)^e$-module such that for all $a\in R$, $v\in M$ we have $ay^n.v=0$ for $n\gg0$, then $M$ is also an $R$-module by letting $a_\lambda v=\sum_{n}\frac{\lambda^{n}}{n!}ay^{n}.v$.
\end{prop}
Proposition \ref{propcorrispmoduli} reduces the study of modules over a conformal superalgebra $R$ to the study of a class of modules over its (extended) annihilation superalgebra.
The following proposition states that, under certain hypotheses, it is sufficient to consider the annihilation superalgebra. We recall that, given a $\Z-$graded Lie superalgebra $\g=\oplus_{i \in \Z}\g_{i}$, we say that $\g$ has finite depth $d\geq0$ if $\g_{-d}\neq 0$ and $\g_{i}=0$ for all $i<-d$.
\begin{prop}[\cite{kac1},\cite{chenglam}]
\label{keythmannihi}
Let $\g$ be the annihilation superalgebra of a conformal superalgebra $R$. Assume that $\g$ satisfies the following conditions:
\begin{description}
	\item[L1] $\g$ is $\Z-$graded with finite depth $d$;
	\item[L2] There exists an element whose centralizer in $\g$ is contained in $\g_{0}$;
	\item[L3] There exists an element $\Theta \in \g_{-d}$ such that $\g_{i-d}=[\Theta,\g_{i}]$, for all $i\geq 0$.
\end{description}
 Finite modules over $R$ are the same as modules $V$ over $\g$, called \textit{finite conformal}, that satisfy the following properties:
\begin{enumerate}
	\item for every $v \in V$, there exists $j_{0} \in \Z$, $j_{0}\geq -d$, such that $\g_{j}.v=0$ when $j\geq j_{0}$;
	\item $V$ is finitely generated as a $\C[\Theta]-$module.
\end{enumerate}
\end{prop}
\begin{rem}
\label{gradingelement}
We point out that condition \textbf{L2} is automatically satisfied when $\g$ contains a \textit{grading element}, i.e. an element $t \in \g$ such that $[t,b]=\degr (b) b$ for all $b \in \g$.
\end{rem}
Let $\g=\oplus_{i \in \Z} \g_{i}$ be a $\Z-$graded Lie superalgebra. We will use the notation $\g_{>0}=\oplus_{i>0}\g_{i}$, $\g_{<0}=\oplus_{i<0}\g_{i}$ and $\g_{\geq 0}=\oplus_{i\geq 0}\g_{i}$. We denote by $U(\g) $ the universal enveloping algebra of $\g$.
\begin{defi}
Let $F$ be a $\g_{\geq 0}-$module. The \textit{generalized Verma module} associated with $F$ is the $\g-$module $\Ind (F)$ defined by:
\begin{equation*}
\Ind (F):= \Ind ^{\g}_{\g_{\geq 0}} (F)=U(\g) \otimes _{U(\g_{\geq 0})} F.
\end{equation*}
\end{defi}
If $F$ is a finite$-$dimensional irreducible $\g_{\geq 0}-$module we will say that $\Ind(F)$ is a \textit{finite Verma module}. We will identify $\Ind (F)$ with $U(\g_{<0}) \otimes  F$ as vector spaces via the Poincar\'e$-$Birkhoff$-$Witt Theorem. The $\Z-$grading of $\g$ induces a $\Z-$grading on $U(\g_{<0})$ and $\Ind (F)$. We will invert the sign of the degree, so that we have a $\Z_{\geq 0}-$grading on $U(\g_{<0})$ and $\Ind (F)$. We will say that an element $v \in U(\g_{<0})_{k}$ is homogeneous of degree $k$. Analogously an element $m \in U(\g_{<0})_{k}  \otimes  F$ is homogeneous of degree $k$. For a proof of the following proposition see \cite{K4}.
\begin{prop}
Let $\g=\oplus_{i \in \Z} \g_{i}$ be a $\Z-$graded Lie superalgebra. If $F$ is an irreducible finite$-$ dimensional $\mathfrak{g}_{\geq 0}-$module, then $\Ind (F)$ has a unique maximal submodule. We denote by $\I (F)$ the quotient of $\Ind (F)$ by the unique maximal submodule.
\end{prop}
\begin{defi}
Given a $\g-$module $V$, we call \textit{singular vectors} the elements of:
\begin{align*}
\Sing (V) =\left\{v \in V \,\, | \, \, \g_{>0}.v=0\right\}.
\end{align*}
Homogeneous components of singular vectors are still singular vectors so we often assume that singular vectors are homogeneous without loss of generality.
In the case $V=\Ind (F)$, for a $\g_{\geq 0}-$module $F$, we will call \textit{trivial singular vectors} the elements of $\Sing (V) $ of degree 0 and \textit{nontrivial singular vectors} the nonzero elements of $\Sing (V) $ of positive degree.
\end{defi}
\begin{thm}[\cite{kacrudakov},\cite{chenglam}]
\label{keythmsingular}
Let $\g$ be a Lie superalgebra that satisfies \textbf{L1}, \textbf{L2}, \textbf{L3}, then:
\begin{enumerate}
\item[(i)] if $F$ is an irreducible finite$-$dimensional $\mathfrak{g}_{\geq 0}-$module, then $\mathfrak{g}_{> 0}$ acts trivially on it;
	\item[(ii)] the map $F \mapsto \I (F)$ is a bijective map between irreducible finite$-$dimensional $\mathfrak{g}_{ 0}-$modules and irreducible finite conformal $\mathfrak{g}-$modules;
	\item[(iii)] the $\mathfrak{g}-$module $\Ind (F)$ is irreducible if and only if the $\mathfrak{g}_{0}-$module $F$ is irreducible and $\Ind (F)$ has no nontrivial singular vectors.
	\end{enumerate}
	\end{thm}
We recall the notion of duality for conformal modules (see for further details \cite{bklr}, \cite{cantacasellikac}). Let $R$ be a conformal superalgebra and $M$ a conformal module over $R$.
\begin{defi}
The conformal dual $M^{*}$ of $M$ is defined by:
\begin{align*}
M^{*}=\left\{f_{\lambda}:M\rightarrow \C[\lambda] \,\, | \,\, f_{\lambda}(\partial m)=\lambda f_{\lambda}(m), \,\, \forall m \in M\right\}.
\end{align*}
The structure of $\C[\partial]-$module is given by $(\partial f)_{\lambda}(m)=-\lambda  f_{\lambda}(m)$, for all $f \in M^{*}$, $ m \in M$. The $\lambda-$action of $R$ is given, for all $a \in R$, $m \in M$, $f \in M^{*}$, by:
\begin{align*}
(a_{\lambda}f)_{\mu}(m)=-(-1)^{p(a)p(f)}f_{\mu-\lambda}(a_{\lambda}m).
\end{align*}
\end{defi}
\begin{defi}
Let $T:M\rightarrow N$ be a morphism of $R-$modules, i.e. a linear map such that for all $a\in R$ and $m\in M$:
\begin{itemize}
	\item [i:] $T(\partial m)=\partial T(m)$,
	\item [ii:] $T(a_{\lambda} m)=a_{\lambda} T(m)$.
\end{itemize}
The dual morphism $T^{*}:N^{*} \rightarrow M^{*}$ is defined, for all $f \in N^{*}$ and $m \in M$, by:
\begin{align*}
\left[T^{*}(f)\right]_{\lambda}(m)=-f_{\lambda}\left(T(m)\right).
\end{align*}
\end{defi}
\begin{thm}[\cite{bklr}, Proposition 2.6]
\label{dualeconforme}
Let $R$ be a conformal superalgebra and $M,N$ $R-$modules. Let $T:M\longrightarrow N$ be a homomorphism of $R-$modules such that $N/ \Ima T$ is a finitely generated torsion$-$free $\C[\partial]-$module. Then the standard map $\Psi:N^{*} / \Ker T^{*}\longrightarrow(M/ \Ker T)^{*}$, given by $[\Psi(\overline{f})]_{\lambda}(\overline{m})=f_{\lambda}(T(m))$
 (where by the bar we denote the corresponding class in the quotient), is an isomorphism of $R-$modules.
\end{thm}
We denote by $\textbf{F}$ the functor that maps a conformal module $M$ over a conformal superalgebra $R$ to its conformal dual $M^{*}$ and maps a morphism between conformal modules $T:M \rightarrow N$ to its dual $T^{*}:N^{*} \rightarrow M^{*}$.
\begin{prop}
\label{esattezzafuntoreduale}
The functor $\textbf{F}$ is exact if we consider only morphisms $T:M \rightarrow N$, where $N/ \Ima T$ is a finitely generated torsion free $\C[\partial]-$module.
\end{prop}
\begin{proof}
Let us consider an exact short sequence of conformal modules:
\begin{align*}
0 \rightarrow M \xrightarrow{d_{1}}N \xrightarrow{d_{2}} P \rightarrow 0.
\end{align*}
Therefore we know that $d_{2} \circ d_{1}=0$, $d_{1}$ is injective, $d_{2}$ is surjective and $\Ker d_{2}=\Ima d_{1}$. We consider the dual of this sequence:
\begin{align*}
0 \rightarrow P^{*} \xrightarrow{d^{*}_{2}} N^{*} \xrightarrow{d^{*}_{1}} M^{*} \rightarrow 0.
\end{align*}
By Theorem \ref{dualeconforme} and Remark 3.11 in \cite{cantacasellikac}, we know that $d^{*}_{1}$ is surjective and $d^{*}_{2}$ is injective. We have to show that $\Ker d^{*}_{1} =\Ima d^{*}_{2}$.
Let us first show that $\Ker d^{*}_{1} \supset \Ima d^{*}_{2}$. Let $\beta\in \Ima d^{*}_{2} \subset N^{*}$. We have $\beta=d^{*}_{2}(\alpha)$ for some $\alpha \in P^{*}$. For all $m \in M$:
\begin{align*}
\left[d^{*}_{1}(\beta)\right]_{\lambda}(m)=-\beta_{\lambda}(d_{1}(m))=\alpha_{\lambda}(d_{2}(d_{1}(m)))=0.
\end{align*}
Let us now show that $\Ker d^{*}_{1} \subset \Ima d^{*}_{2}$. Let $\beta \in \Ker d^{*}_{1} \subset N^{*}$. For all $m \in M$:
\begin{align*}
0=\left[d^{*}_{1}(\beta)\right]_{\lambda}(m)=-\beta_{\lambda}(d_{1}(m)).
\end{align*}
Since $\Ker d_{2}=\Ima d_{1}$, this condition tells that $\beta$ vanishes on $\Ker d_{2}$. We also know that for every $p \in P$, $p=d_{2}(n_{p})$, for some $n_{p} \in N$. We define $\alpha \in P^{*}$ as follows, for all $p \in P$:
\begin{align*}
\alpha_{\lambda}(p)=\alpha_{\lambda}(d_{2}(n_{p}))=
-\beta_{\lambda}(n_{p}).
\end{align*}
Let us show that $\alpha$ actually lies in $P^{*}$.
For every $p \in P$:
\begin{align*}
\alpha_{\lambda}(\partial p)=\alpha_{\lambda}(\partial d_{2}(n_{p}))=\alpha_{\lambda}( d_{2}(\partial n_{p}))=
-\beta_{\lambda}(\partial n_{p}).
\end{align*}
Since $\beta \in N^{*}$, we know that $-\beta_{\lambda}(\partial n_{p})=-\lambda \beta_{\lambda}(n_{p})$. Therefore $\alpha_{\lambda}(\partial p)=\lambda \alpha_{\lambda}( p)$.\\
We have for all $n \in N$ that:
\begin{align*}
\left[d^{*}_{2}(\alpha)\right]_{\lambda}(n)=-\alpha_{\lambda}(d_{2}(n))=\beta_{\lambda}(n).
\end{align*}
\end{proof}
\section{The conformal superalgebra $K'_{4}$}
\label{section$K'_{4}$}
In this section we recall some notions and properties about the conformal superalgebra $K'_{4}$ (for further details see \cite{K4},\cite{kac1},\cite{fattorikac}).
We first recall the notion of the contact Lie superalgebra. Let $\inlinewedge(N)$ be the Grassmann superalgebra in the $N$ odd indeterminates $\xi_{1},...,\xi_{N}$. Let $t$ be an even indeterminate and $\inlinewedge (1,N)=\C[t,t^{-1}] \otimes \inlinewedge(N)$. We consider the Lie superalgebra of derivations of $\inlinewedge (1,N)$:
\begin{equation*}
W(1,N)=\bigg\{ D=a \partial_{t}+\sum ^{N}_{i=1} a_{i} \partial_{i} \,\, | \,\, a,a_{i} \in \displaywedge (1,N)\bigg\},
\end{equation*}
where $\partial_{t}=\frac{\partial}{\partial{t}}$ and $\partial_{i} =\frac{\partial}{\partial{\xi_{i}}}$ for every $i \in \left\{1,...,N \right\}$.\\
Let us consider the contact form $\omega = dt-\sum_{i=1}^{N}\xi_{i} d\xi_{i} $. The contact Lie superalgebra $K(1,N)$ is defined by:
\begin{equation*}
K(1,N)=\left\{D \in W(1,N) \,\, | \,\, D\omega=f_{D}\omega \,\, for \,\, some \,\, f_{D} \in \displaywedge (1,N)\right\}.
\end{equation*}
We denote by $K'(1,N)$ the derived subalgebra $[K(1,N),K(1,N)]$ of $K(1,N)$.
Analogously, let $\inlinewedge  (1,N)_{+}=\C[t] \otimes \inlinewedge (N)$. We define the Lie superalgebra $W(1,N)_{+}$ (resp. $K(1,N)_{+}$) similarly to $W(1,N)$ (resp. $K(1,N)$) using $\inlinewedge  (1,N)_{+}$ instead of $\inlinewedge  (1,N)$.\\
One can define on $\inlinewedge (1,N)$ a Lie superalgebra structure as follows: for all $f,g \in \inlinewedge(1,N)$ we let
\begin{equation}
\label{bracketlie}
[f,g]=\Big(2f-\sum_{i=1}^{N} \xi_{i}  \partial_{i} f \Big)(\partial_{t}{g})-(\partial_{t}{f})\Big(2g-\sum_{i=1}^{N} \xi_{i} \partial_{i} g\Big)+(-1)^{p(f)}\Big(\sum_{i=1}^{N} \partial_{i}  f \partial_{i}  g \Big).
\end{equation}
We recall that $K(1,N) \cong \inlinewedge(1,N)$ as Lie superalgebras via the following map (see \cite{{chengkac2}}):
\begin{gather*}
\displaywedge(1,N) \longrightarrow  K(1,N) \\
f \longmapsto 2f \partial_{t}+(-1)^{p(f)} \sum_{i=1}^{N} (\xi_{i} \partial_{t} f+ \partial_{i}f )(\xi_{i} \partial_{t} + \partial_{i}).
\end{gather*}
From now on we will always identify elements of $K(1,N) $ with elements of $\inlinewedge(1,N)$ and we will omit the symbol $\wedge$ between the $\xi_{i}$'s. 
We adopt the following notation: we denote by $\mathcal I$ the set of finite sequences of elements in $\{1,\ldots,N\}$; we will write $I=i_1\cdots i_r$ instead of $I=(i_1,\ldots,i_r)$. Given $I=i_1\cdots i_r$ and $J=j_1\cdots j_s$, we will denote $i_1\cdots i_rj_1\cdots j_s$ by $IJ$; if $I=i_1 \cdots i_r\in \mathcal I$  we let $\xi_{I}=\xi_{i_1}\cdots \xi_{i_r}$ and $|\xi_{I}|=|I|=r$. We denote by $\mathcal{I}_{\neq}$ the subset of $\mathcal{I}$ of sequences with distinct entries. We consider on $K(1,N)$ the standard grading, i.e. for every $m \in \Z$ and $I \in \mathcal I$, $\degr(t^{m} \xi_{I})=2m+|I|-2$.\\
Now we want to recall the definition of the conformal superalgebra $K_{N}$.
In order to do this, we construct a formal distribution superalgebra using the following family of formal distributions:
\begin{equation*}
\mathcal{F}=\bigg\{A(z):= \sum_{m \in \Z}(At^{m})z^{-m-1}=A \delta(t-z), \,\, \forall    A \in \displaywedge(N)\bigg\}. 
\end{equation*}
Note that the set of all the coefficients of formal distributions in $\mathcal{F}$ spans $\inlinewedge(1,N)$ and the distributions are mutually local (for a proof see \cite[Proposition 3.1]{K4}). The conformal superalgebra associated with $(K(1,N),\mathcal{F})$ is identified with $K_{N}:=\C[\partial] \otimes \inlinewedge(N)$. For all $I,J \in \mathcal{I}$ the $\lambda-$bracket is given by
\begin{equation}
\label{lambdabracket}
[{\xi_I}_{\lambda}\xi_J]=(|I|-2)\partial(\xi_{IJ})+(-1)^{|I|}\sum^{N}_{i=1}(\partial_{i}\xi_I)(\partial_{i}\xi_J)+\lambda(|I|+|J|-4)\xi_{IJ},
\end{equation}
for a proof see \cite[Proposition 3.1]{K4}. $K_{N}$ is simple except for the case $N=4$. If $N=4$, $K_{4}=K'_{4}\oplus \C \xi_{1234}$, where $K'_{4}$ is the derived subalgebra, i.e. the $\C-$span of the $n-$products. $K'_{4}$ is a simple conformal superalgebra (see \cite{fattorikac}).
\begin{prop}[\cite{K4}]
\label{anniliK4}
The annihilation superalgebra $\mathcal{A}(K'_{4})$, associated with $K'_{4}$, is a central extension of $K(1,4)_{+}$ by a one$-$dimensional center $\C C$:
\begin{align*}
\mathcal{A}(K'_{4})=K(1,4)_{+} \oplus \C C.
\end{align*}
The extension is given by the $2-$cocycle $\psi \in  Z^{2}(K(1,4)_{+}, \C)$ which computed on basis elements returns non$-$zero values in the following cases only (up to skew-symmetry of $\psi$):
\begin{align*}
&\psi(1  ,\xi_{1234} )=-2,\\
&\psi(\xi_{i}  ,\partial_{i} \xi_{1234}  )=-1.
\end{align*} 
\end{prop}

We denote with $\mathfrak{g}:=\mathcal{A}(K'_{4})=K(1,4)_{+} \oplus \C C$. We recall from \cite{K4} the description of $\g$. The grading on $\mathfrak{g}$ is the standard grading of $K(1,4)_{+} $ and $C$ has degree $0$. We have:
\begin{align*}
&\g_{-2} =\left< 1 \right>, \\
&\g_{-1} =\left< \xi_{1},\xi_{2},\xi_{3},\xi_{4} \right> ,\\
&\g_{0} =\left< \left\{C,t, \xi_{ij}: \,\, 1 \leq i<j \leq 4 \right\}\right>.
\end{align*}
The annihilation superalgebra $\g$ satisfies \textbf{L1}, \textbf{L2}, \textbf{L3}: \textbf{L1} is straightforward; \textbf{L2} follows by Remark \ref{gradingelement} since $t$ is a grading element for $\g$; \textbf{L3} follows from the choice $\Theta:=-1/2 \in \g_{-2}$.
We recall that $\g_0=\left< \left\{ C,t, \xi_{ij} :\, 1 \leq i<j \leq 4 \right\} \right>\cong \mathfrak{so}(4)\oplus \C { t }\oplus \C C$, where $\mathfrak{so}(4)$ is the Lie algebra of $4\times 4$ skew$-$symmetric matrices. In the above isomorphism the element $\xi_{ij}$ corresponds to the skew-symmetric matrix $-E_{i,j}+E_{j,i} \in \mathfrak{so}(4)$. 
We consider the following basis of a Cartan subalgebra $\mathfrak{h}$:
\begin{equation}\label{hx} h_{x}:=-i\xi_{12}+i\xi_{34}, \,\,h_{y}:=-i\xi_{12}-i\xi_{34}.
\end{equation}
Let $\alpha_x,\alpha_y\in \mathfrak{h}^*$ be such that $\alpha_x(h_x)=\alpha_y(h_y)=2$ and $\alpha_x(h_y)=\alpha_y(h_x)=0$. 
The set of roots is $\Delta=\{\alpha_x,-\alpha_x,\alpha_y,-\alpha_y\}$ and we have the following root decomposition:
\begin{align*}
\mathfrak{so}(4)=\mathfrak{h} \oplus \left(\oplus_{\alpha \in \Delta} \g_{\alpha}\right) \,\, \text{with} \,\, \g_{\alpha_x}=\C e_{x},\,\g_{-\alpha_x}=\C f_{x},\,\g_{\alpha_y}=\C e_{y},\,\g_{-\alpha_y}=\C f_{y}
\end{align*}
where
\begin{align*}
&e_{x}=\frac{1}{2}(-\xi_{13}-\xi_{24}-i\xi_{14}+i\xi_{23}),\,\,\, &&e_{y}=\frac{1}{2}(-\xi_{13}+\xi_{24}+i\xi_{14}+i\xi_{23}),\\
&f_{x}=\frac{1}{2}(\xi_{13}+\xi_{24}-i\xi_{14}+i\xi_{23}),\,\,\, &&f_{y}=\frac{1}{2}(\xi_{13}-\xi_{24}+i\xi_{14}+i\xi_{23}).
\end{align*}
We will use the following notation:
\begin{align*}
e_1=e_x+e_y=-\xi_{13}+i\xi_{23},\,\,\,
e_2=e_x-e_y=-\xi_{24}-i\xi_{14}.
\end{align*}
The set $\left\{e_1,e_2\right\}$ is a basis of the nilpotent subalgebra $\g_{\alpha_x} \oplus \g_{\alpha_y}$. We denote by $\g_{0}^{ss}$ the semisimple part of $\g_{0}$.
\begin{rem}
	\label{notazioinecasellig0}
	The sets $\{e_x,f_x,h_x\}$ and $\{e_y,f_y,h_y\}$ span two copies of $\mathfrak{sl}_2$ and we think of $\g^{ss}_0$ in the standard way as a Lie algebra of derivations.
	We have that:
	\begin{align*}
	\g^{ss}_{0}= \langle e_{x},f_{x},h_{x}\rangle \oplus \langle e_{y},f_{y},h_{y} \rangle \cong \langle  x_{1} \partial_{x_{2}}, x_{2} \partial_{x_{1}},x_{1} \partial_{x_{1}}-x_{2} \partial_{x_{2}} \rangle \oplus \langle  y_{1} \partial_{y_{2}}, y_{2} \partial_{y_{1}},y_{1} \partial_{y_{1}}-y_{2} \partial_{y_{2}} \rangle .
	\end{align*}
\end{rem}
By direct computations, we obtain the following results.
\begin{lem}[\cite{K4}] \label{g1}
The subalgebra $\g_{>0}$ is generated by $\g_{1}$, i.e. $\g_{i}=\g_{1}^{i}$ for all $i\geq 2$ and as $\g_{0}-$modules:
	\begin{align*}
	\g_{1} \cong \langle t\xi_{i} :\, 1\leq i\leq 4\rangle \oplus \langle \xi_{I}:\, |I|=3\rangle.
	\end{align*}
	The $\g_{0}-$modules $\langle t\xi_{i}:\, 1\leq i\leq 4\rangle$ and $\langle \xi_{I}:\, |I|=3\rangle$ are irreducible and the corresponding lowest weight vectors are $t (\xi_{1}+i \xi_{2})$ and $(\xi_{1}+i\xi_{2})\xi_{3}\xi_{4}$.
	\end{lem}
	By Lemma \ref{g1} to check whether a vector $\vec{m}$ in a $\g$-module is a highest weight singular vector it is sufficient to show that it is annihilated by $e_1$, $e_2$, $t(\xi_1+i\xi_2)$ and $(\xi_1+i\xi_2)\xi_3\xi_4$. 
	\begin{lem}[\cite{K4}]
	\label{appoggiog-1}
	As $\g_{0}^{ss}-$modules:
	\begin{align*}
	\g_{-1} \cong \langle x_{1}y_1,x_1y_2,x_{2}y_1,x_2y_2\rangle.
	\end{align*}
	The isomorphism is given by:
	\begin{align*}
	\xi_{2}+i\xi_{1} \leftrightarrow x_{1}y_{1}, \,\,\, \xi_{2}-i\xi_{1} \leftrightarrow x_{2}y_{2}, \,\,\,  -\xi_{4}+i\xi_{3} \leftrightarrow x_{1}y_{2},  \,\,\, \xi_{4}+i\xi_{3} \leftrightarrow x_{2}y_{1}.
	\end{align*}
\end{lem}
From now on it is always assumed that $F$ is a finite$-$dimensional irreducible $\mathfrak{g}_{\geq 0}-$module.
\begin{rem}
Since $\Ind (F) \cong U(\mathfrak{g}_{<0}) \otimes F$, it follows that $\Ind (F) \cong \C[\Theta] \otimes \inlinewedge(4) \otimes F$. Indeed, let us denote by $\eta_{i}$ the image in $U(\g)$ of $\xi_{i} \in  \inlinewedge(4) $, for all $i \in \left\{1,2,3,4\right\}$.
 In $U(\g)$ we have that $\eta_{i}^{2}=\Theta$, for all $i \in \left\{1,2,3,4\right\}$: since $[\xi_{i},\xi_{i}]=-1$ in $\g$, we have $\eta_{i}\eta_{i}=-\eta_{i}\eta_{i}-1$ in $U(\g)$.
\end{rem}
 Given $I=i_{1}\cdots i_{k} \in \mathcal I_{\neq}$, we will use the notation  $\eta_{I}$ to denote the element  $\eta_{i_{1}}  \cdots \eta_{i_{k}} \in  U(\g_{<0})$ and we will denote  $|\eta_{I}|=|I|=k$.
\begin{rem}
 Since $C$ is central, by Schur's lemma, $C$ acts as a scalar on $F$.
\end{rem}
We will write the weights $\mu=(m,n,\mu_{t},\mu_{C})$ of weight vectors of $\g_{0}-$modules with respect to the action of the vectors $h_{x},h_{y},{ t } $ and $ C$.
Motivated by Lemma \ref{appoggiog-1}, we will use the notation 
\begin{align}
\label{notazionewcaselli}
w_{11}=\eta_{2}+i\eta_{1}, \,\, w_{22}=\eta_{2}-i\eta_{1}, \,\, w_{12}= -\eta_{4}+i\eta_{3}, \,\, w_{21}=\eta_{4}+i\eta_{3}.
\end{align}
We point out that
\begin{align}
\label{bracketwcaselli}
[w_{11},w_{22}]=4 \Theta, \quad [w_{12},w_{21}]=-4 \Theta
\end{align}
 and all other brackets between the $w'$s are 0. Moreover in $U(\g_{<0})$ we have:
\begin{align}
\label{quadratideiwcaselli}
w_{11}^{2}=w_{22}^{2}=w_{12}^{2}=w_{21}^{2}=0.
\end{align}
 Indeed for example $w_{11}^{2}=(\eta_{2}+i\eta_{1})(\eta_{2}+i\eta_{1})=\Theta+i \eta_{21}+i\eta_{12}-\Theta=0$.

In \cite{K4} it is presented the classification of all highest weight singular vectors (i.e. singular vectors $\vec{m}$ such that $e_1.\vec{m}=e_2.\vec{m}=0$). We recall the following remark from \cite{K4}. 
\begin{rem}
		\label{costruzionemorfismi}
		We recall that highest weight singular vectors allow to construct $\g-$morphisms between finite Verma modules. Indeed, let us call $M(\mu_{1},\mu_{2},\mu_{3},\mu_{4})$ the Verma module $\Ind (F (\mu_{1},\mu_{2},\mu_{3},\mu_{4}))$, where $F (\mu_{1},\mu_{2},\mu_{3},\mu_{4})$ is the irreducible $\g_{0}-$module with highest weight $(\mu_{1},\mu_{2},\mu_{3},\mu_{4})$. We call a Verma module \textit{degenerate} if it is not irreducible.
We point out that, given $M(\mu_{1},\mu_{2},\mu_{3},\mu_{4})$ and $M(\widetilde{\mu}_{1},\widetilde{\mu}_{2},\widetilde{\mu}_{3},\widetilde{\mu}_{4})$ finite Verma modules, we can construct a non trivial morphism of $\g-$modules from the former to the latter if and only if there exists a highest weight singular vector $\vec{m}$ in $M(\widetilde{\mu}_{1},\widetilde{\mu}_{2},\widetilde{\mu}_{3},\widetilde{\mu}_{4})$ of highest weight $(\mu_{1},\mu_{2},\mu_{3},\mu_{4})$. The map is uniquely determined by:
 	  \begin{align*}
		\nabla: \, M(\mu_{1},\mu_{2},\mu_{3},\mu_{4})&\longrightarrow M(\widetilde{\mu}_{1},\widetilde{\mu}_{2},\widetilde{\mu}_{3},\widetilde{\mu}_{4})\\
		               v_{\mu} &\longmapsto \vec{m},
		\end{align*}
 where $v_{\mu}$ is a highest weight vector of $F(\mu_{1},\mu_{2},\mu_{3},\mu_{4})$.
If $\vec{m}$ is a singular vector of degree $d$, we say that $\nabla$ is a morphism of degree $d$.
		\end{rem}
		Remark \ref{costruzionemorfismi} is used in \cite{K4} to construct the maps in Figure \ref{figura} of all possible morphisms between finite Verma modules in the case of $K'_{4}$. The maps will be described explicitly in section \ref{sezmorfismi}. 
		\begin{figure}[h!]
		\centering
		\caption{}
		 \label{figura}
		\begin{tikzpicture}
 
\draw[->,black] (0,0) -- (6,0) node[anchor=north west]{};
\draw[->,black] (0,0) -- (0,6) node[anchor=south east]{};
\node[black] at (3,6.5) {$(m,n,-\frac{m+n}{2},\frac{m-n}{2})$};
\node[black] at (6,6.5) {\textbf{A}};
\node[black] at (6.5,0) {$m$};
\node[black] at (0,6.5) {$n$};
\draw[black] (0,0) -- (5.5,5.5);
\draw[black] (0,1) -- (4.5,5.5);
\draw[black] (0,2) -- (3.5,5.5);
\draw[black] (0,3) -- (2.5,5.5);
\draw[black] (0,4) -- (1.5,5.5);
\draw[black] (0,5) -- (0.5,5.5);
\draw[black] (1,0) -- (5.5,4.5);
\draw[black] (2,0) -- (5.5,3.5);
\draw[black] (3,0) -- (5.5,2.5);
\draw[black] (4,0) -- (5.5,1.5);
\draw[black] (5,0) -- (5.5,0.5);

\draw[-latex] (1,0) to[out=200,in=70] (-0.9,-1.9);
\draw[-latex] (0,1) to[out=-110,in=20] (-1.9,-0.9);

\draw[-latex,black] (0,1.9) -- (-0.9,0.1);
\draw[-latex,black] (0,2.9) -- (-0.9,1.1);
\draw[-latex,black] (0,3.9) -- (-0.9,2.1);
\draw[-latex,black] (0,4.9) -- (-0.9,3.1);
\draw[-latex,black] (0,5.9) -- (-0.9,4.1);
\draw[-latex,black] (-0.5,5.9) -- (-0.9,5.1);

\draw[-latex,black] (2,0) -- (0.1,-0.9);
\draw[-latex,black] (3,0) -- (1.1,-0.9);
\draw[-latex,black] (4,0) -- (2.1,-0.9);
\draw[-latex,black] (5,0) -- (3.1,-0.9);
\draw[-latex,black] (6,0) -- (4.1,-0.9);
\draw[-latex,black] (6,-0.5) -- (5.1,-0.9);

\draw[-latex,black] (0,-1) -- (-0.9,-2.9);
\draw[-latex,black] (0,-2) -- (-0.9,-3.9);
\draw[-latex,black] (0,-3) -- (-0.9,-4.9);
\draw[-latex,black] (0,-4) -- (-0.9,-5.9);
\draw[-latex,black] (0,-5) -- (-0.9,-6.9);
\draw[-latex,black] (0,-6) -- (-0.5,-7.1);

\draw[-latex,black] (-1,0) -- (-2.9,-0.9);
\draw[-latex,black] (-2,0) -- (-3.9,-0.9);
\draw[-latex,black] (-3,0) -- (-4.9,-0.9);
\draw[-latex,black] (-4,0) -- (-5.9,-0.9);
\draw[-latex,black] (-5,0) -- (-6.9,-0.9);
\draw[-latex,black] (-6,0) -- (-7.1,-0.5);

\draw[-latex,black] (5.5,5.5) -- (5.1,5.1);
\draw[-latex,black] (5,5) -- (4.1,4.1);
\draw[-latex,black] (4,4) -- (3.1,3.1);
\draw[-latex,black] (3,3) -- (2.1,2.1);
\draw[-latex,black] (2,2) -- (1.1,1.1);
\draw[-latex,black] (1,1) -- (0.1,0.1);

\draw[-latex,black] (5.5,4.5) -- (5.1,4.1);
\draw[-latex,black] (5,4) -- (4.1,3.1);
\draw[-latex,black] (4,3) -- (3.1,2.1);
\draw[-latex,black] (3,2) -- (2.1,1.1);
\draw[-latex,black] (2,1) -- (1.1,0.1);

\draw[-latex,black] (5.5,3.5) -- (5.1,3.1);
\draw[-latex,black] (5,3) -- (4.1,2.1);
\draw[-latex,black] (4,2) -- (3.1,1.1);
\draw[-latex,black] (3,1) -- (2.1,0.1);

\draw[-latex,black] (5.5,2.5) -- (5.1,2.1);
\draw[-latex,black] (5,2) -- (4.1,1.1);
\draw[-latex,black] (4,1) -- (3.1,0.1);

\draw[-latex,black] (5.5,1.5) -- (5.1,1.1);
\draw[-latex,black] (5,1) -- (4.1,0.1);

\draw[-latex,black] (5.5,0.5) -- (5.1,0.1);

\draw[-latex,black] (4.5,5.5) -- (4.1,5.1);
\draw[-latex,black] (4,5) -- (3.1,4.1);
\draw[-latex,black] (3,4) -- (2.1,3.1);
\draw[-latex,black] (2,3) -- (1.1,2.1);
\draw[-latex,black] (1,2) -- (0.1,1.1);

\draw[-latex,black] (3.5,5.5) -- (3.1,5.1);
\draw[-latex,black] (3,5) -- (2.1,4.1);
\draw[-latex,black] (2,4) -- (1.1,3.1);
\draw[-latex,black] (1,3) -- (0.1,2.1);

\draw[-latex,black] (2.5,5.5) -- (2.1,5.1);
\draw[-latex,black] (2,5) -- (1.1,4.1);
\draw[-latex,black] (1,4) -- (0.1,3.1);

\draw[-latex,black] (1.5,5.5) -- (1.1,5.1);
\draw[-latex,black] (1,5) -- (0.1,4.1);

\draw[-latex,black] (0.5,5.5) -- (0.1,5.1);

\coordinate (center) at (0,0);
\fill[white] (center) + (0, 0.1) arc (90:270:0.1);
\fill[white] (center)+ (0, -0.1) arc (270:450:0.1);
\draw[black] (center)+ (0, -0.1) arc (270:450:0.1);
\draw[black] (center)+ (0, 0.1) arc (90:270:0.1);
\coordinate (center) at (0,1);
\fill[white] (center) + (0, 0.1) arc (90:270:0.1);
\fill[white] (center)+ (0, -0.1) arc (270:450:0.1);
\draw[black] (center)+ (0, -0.1) arc (270:450:0.1);
\draw[black] (center)+ (0, 0.1) arc (90:270:0.1);
\coordinate (center) at (0,2);
\fill[white] (center) + (0, 0.1) arc (90:270:0.1);
\fill[white] (center)+ (0, -0.1) arc (270:450:0.1);
\draw[black] (center)+ (0, -0.1) arc (270:450:0.1);
\draw[black] (center)+ (0, 0.1) arc (90:270:0.1);
\coordinate (center) at (0,3);
\fill[white] (center) + (0, 0.1) arc (90:270:0.1);
\fill[white] (center)+ (0, -0.1) arc (270:450:0.1);
\draw[black] (center)+ (0, -0.1) arc (270:450:0.1);
\draw[black] (center)+ (0, 0.1) arc (90:270:0.1);
\coordinate (center) at (0,4);
\fill[white] (center) + (0, 0.1) arc (90:270:0.1);
\fill[white] (center)+ (0, -0.1) arc (270:450:0.1);
\draw[black] (center)+ (0, -0.1) arc (270:450:0.1);
\draw[black] (center)+ (0, 0.1) arc (90:270:0.1);
\coordinate (center) at (0,5);
\fill[white] (center) + (0, 0.1) arc (90:270:0.1);
\fill[white] (center)+ (0, -0.1) arc (270:450:0.1);
\draw[black] (center)+ (0, -0.1) arc (270:450:0.1);
\draw[black] (center)+ (0, 0.1) arc (90:270:0.1);
\coordinate (center) at (1,0);
\fill[white] (center) + (0, 0.1) arc (90:270:0.1);
\fill[white] (center)+ (0, -0.1) arc (270:450:0.1);
\draw[black] (center)+ (0, -0.1) arc (270:450:0.1);
\draw[black] (center)+ (0, 0.1) arc (90:270:0.1);
\coordinate (center) at (2,0);
\fill[white] (center) + (0, 0.1) arc (90:270:0.1);
\fill[white] (center)+ (0, -0.1) arc (270:450:0.1);
\draw[black] (center)+ (0, -0.1) arc (270:450:0.1);
\draw[black] (center)+ (0, 0.1) arc (90:270:0.1);
\coordinate (center) at (3,0);
\fill[white] (center) + (0, 0.1) arc (90:270:0.1);
\fill[white] (center)+ (0, -0.1) arc (270:450:0.1);
\draw[black] (center)+ (0, -0.1) arc (270:450:0.1);
\draw[black] (center)+ (0, 0.1) arc (90:270:0.1);
\coordinate (center) at (4,0);
\fill[white] (center) + (0, 0.1) arc (90:270:0.1);
\fill[white] (center)+ (0, -0.1) arc (270:450:0.1);
\draw[black] (center)+ (0, -0.1) arc (270:450:0.1);
\draw[black] (center)+ (0, 0.1) arc (90:270:0.1);
\coordinate (center) at (5,0);
\fill[white] (center) + (0, 0.1) arc (90:270:0.1);
\fill[white] (center)+ (0, -0.1) arc (270:450:0.1);
\draw[black] (center)+ (0, -0.1) arc (270:450:0.1);
\draw[black] (center)+ (0, 0.1) arc (90:270:0.1);

\coordinate (center) at (1,1);
\fill[white] (center) + (0, 0.1) arc (90:270:0.1);
\fill[white] (center)+ (0, -0.1) arc (270:450:0.1);
\draw[black] (center)+ (0, -0.1) arc (270:450:0.1);
\draw[black] (center)+ (0, 0.1) arc (90:270:0.1);
\coordinate (center) at (1,2);
\fill[white] (center) + (0, 0.1) arc (90:270:0.1);
\fill[white] (center)+ (0, -0.1) arc (270:450:0.1);
\draw[black] (center)+ (0, -0.1) arc (270:450:0.1);
\draw[black] (center)+ (0, 0.1) arc (90:270:0.1);
\coordinate (center) at (1,3);
\fill[white] (center) + (0, 0.1) arc (90:270:0.1);
\fill[white] (center)+ (0, -0.1) arc (270:450:0.1);
\draw[black] (center)+ (0, -0.1) arc (270:450:0.1);
\draw[black] (center)+ (0, 0.1) arc (90:270:0.1);
\coordinate (center) at (1,4);
\fill[white] (center) + (0, 0.1) arc (90:270:0.1);
\fill[white] (center)+ (0, -0.1) arc (270:450:0.1);
\draw[black] (center)+ (0, -0.1) arc (270:450:0.1);
\draw[black] (center)+ (0, 0.1) arc (90:270:0.1);
\coordinate (center) at (1,5);
\fill[white] (center) + (0, 0.1) arc (90:270:0.1);
\fill[white] (center)+ (0, -0.1) arc (270:450:0.1);
\draw[black] (center)+ (0, -0.1) arc (270:450:0.1);
\draw[black] (center)+ (0, 0.1) arc (90:270:0.1);
\coordinate (center) at (2,1);
\fill[white] (center) + (0, 0.1) arc (90:270:0.1);
\fill[white] (center)+ (0, -0.1) arc (270:450:0.1);
\draw[black] (center)+ (0, -0.1) arc (270:450:0.1);
\draw[black] (center)+ (0, 0.1) arc (90:270:0.1);
\coordinate (center) at (2,2);
\fill[white] (center) + (0, 0.1) arc (90:270:0.1);
\fill[white] (center)+ (0, -0.1) arc (270:450:0.1);
\draw[black] (center)+ (0, -0.1) arc (270:450:0.1);
\draw[black] (center)+ (0, 0.1) arc (90:270:0.1);
\coordinate (center) at (2,3);
\fill[white] (center) + (0, 0.1) arc (90:270:0.1);
\fill[white] (center)+ (0, -0.1) arc (270:450:0.1);
\draw[black] (center)+ (0, -0.1) arc (270:450:0.1);
\draw[black] (center)+ (0, 0.1) arc (90:270:0.1);
\coordinate (center) at (2,4);
\fill[white] (center) + (0, 0.1) arc (90:270:0.1);
\fill[white] (center)+ (0, -0.1) arc (270:450:0.1);
\draw[black] (center)+ (0, -0.1) arc (270:450:0.1);
\draw[black] (center)+ (0, 0.1) arc (90:270:0.1);
\coordinate (center) at (2,5);
\fill[white] (center) + (0, 0.1) arc (90:270:0.1);
\fill[white] (center)+ (0, -0.1) arc (270:450:0.1);
\draw[black] (center)+ (0, -0.1) arc (270:450:0.1);
\draw[black] (center)+ (0, 0.1) arc (90:270:0.1);
\coordinate (center) at (3,1);
\fill[white] (center) + (0, 0.1) arc (90:270:0.1);
\fill[white] (center)+ (0, -0.1) arc (270:450:0.1);
\draw[black] (center)+ (0, -0.1) arc (270:450:0.1);
\draw[black] (center)+ (0, 0.1) arc (90:270:0.1);
\coordinate (center) at (3,2);
\fill[white] (center) + (0, 0.1) arc (90:270:0.1);
\fill[white] (center)+ (0, -0.1) arc (270:450:0.1);
\draw[black] (center)+ (0, -0.1) arc (270:450:0.1);
\draw[black] (center)+ (0, 0.1) arc (90:270:0.1);
\coordinate (center) at (3,3);
\fill[white] (center) + (0, 0.1) arc (90:270:0.1);
\fill[white] (center)+ (0, -0.1) arc (270:450:0.1);
\draw[black] (center)+ (0, -0.1) arc (270:450:0.1);
\draw[black] (center)+ (0, 0.1) arc (90:270:0.1);
\coordinate (center) at (3,4);
\fill[white] (center) + (0, 0.1) arc (90:270:0.1);
\fill[white] (center)+ (0, -0.1) arc (270:450:0.1);
\draw[black] (center)+ (0, -0.1) arc (270:450:0.1);
\draw[black] (center)+ (0, 0.1) arc (90:270:0.1);
\coordinate (center) at (3,5);
\fill[white] (center) + (0, 0.1) arc (90:270:0.1);
\fill[white] (center)+ (0, -0.1) arc (270:450:0.1);
\draw[black] (center)+ (0, -0.1) arc (270:450:0.1);
\draw[black] (center)+ (0, 0.1) arc (90:270:0.1);
\coordinate (center) at (4,1);
\fill[white] (center) + (0, 0.1) arc (90:270:0.1);
\fill[white] (center)+ (0, -0.1) arc (270:450:0.1);
\draw[black] (center)+ (0, -0.1) arc (270:450:0.1);
\draw[black] (center)+ (0, 0.1) arc (90:270:0.1);
\coordinate (center) at (4,2);
\fill[white] (center) + (0, 0.1) arc (90:270:0.1);
\fill[white] (center)+ (0, -0.1) arc (270:450:0.1);
\draw[black] (center)+ (0, -0.1) arc (270:450:0.1);
\draw[black] (center)+ (0, 0.1) arc (90:270:0.1);
\coordinate (center) at (4,3);
\fill[white] (center) + (0, 0.1) arc (90:270:0.1);
\fill[white] (center)+ (0, -0.1) arc (270:450:0.1);
\draw[black] (center)+ (0, -0.1) arc (270:450:0.1);
\draw[black] (center)+ (0, 0.1) arc (90:270:0.1);
\coordinate (center) at (4,4);
\fill[white] (center) + (0, 0.1) arc (90:270:0.1);
\fill[white] (center)+ (0, -0.1) arc (270:450:0.1);
\draw[black] (center)+ (0, -0.1) arc (270:450:0.1);
\draw[black] (center)+ (0, 0.1) arc (90:270:0.1);
\coordinate (center) at (4,5);
\fill[white] (center) + (0, 0.1) arc (90:270:0.1);
\fill[white] (center)+ (0, -0.1) arc (270:450:0.1);
\draw[black] (center)+ (0, -0.1) arc (270:450:0.1);
\draw[black] (center)+ (0, 0.1) arc (90:270:0.1);
\coordinate (center) at (5,1);
\fill[white] (center) + (0, 0.1) arc (90:270:0.1);
\fill[white] (center)+ (0, -0.1) arc (270:450:0.1);
\draw[black] (center)+ (0, -0.1) arc (270:450:0.1);
\draw[black] (center)+ (0, 0.1) arc (90:270:0.1);
\coordinate (center) at (5,2);
\fill[white] (center) + (0, 0.1) arc (90:270:0.1);
\fill[white] (center)+ (0, -0.1) arc (270:450:0.1);
\draw[black] (center)+ (0, -0.1) arc (270:450:0.1);
\draw[black] (center)+ (0, 0.1) arc (90:270:0.1);
\coordinate (center) at (5,3);
\fill[white] (center) + (0, 0.1) arc (90:270:0.1);
\fill[white] (center)+ (0, -0.1) arc (270:450:0.1);
\draw[black] (center)+ (0, -0.1) arc (270:450:0.1);
\draw[black] (center)+ (0, 0.1) arc (90:270:0.1);
\coordinate (center) at (5,4);
\fill[white] (center) + (0, 0.1) arc (90:270:0.1);
\fill[white] (center)+ (0, -0.1) arc (270:450:0.1);
\draw[black] (center)+ (0, -0.1) arc (270:450:0.1);
\draw[black] (center)+ (0, 0.1) arc (90:270:0.1);
\coordinate (center) at (5,5);
\fill[white] (center) + (0, 0.1) arc (90:270:0.1);
\fill[white] (center)+ (0, -0.1) arc (270:450:0.1);
\draw[black] (center)+ (0, -0.1) arc (270:450:0.1);
\draw[black] (center)+ (0, 0.1) arc (90:270:0.1);

\draw[->,black] (-1,-1) -- (-7,-1) node[anchor=north west]{};
\draw[->,black] (-1,-1) -- (-1,-7) node[anchor=south east]{};
\node[black] at (-1,-7.5) {$n$};
\node[black] at (-7.5,-1) {$m$};
 \node[black] at (-4,-7) {($m,n,\frac{m+n}{2}+2,\frac{n-m}{2}$)};
\node[black] at (-7,-7) {\textbf{C}};
\draw[black] (-1,-1) -- (-6.5,-6.5);
\draw[black] (-2,-1) -- (-6.5,-5.5);
\draw[black] (-3,-1) -- (-6.5,-4.5);
\draw[black] (-4,-1) -- (-6.5,-3.5);
\draw[black] (-5,-1) -- (-6.5,-2.5);
\draw[black] (-6,-1) -- (-6.5,-1.5);
\draw[black] (-1,-2) -- (-5.5,-6.5);
\draw[black] (-1,-3) -- (-4.5,-6.5);
\draw[black] (-1,-4) -- (-3.5,-6.5);
\draw[black] (-1,-5) -- (-2.5,-6.5);
\draw[black] (-1,-6) -- (-1.5,-6.5);

\draw[-latex,black] (-1,-1) -- (-1.9,-1.9);
\draw[-latex,black] (-2,-2) -- (-2.9,-2.9);
\draw[-latex,black] (-3,-3) -- (-3.9,-3.9);
\draw[-latex,black] (-4,-4) -- (-4.9,-4.9);
\draw[-latex,black] (-5,-5) -- (-5.9,-5.9);
\draw[-latex,black] (-6,-6) -- (-6.5,-6.5);

\draw[-latex,black] (-1,-2) -- (-1.9,-2.9);
\draw[-latex,black] (-2,-3) -- (-2.9,-3.9);
\draw[-latex,black] (-3,-4) -- (-3.9,-4.9);
\draw[-latex,black] (-4,-5) -- (-4.9,-5.9);
\draw[-latex,black] (-5,-6) -- (-5.5,-6.5);

\draw[-latex,black] (-1,-3) -- (-1.9,-3.9);
\draw[-latex,black] (-2,-4) -- (-2.9,-4.9);
\draw[-latex,black] (-3,-5) -- (-3.9,-5.9);
\draw[-latex,black] (-4,-6) -- (-4.5,-6.5);

\draw[-latex,black] (-1,-4) -- (-1.9,-4.9);
\draw[-latex,black] (-2,-5) -- (-2.9,-5.9);
\draw[-latex,black] (-3,-6) -- (-3.5,-6.5);

\draw[-latex,black] (-1,-5) -- (-1.9,-5.9);
\draw[-latex,black] (-2,-6) -- (-2.5,-6.5);

\draw[-latex,black] (-1,-6) -- (-1.5,-6.5);

\draw[-latex,black] (-2,-1) -- (-2.9,-1.9);
\draw[-latex,black] (-3,-2) -- (-3.9,-2.9);
\draw[-latex,black] (-4,-3) -- (-4.9,-3.9);
\draw[-latex,black] (-5,-4) -- (-5.9,-4.9);
\draw[-latex,black] (-6,-5) -- (-6.5,-5.5);

\draw[-latex,black] (-3,-1) -- (-3.9,-1.9);
\draw[-latex,black] (-4,-2) -- (-4.9,-2.9);
\draw[-latex,black] (-5,-3) -- (-5.9,-3.9);
\draw[-latex,black] (-6,-4) -- (-6.5,-4.5);

\draw[-latex,black] (-4,-1) -- (-4.9,-1.9);
\draw[-latex,black] (-5,-2) -- (-5.9,-2.9);
\draw[-latex,black] (-6,-3) -- (-6.5,-3.5);

\draw[-latex,black] (-5,-1) -- (-5.9,-1.9);
\draw[-latex,black] (-6,-2) -- (-6.5,-2.5);

\draw[-latex,black] (-6,-1) -- (-6.5,-1.5);

\coordinate (center) at (-1,-1);
\fill[white] (center) + (0, 0.1) arc (90:270:0.1);
\fill[white] (center)+ (0, -0.1) arc (270:450:0.1);
\draw[black] (center)+ (0, -0.1) arc (270:450:0.1);
\draw[black] (center)+ (0, 0.1) arc (90:270:0.1);
\coordinate (center) at (-1,-2);
\fill[white] (center) + (0, 0.1) arc (90:270:0.1);
\fill[white] (center)+ (0, -0.1) arc (270:450:0.1);
\draw[black] (center)+ (0, -0.1) arc (270:450:0.1);
\draw[black] (center)+ (0, 0.1) arc (90:270:0.1);
\coordinate (center) at (-1,-3);
\fill[white] (center) + (0, 0.1) arc (90:270:0.1);
\fill[white] (center)+ (0, -0.1) arc (270:450:0.1);
\draw[black] (center)+ (0, -0.1) arc (270:450:0.1);
\draw[black] (center)+ (0, 0.1) arc (90:270:0.1);
\coordinate (center) at (-1,-4);
\fill[white] (center) + (0, 0.1) arc (90:270:0.1);
\fill[white] (center)+ (0, -0.1) arc (270:450:0.1);
\draw[black] (center)+ (0, -0.1) arc (270:450:0.1);
\draw[black] (center)+ (0, 0.1) arc (90:270:0.1);
\coordinate (center) at (-1,-5);
\fill[white] (center) + (0, 0.1) arc (90:270:0.1);
\fill[white] (center)+ (0, -0.1) arc (270:450:0.1);
\draw[black] (center)+ (0, -0.1) arc (270:450:0.1);
\draw[black] (center)+ (0, 0.1) arc (90:270:0.1);
\coordinate (center) at (-1,-6);
\fill[white] (center) + (0, 0.1) arc (90:270:0.1);
\fill[white] (center)+ (0, -0.1) arc (270:450:0.1);
\draw[black] (center)+ (0, -0.1) arc (270:450:0.1);
\draw[black] (center)+ (0, 0.1) arc (90:270:0.1);
\coordinate (center) at (-2,-1);
\fill[white] (center) + (0, 0.1) arc (90:270:0.1);
\fill[white] (center)+ (0, -0.1) arc (270:450:0.1);
\draw[black] (center)+ (0, -0.1) arc (270:450:0.1);
\draw[black] (center)+ (0, 0.1) arc (90:270:0.1);
\coordinate (center) at (-3,-1);
\fill[white] (center) + (0, 0.1) arc (90:270:0.1);
\fill[white] (center)+ (0, -0.1) arc (270:450:0.1);
\draw[black] (center)+ (0, -0.1) arc (270:450:0.1);
\draw[black] (center)+ (0, 0.1) arc (90:270:0.1);
\coordinate (center) at (-4,-1);
\fill[white] (center) + (0, 0.1) arc (90:270:0.1);
\fill[white] (center)+ (0, -0.1) arc (270:450:0.1);
\draw[black] (center)+ (0, -0.1) arc (270:450:0.1);
\draw[black] (center)+ (0, 0.1) arc (90:270:0.1);
\coordinate (center) at (-5,-1);
\fill[white] (center) + (0, 0.1) arc (90:270:0.1);
\fill[white] (center)+ (0, -0.1) arc (270:450:0.1);
\draw[black] (center)+ (0, -0.1) arc (270:450:0.1);
\draw[black] (center)+ (0, 0.1) arc (90:270:0.1);
\coordinate (center) at (-6,-1);
\fill[white] (center) + (0, 0.1) arc (90:270:0.1);
\fill[white] (center)+ (0, -0.1) arc (270:450:0.1);
\draw[black] (center)+ (0, -0.1) arc (270:450:0.1);
\draw[black] (center)+ (0, 0.1) arc (90:270:0.1);

\coordinate (center) at (-2,-2);
\fill[white] (center) + (0, 0.1) arc (90:270:0.1);
\fill[white] (center)+ (0, -0.1) arc (270:450:0.1);
\draw[black] (center)+ (0, -0.1) arc (270:450:0.1);
\draw[black] (center)+ (0, 0.1) arc (90:270:0.1);
\coordinate (center) at (-2,-3);
\fill[white] (center) + (0, 0.1) arc (90:270:0.1);
\fill[white] (center)+ (0, -0.1) arc (270:450:0.1);
\draw[black] (center)+ (0, -0.1) arc (270:450:0.1);
\draw[black] (center)+ (0, 0.1) arc (90:270:0.1);

\coordinate (center) at (-2,-4);
\fill[white] (center) + (0, 0.1) arc (90:270:0.1);
\fill[white] (center)+ (0, -0.1) arc (270:450:0.1);
\draw[black] (center)+ (0, -0.1) arc (270:450:0.1);
\draw[black] (center)+ (0, 0.1) arc (90:270:0.1);
\coordinate (center) at (-2,-5);
\fill[white] (center) + (0, 0.1) arc (90:270:0.1);
\fill[white] (center)+ (0, -0.1) arc (270:450:0.1);
\draw[black] (center)+ (0, -0.1) arc (270:450:0.1);
\draw[black] (center)+ (0, 0.1) arc (90:270:0.1);
\coordinate (center) at (-2,-6);
\fill[white] (center) + (0, 0.1) arc (90:270:0.1);
\fill[white] (center)+ (0, -0.1) arc (270:450:0.1);
\draw[black] (center)+ (0, -0.1) arc (270:450:0.1);
\draw[black] (center)+ (0, 0.1) arc (90:270:0.1);
\coordinate (center) at (-3,-2);
\fill[white] (center) + (0, 0.1) arc (90:270:0.1);
\fill[white] (center)+ (0, -0.1) arc (270:450:0.1);
\draw[black] (center)+ (0, -0.1) arc (270:450:0.1);
\draw[black] (center)+ (0, 0.1) arc (90:270:0.1);
\coordinate (center) at (-3,-3);
\fill[white] (center) + (0, 0.1) arc (90:270:0.1);
\fill[white] (center)+ (0, -0.1) arc (270:450:0.1);
\draw[black] (center)+ (0, -0.1) arc (270:450:0.1);
\draw[black] (center)+ (0, 0.1) arc (90:270:0.1);
\coordinate (center) at (-3,-4);
\fill[white] (center) + (0, 0.1) arc (90:270:0.1);
\fill[white] (center)+ (0, -0.1) arc (270:450:0.1);
\draw[black] (center)+ (0, -0.1) arc (270:450:0.1);
\draw[black] (center)+ (0, 0.1) arc (90:270:0.1);
\coordinate (center) at (-3,-5);
\fill[white] (center) + (0, 0.1) arc (90:270:0.1);
\fill[white] (center)+ (0, -0.1) arc (270:450:0.1);
\draw[black] (center)+ (0, -0.1) arc (270:450:0.1);
\draw[black] (center)+ (0, 0.1) arc (90:270:0.1);
\coordinate (center) at (-3,-6);
\fill[white] (center) + (0, 0.1) arc (90:270:0.1);
\fill[white] (center)+ (0, -0.1) arc (270:450:0.1);
\draw[black] (center)+ (0, -0.1) arc (270:450:0.1);
\draw[black] (center)+ (0, 0.1) arc (90:270:0.1);
\coordinate (center) at (-4,-2);
\fill[white] (center) + (0, 0.1) arc (90:270:0.1);
\fill[white] (center)+ (0, -0.1) arc (270:450:0.1);
\draw[black] (center)+ (0, -0.1) arc (270:450:0.1);
\draw[black] (center)+ (0, 0.1) arc (90:270:0.1);
\coordinate (center) at (-4,-3);
\fill[white] (center) + (0, 0.1) arc (90:270:0.1);
\fill[white] (center)+ (0, -0.1) arc (270:450:0.1);
\draw[black] (center)+ (0, -0.1) arc (270:450:0.1);
\draw[black] (center)+ (0, 0.1) arc (90:270:0.1);
\coordinate (center) at (-4,-4);
\fill[white] (center) + (0, 0.1) arc (90:270:0.1);
\fill[white] (center)+ (0, -0.1) arc (270:450:0.1);
\draw[black] (center)+ (0, -0.1) arc (270:450:0.1);
\draw[black] (center)+ (0, 0.1) arc (90:270:0.1);
\coordinate (center) at (-4,-5);
\fill[white] (center) + (0, 0.1) arc (90:270:0.1);
\fill[white] (center)+ (0, -0.1) arc (270:450:0.1);
\draw[black] (center)+ (0, -0.1) arc (270:450:0.1);
\draw[black] (center)+ (0, 0.1) arc (90:270:0.1);
\coordinate (center) at (-4,-6);
\fill[white] (center) + (0, 0.1) arc (90:270:0.1);
\fill[white] (center)+ (0, -0.1) arc (270:450:0.1);
\draw[black] (center)+ (0, -0.1) arc (270:450:0.1);
\draw[black] (center)+ (0, 0.1) arc (90:270:0.1);
\coordinate (center) at (-5,-2);
\fill[white] (center) + (0, 0.1) arc (90:270:0.1);
\fill[white] (center)+ (0, -0.1) arc (270:450:0.1);
\draw[black] (center)+ (0, -0.1) arc (270:450:0.1);
\draw[black] (center)+ (0, 0.1) arc (90:270:0.1);
\coordinate (center) at (-5,-3);
\fill[white] (center) + (0, 0.1) arc (90:270:0.1);
\fill[white] (center)+ (0, -0.1) arc (270:450:0.1);
\draw[black] (center)+ (0, -0.1) arc (270:450:0.1);
\draw[black] (center)+ (0, 0.1) arc (90:270:0.1);
\coordinate (center) at (-5,-4);
\fill[white] (center) + (0, 0.1) arc (90:270:0.1);
\fill[white] (center)+ (0, -0.1) arc (270:450:0.1);
\draw[black] (center)+ (0, -0.1) arc (270:450:0.1);
\draw[black] (center)+ (0, 0.1) arc (90:270:0.1);
\coordinate (center) at (-5,-5);
\fill[white] (center) + (0, 0.1) arc (90:270:0.1);
\fill[white] (center)+ (0, -0.1) arc (270:450:0.1);
\draw[black] (center)+ (0, -0.1) arc (270:450:0.1);
\draw[black] (center)+ (0, 0.1) arc (90:270:0.1);

\coordinate (center) at (-5,-6);
\fill[white] (center) + (0, 0.1) arc (90:270:0.1);
\fill[white] (center)+ (0, -0.1) arc (270:450:0.1);
\draw[black] (center)+ (0, -0.1) arc (270:450:0.1);
\draw[black] (center)+ (0, 0.1) arc (90:270:0.1);
\coordinate (center) at (-6,-2);
\fill[white] (center) + (0, 0.1) arc (90:270:0.1);
\fill[white] (center)+ (0, -0.1) arc (270:450:0.1);
\draw[black] (center)+ (0, -0.1) arc (270:450:0.1);
\draw[black] (center)+ (0, 0.1) arc (90:270:0.1);
\coordinate (center) at (-6,-3);
\fill[white] (center) + (0, 0.1) arc (90:270:0.1);
\fill[white] (center)+ (0, -0.1) arc (270:450:0.1);
\draw[black] (center)+ (0, -0.1) arc (270:450:0.1);
\draw[black] (center)+ (0, 0.1) arc (90:270:0.1);
\coordinate (center) at (-6,-4);
\fill[white] (center) + (0, 0.1) arc (90:270:0.1);
\fill[white] (center)+ (0, -0.1) arc (270:450:0.1);
\draw[black] (center)+ (0, -0.1) arc (270:450:0.1);
\draw[black] (center)+ (0, 0.1) arc (90:270:0.1);
\coordinate (center) at (-6,-5);
\fill[white] (center) + (0, 0.1) arc (90:270:0.1);
\fill[white] (center)+ (0, -0.1) arc (270:450:0.1);
\draw[black] (center)+ (0, -0.1) arc (270:450:0.1);
\draw[black] (center)+ (0, 0.1) arc (90:270:0.1);

\coordinate (center) at (-6,-6);
\fill[white] (center) + (0, 0.1) arc (90:270:0.1);
\fill[white] (center)+ (0, -0.1) arc (270:450:0.1);
\draw[black] (center)+ (0, -0.1) arc (270:450:0.1);
\draw[black] (center)+ (0, 0.1) arc (90:270:0.1);

\draw[->,black] (0,-1) -- (6,-1) node[anchor=north west]{};
\draw[->,black] (0,-1) -- (0,-7) node[anchor=south east]{};
\node[black] at (0,-7.5) {$n$};
\node[black] at (6.5,-1) {$m$};
\node[black] at (3,-7) {$(m,n,1+\frac{n-m}{2},1+\frac{n+m}{2})$};
\node[black] at (6,-7) {\textbf{D}};
\draw[black] (1,-1) -- (0,-2);
\draw[black] (2,-1) -- (0,-3);
\draw[black] (3,-1) -- (0,-4);
\draw[black] (4,-1) -- (0,-5);
\draw[black] (5,-1) -- (0,-6);
\draw[black] (5.5,-1.5) -- (0.5,-6.5);
\draw[black] (5.5,-2.5) -- (1.5,-6.5);
\draw[black] (5.5,-3.5) -- (2.5,-6.5);
\draw[black] (5.5,-4.5) -- (3.5,-6.5);
\draw[black] (5.5,-5.5) -- (4.5,-6.5);

\draw[-latex,black] (1,-1) -- (0.1,-1.9);

\draw[-latex,black] (2,-1) -- (1.1,-1.9);
\draw[-latex,black] (1,-2) -- (0.1,-2.9);

\draw[-latex,black] (3,-1) -- (2.1,-1.9);
\draw[-latex,black] (2,-2) -- (1.1,-2.9);
\draw[-latex,black] (1,-3) -- (0.1,-3.9);

\draw[-latex,black] (4,-1) -- (3.1,-1.9);
\draw[-latex,black] (3,-2) -- (2.1,-2.9);
\draw[-latex,black] (2,-3) -- (1.1,-3.9);
\draw[-latex,black] (1,-4) -- (0.1,-4.9);

\draw[-latex,black] (5,-1) -- (4.1,-1.9);
\draw[-latex,black] (4,-2) -- (3.1,-2.9);
\draw[-latex,black] (3,-3) -- (2.1,-3.9);
\draw[-latex,black] (2,-4) -- (1.1,-4.9);
\draw[-latex,black] (1,-5) -- (0.1,-5.9);

\draw[-latex,black] (5.5,-1.5) -- (5.1,-1.9);
\draw[-latex,black] (5,-2) -- (4.1,-2.9);
\draw[-latex,black] (4,-3) -- (3.1,-3.9);
\draw[-latex,black] (3,-4) -- (2.1,-4.9);
\draw[-latex,black] (2,-5) -- (1.1,-5.9);
\draw[-latex,black] (1,-6) -- (0.5,-6.5);

\draw[-latex,black] (5.5,-2.5) -- (5.1,-2.9);
\draw[-latex,black] (5,-3) -- (4.1,-3.9);
\draw[-latex,black] (4,-4) -- (3.1,-4.9);
\draw[-latex,black] (3,-5) -- (2.1,-5.9);
\draw[-latex,black] (2,-6) -- (1.5,-6.5);

\draw[-latex,black] (5.5,-3.5) -- (5.1,-3.9);
\draw[-latex,black] (5,-4) -- (4.1,-4.9);
\draw[-latex,black] (4,-5) -- (3.1,-5.9);
\draw[-latex,black] (3,-6) -- (2.5,-6.5);

\draw[-latex,black] (5.5,-4.5) -- (5.1,-4.9);
\draw[-latex,black] (5,-5) -- (4.1,-5.9);
\draw[-latex,black] (4,-6) -- (3.5,-6.5);

\draw[-latex,black] (5.5,-5.5) -- (5.1,-5.9);
\draw[-latex,black] (5,-6) -- (4.5,-6.5);

\coordinate (center) at (0,-2);
\fill[white] (center) + (0, 0.1) arc (90:270:0.1);
\fill[white] (center)+ (0, -0.1) arc (270:450:0.1);
\draw[black] (center)+ (0, -0.1) arc (270:450:0.1);
\draw[black] (center)+ (0, 0.1) arc (90:270:0.1);
\coordinate (center) at (0,-3);
\fill[white] (center) + (0, 0.1) arc (90:270:0.1);
\fill[white] (center)+ (0, -0.1) arc (270:450:0.1);
\draw[black] (center)+ (0, -0.1) arc (270:450:0.1);
\draw[black] (center)+ (0, 0.1) arc (90:270:0.1);
\coordinate (center) at (0,-4);
\fill[white] (center) + (0, 0.1) arc (90:270:0.1);
\fill[white] (center)+ (0, -0.1) arc (270:450:0.1);
\draw[black] (center)+ (0, -0.1) arc (270:450:0.1);
\draw[black] (center)+ (0, 0.1) arc (90:270:0.1);
\coordinate (center) at (0,-5);
\fill[white] (center) + (0, 0.1) arc (90:270:0.1);
\fill[white] (center)+ (0, -0.1) arc (270:450:0.1);
\draw[black] (center)+ (0, -0.1) arc (270:450:0.1);
\draw[black] (center)+ (0, 0.1) arc (90:270:0.1);
\coordinate (center) at (0,-6);
\fill[white] (center) + (0, 0.1) arc (90:270:0.1);
\fill[white] (center)+ (0, -0.1) arc (270:450:0.1);
\draw[black] (center)+ (0, -0.1) arc (270:450:0.1);
\draw[black] (center)+ (0, 0.1) arc (90:270:0.1);
\coordinate (center) at (1,-1);
\fill[white] (center) + (0, 0.1) arc (90:270:0.1);
\fill[white] (center)+ (0, -0.1) arc (270:450:0.1);
\draw[black] (center)+ (0, -0.1) arc (270:450:0.1);
\draw[black] (center)+ (0, 0.1) arc (90:270:0.1);
\coordinate (center) at (2,-1);
\fill[white] (center) + (0, 0.1) arc (90:270:0.1);
\fill[white] (center)+ (0, -0.1) arc (270:450:0.1);
\draw[black] (center)+ (0, -0.1) arc (270:450:0.1);
\draw[black] (center)+ (0, 0.1) arc (90:270:0.1);
\coordinate (center) at (3,-1);
\fill[white] (center) + (0, 0.1) arc (90:270:0.1);
\fill[white] (center)+ (0, -0.1) arc (270:450:0.1);
\draw[black] (center)+ (0, -0.1) arc (270:450:0.1);
\draw[black] (center)+ (0, 0.1) arc (90:270:0.1);
\coordinate (center) at (4,-1);
\fill[white] (center) + (0, 0.1) arc (90:270:0.1);
\fill[white] (center)+ (0, -0.1) arc (270:450:0.1);
\draw[black] (center)+ (0, -0.1) arc (270:450:0.1);
\draw[black] (center)+ (0, 0.1) arc (90:270:0.1);
\coordinate (center) at (5,-1);
\fill[white] (center) + (0, 0.1) arc (90:270:0.1);
\fill[white] (center)+ (0, -0.1) arc (270:450:0.1);
\draw[black] (center)+ (0, -0.1) arc (270:450:0.1);
\draw[black] (center)+ (0, 0.1) arc (90:270:0.1);

\coordinate (center) at (1,-2);
\fill[white] (center) + (0, 0.1) arc (90:270:0.1);
\fill[white] (center)+ (0, -0.1) arc (270:450:0.1);
\draw[black] (center)+ (0, -0.1) arc (270:450:0.1);
\draw[black] (center)+ (0, 0.1) arc (90:270:0.1);
\coordinate (center) at (1,-3);
\fill[white] (center) + (0, 0.1) arc (90:270:0.1);
\fill[white] (center)+ (0, -0.1) arc (270:450:0.1);
\draw[black] (center)+ (0, -0.1) arc (270:450:0.1);
\draw[black] (center)+ (0, 0.1) arc (90:270:0.1);

\coordinate (center) at (1,-4);
\fill[white] (center) + (0, 0.1) arc (90:270:0.1);
\fill[white] (center)+ (0, -0.1) arc (270:450:0.1);
\draw[black] (center)+ (0, -0.1) arc (270:450:0.1);
\draw[black] (center)+ (0, 0.1) arc (90:270:0.1);
\coordinate (center) at (1,-5);
\fill[white] (center) + (0, 0.1) arc (90:270:0.1);
\fill[white] (center)+ (0, -0.1) arc (270:450:0.1);
\draw[black] (center)+ (0, -0.1) arc (270:450:0.1);
\draw[black] (center)+ (0, 0.1) arc (90:270:0.1);
\coordinate (center) at (1,-6);
\fill[white] (center) + (0, 0.1) arc (90:270:0.1);
\fill[white] (center)+ (0, -0.1) arc (270:450:0.1);
\draw[black] (center)+ (0, -0.1) arc (270:450:0.1);
\draw[black] (center)+ (0, 0.1) arc (90:270:0.1);
\coordinate (center) at (2,-2);
\fill[white] (center) + (0, 0.1) arc (90:270:0.1);
\fill[white] (center)+ (0, -0.1) arc (270:450:0.1);
\draw[black] (center)+ (0, -0.1) arc (270:450:0.1);
\draw[black] (center)+ (0, 0.1) arc (90:270:0.1);
\coordinate (center) at (2,-3);
\fill[white] (center) + (0, 0.1) arc (90:270:0.1);
\fill[white] (center)+ (0, -0.1) arc (270:450:0.1);
\draw[black] (center)+ (0, -0.1) arc (270:450:0.1);
\draw[black] (center)+ (0, 0.1) arc (90:270:0.1);
\coordinate (center) at (2,-4);
\fill[white] (center) + (0, 0.1) arc (90:270:0.1);
\fill[white] (center)+ (0, -0.1) arc (270:450:0.1);
\draw[black] (center)+ (0, -0.1) arc (270:450:0.1);
\draw[black] (center)+ (0, 0.1) arc (90:270:0.1);
\coordinate (center) at (2,-5);
\fill[white] (center) + (0, 0.1) arc (90:270:0.1);
\fill[white] (center)+ (0, -0.1) arc (270:450:0.1);
\draw[black] (center)+ (0, -0.1) arc (270:450:0.1);
\draw[black] (center)+ (0, 0.1) arc (90:270:0.1);
\coordinate (center) at (2,-6);
\fill[white] (center) + (0, 0.1) arc (90:270:0.1);
\fill[white] (center)+ (0, -0.1) arc (270:450:0.1);
\draw[black] (center)+ (0, -0.1) arc (270:450:0.1);
\draw[black] (center)+ (0, 0.1) arc (90:270:0.1);
\coordinate (center) at (3,-2);
\fill[white] (center) + (0, 0.1) arc (90:270:0.1);
\fill[white] (center)+ (0, -0.1) arc (270:450:0.1);
\draw[black] (center)+ (0, -0.1) arc (270:450:0.1);
\draw[black] (center)+ (0, 0.1) arc (90:270:0.1);
\coordinate (center) at (3,-3);
\fill[white] (center) + (0, 0.1) arc (90:270:0.1);
\fill[white] (center)+ (0, -0.1) arc (270:450:0.1);
\draw[black] (center)+ (0, -0.1) arc (270:450:0.1);
\draw[black] (center)+ (0, 0.1) arc (90:270:0.1);
\coordinate (center) at (3,-4);
\fill[white] (center) + (0, 0.1) arc (90:270:0.1);
\fill[white] (center)+ (0, -0.1) arc (270:450:0.1);
\draw[black] (center)+ (0, -0.1) arc (270:450:0.1);
\draw[black] (center)+ (0, 0.1) arc (90:270:0.1);
\coordinate (center) at (3,-5);
\fill[white] (center) + (0, 0.1) arc (90:270:0.1);
\fill[white] (center)+ (0, -0.1) arc (270:450:0.1);
\draw[black] (center)+ (0, -0.1) arc (270:450:0.1);
\draw[black] (center)+ (0, 0.1) arc (90:270:0.1);
\coordinate (center) at (3,-6);
\fill[white] (center) + (0, 0.1) arc (90:270:0.1);
\fill[white] (center)+ (0, -0.1) arc (270:450:0.1);
\draw[black] (center)+ (0, -0.1) arc (270:450:0.1);
\draw[black] (center)+ (0, 0.1) arc (90:270:0.1);
\coordinate (center) at (4,-2);
\fill[white] (center) + (0, 0.1) arc (90:270:0.1);
\fill[white] (center)+ (0, -0.1) arc (270:450:0.1);
\draw[black] (center)+ (0, -0.1) arc (270:450:0.1);
\draw[black] (center)+ (0, 0.1) arc (90:270:0.1);
\coordinate (center) at (4,-3);
\fill[white] (center) + (0, 0.1) arc (90:270:0.1);
\fill[white] (center)+ (0, -0.1) arc (270:450:0.1);
\draw[black] (center)+ (0, -0.1) arc (270:450:0.1);
\draw[black] (center)+ (0, 0.1) arc (90:270:0.1);
\coordinate (center) at (4,-4);
\fill[white] (center) + (0, 0.1) arc (90:270:0.1);
\fill[white] (center)+ (0, -0.1) arc (270:450:0.1);
\draw[black] (center)+ (0, -0.1) arc (270:450:0.1);
\draw[black] (center)+ (0, 0.1) arc (90:270:0.1);
\coordinate (center) at (4,-5);
\fill[white] (center) + (0, 0.1) arc (90:270:0.1);
\fill[white] (center)+ (0, -0.1) arc (270:450:0.1);
\draw[black] (center)+ (0, -0.1) arc (270:450:0.1);
\draw[black] (center)+ (0, 0.1) arc (90:270:0.1);

\coordinate (center) at (4,-6);
\fill[white] (center) + (0, 0.1) arc (90:270:0.1);
\fill[white] (center)+ (0, -0.1) arc (270:450:0.1);
\draw[black] (center)+ (0, -0.1) arc (270:450:0.1);
\draw[black] (center)+ (0, 0.1) arc (90:270:0.1);
\coordinate (center) at (5,-2);
\fill[white] (center) + (0, 0.1) arc (90:270:0.1);
\fill[white] (center)+ (0, -0.1) arc (270:450:0.1);
\draw[black] (center)+ (0, -0.1) arc (270:450:0.1);
\draw[black] (center)+ (0, 0.1) arc (90:270:0.1);
\coordinate (center) at (5,-3);
\fill[white] (center) + (0, 0.1) arc (90:270:0.1);
\fill[white] (center)+ (0, -0.1) arc (270:450:0.1);
\draw[black] (center)+ (0, -0.1) arc (270:450:0.1);
\draw[black] (center)+ (0, 0.1) arc (90:270:0.1);
\coordinate (center) at (5,-4);
\fill[white] (center) + (0, 0.1) arc (90:270:0.1);
\fill[white] (center)+ (0, -0.1) arc (270:450:0.1);
\draw[black] (center)+ (0, -0.1) arc (270:450:0.1);
\draw[black] (center)+ (0, 0.1) arc (90:270:0.1);
\coordinate (center) at (5,-5);
\fill[white] (center) + (0, 0.1) arc (90:270:0.1);
\fill[white] (center)+ (0, -0.1) arc (270:450:0.1);
\draw[black] (center)+ (0, -0.1) arc (270:450:0.1);
\draw[black] (center)+ (0, 0.1) arc (90:270:0.1);

\coordinate (center) at (5,-6);
\fill[white] (center) + (0, 0.1) arc (90:270:0.1);
\fill[white] (center)+ (0, -0.1) arc (270:450:0.1);
\draw[black] (center)+ (0, -0.1) arc (270:450:0.1);
\draw[black] (center)+ (0, 0.1) arc (90:270:0.1);

\draw[->,black] (-1,0) -- (-1,6) node[anchor=north west]{};
\draw[->,black] (-1,0) -- (-7,0) node[anchor=south east]{};
\node[black] at (-7.5,0) {$m$};
\node[black] at (-1,6.5) {$n$};
\node[black] at (-4,6.5) {$(m,n,1+\frac{m-n}{2},-1-\frac{m+n}{2})$};
\node[black] at (-7,6.5) {\textbf{B}};
\draw[black] (-2,0) -- (-1,1);
\draw[black] (-3,0) -- (-1,2);
\draw[black] (-4,0) -- (-1,3);
\draw[black] (-5,0) -- (-1,4);
\draw[black] (-6,0) -- (-1,5);
\draw[black] (-6.5,0.5) -- (-1.5,5.5);
\draw[black] (-6.5,1.5) -- (-2.5,5.5);
\draw[black] (-6.5,2.5) -- (-3.5,5.5);
\draw[black] (-6.5,3.5) -- (-4.5,5.5);
\draw[black] (-6.5,4.5) -- (-5.5,5.5);

\draw[-latex,black] (-1,1) -- (-1.9,0.1);

\draw[-latex,black] (-1,2) -- (-1.9,1.1);
\draw[-latex,black] (-2,1) -- (-2.9,0.1);

\draw[-latex,black] (-1,3) -- (-1.9,2.1);
\draw[-latex,black] (-2,2) -- (-2.9,1.1);
\draw[-latex,black] (-3,1) -- (-3.9,0.1);

\draw[-latex,black] (-1,4) -- (-1.9,3.1);
\draw[-latex,black] (-2,3) -- (-2.9,2.1);
\draw[-latex,black] (-3,2) -- (-3.9,1.1);
\draw[-latex,black] (-4,1) -- (-4.9,0.1);

\draw[-latex,black] (-1,5) -- (-1.9,4.1);
\draw[-latex,black] (-2,4) -- (-2.9,3.1);
\draw[-latex,black] (-3,3) -- (-3.9,2.1);
\draw[-latex,black] (-4,2) -- (-4.9,1.1);
\draw[-latex,black] (-5,1) -- (-5.9,0.1);

\draw[-latex,black] (-1.5,5.5) -- (-1.9,5.1);
\draw[-latex,black] (-2,5) -- (-2.9,4.1);
\draw[-latex,black] (-3,4) -- (-3.9,3.1);
\draw[-latex,black] (-4,3) -- (-4.9,2.1);
\draw[-latex,black] (-5,2) -- (-5.9,1.1);
\draw[-latex,black] (-6,1) -- (-6.5,0.5);

\draw[-latex,black] (-2.5,5.5) -- (-2.9,5.1);
\draw[-latex,black] (-3,5) -- (-3.9,4.1);
\draw[-latex,black] (-4,4) -- (-4.9,3.1);
\draw[-latex,black] (-5,3) -- (-5.9,2.1);
\draw[-latex,black] (-6,2) -- (-6.5,1.5);

\draw[-latex,black] (-3.5,5.5) -- (-3.9,5.1);
\draw[-latex,black] (-4,5) -- (-4.9,4.1);
\draw[-latex,black] (-5,4) -- (-5.9,3.1);
\draw[-latex,black] (-6,3) -- (-6.5,2.5);

\draw[-latex,black] (-4.5,5.5) -- (-4.9,5.1);
\draw[-latex,black] (-5,5) -- (-5.9,4.1);
\draw[-latex,black] (-6,4) -- (-6.5,3.5);

\draw[-latex,black] (-5.5,5.5) -- (-5.9,5.1);
\draw[-latex,black] (-6,5) -- (-6.5,4.5);
%

\coordinate (center) at (-1,1);
\fill[white] (center) + (0, 0.1) arc (90:270:0.1);
\fill[white] (center)+ (0, -0.1) arc (270:450:0.1);
\draw[black] (center)+ (0, -0.1) arc (270:450:0.1);
\draw[black] (center)+ (0, 0.1) arc (90:270:0.1);
\coordinate (center) at (-1,2);
\fill[white] (center) + (0, 0.1) arc (90:270:0.1);
\fill[white] (center)+ (0, -0.1) arc (270:450:0.1);
\draw[black] (center)+ (0, -0.1) arc (270:450:0.1);
\draw[black] (center)+ (0, 0.1) arc (90:270:0.1);
\coordinate (center) at (-1,3);
\fill[white] (center) + (0, 0.1) arc (90:270:0.1);
\fill[white] (center)+ (0, -0.1) arc (270:450:0.1);
\draw[black] (center)+ (0, -0.1) arc (270:450:0.1);
\draw[black] (center)+ (0, 0.1) arc (90:270:0.1);
\coordinate (center) at (-1,4);
\fill[white] (center) + (0, 0.1) arc (90:270:0.1);
\fill[white] (center)+ (0, -0.1) arc (270:450:0.1);
\draw[black] (center)+ (0, -0.1) arc (270:450:0.1);
\draw[black] (center)+ (0, 0.1) arc (90:270:0.1);
\coordinate (center) at (-1,5);
\fill[white] (center) + (0, 0.1) arc (90:270:0.1);
\fill[white] (center)+ (0, -0.1) arc (270:450:0.1);
\draw[black] (center)+ (0, -0.1) arc (270:450:0.1);
\draw[black] (center)+ (0, 0.1) arc (90:270:0.1);
\coordinate (center) at (-2,0);
\fill[white] (center) + (0, 0.1) arc (90:270:0.1);
\fill[white] (center)+ (0, -0.1) arc (270:450:0.1);
\draw[black] (center)+ (0, -0.1) arc (270:450:0.1);
\draw[black] (center)+ (0, 0.1) arc (90:270:0.1);
\coordinate (center) at (-3,0);
\fill[white] (center) + (0, 0.1) arc (90:270:0.1);
\fill[white] (center)+ (0, -0.1) arc (270:450:0.1);
\draw[black] (center)+ (0, -0.1) arc (270:450:0.1);
\draw[black] (center)+ (0, 0.1) arc (90:270:0.1);
\coordinate (center) at (-4,0);
\fill[white] (center) + (0, 0.1) arc (90:270:0.1);
\fill[white] (center)+ (0, -0.1) arc (270:450:0.1);
\draw[black] (center)+ (0, -0.1) arc (270:450:0.1);
\draw[black] (center)+ (0, 0.1) arc (90:270:0.1);
\coordinate (center) at (-5,0);
\fill[white] (center) + (0, 0.1) arc (90:270:0.1);
\fill[white] (center)+ (0, -0.1) arc (270:450:0.1);
\draw[black] (center)+ (0, -0.1) arc (270:450:0.1);
\draw[black] (center)+ (0, 0.1) arc (90:270:0.1);
\coordinate (center) at (-6,0);
\fill[white] (center) + (0, 0.1) arc (90:270:0.1);
\fill[white] (center)+ (0, -0.1) arc (270:450:0.1);
\draw[black] (center)+ (0, -0.1) arc (270:450:0.1);
\draw[black] (center)+ (0, 0.1) arc (90:270:0.1);

\coordinate (center) at (-2,1);
\fill[white] (center) + (0, 0.1) arc (90:270:0.1);
\fill[white] (center)+ (0, -0.1) arc (270:450:0.1);
\draw[black] (center)+ (0, -0.1) arc (270:450:0.1);
\draw[black] (center)+ (0, 0.1) arc (90:270:0.1);
\coordinate (center) at (-2,2);
\fill[white] (center) + (0, 0.1) arc (90:270:0.1);
\fill[white] (center)+ (0, -0.1) arc (270:450:0.1);
\draw[black] (center)+ (0, -0.1) arc (270:450:0.1);
\draw[black] (center)+ (0, 0.1) arc (90:270:0.1);
\coordinate (center) at (-2,3);
\fill[white] (center) + (0, 0.1) arc (90:270:0.1);
\fill[white] (center)+ (0, -0.1) arc (270:450:0.1);
\draw[black] (center)+ (0, -0.1) arc (270:450:0.1);
\draw[black] (center)+ (0, 0.1) arc (90:270:0.1);
\coordinate (center) at (-2,4);
\fill[white] (center) + (0, 0.1) arc (90:270:0.1);
\fill[white] (center)+ (0, -0.1) arc (270:450:0.1);
\draw[black] (center)+ (0, -0.1) arc (270:450:0.1);
\draw[black] (center)+ (0, 0.1) arc (90:270:0.1);
\coordinate (center) at (-2,5);
\fill[white] (center) + (0, 0.1) arc (90:270:0.1);
\fill[white] (center)+ (0, -0.1) arc (270:450:0.1);
\draw[black] (center)+ (0, -0.1) arc (270:450:0.1);
\draw[black] (center)+ (0, 0.1) arc (90:270:0.1);
\coordinate (center) at (-3,1);
\fill[white] (center) + (0, 0.1) arc (90:270:0.1);
\fill[white] (center)+ (0, -0.1) arc (270:450:0.1);
\draw[black] (center)+ (0, -0.1) arc (270:450:0.1);
\draw[black] (center)+ (0, 0.1) arc (90:270:0.1);
\coordinate (center) at (-3,2);
\fill[white] (center) + (0, 0.1) arc (90:270:0.1);
\fill[white] (center)+ (0, -0.1) arc (270:450:0.1);
\draw[black] (center)+ (0, -0.1) arc (270:450:0.1);
\draw[black] (center)+ (0, 0.1) arc (90:270:0.1);
\coordinate (center) at (-3,3);
\fill[white] (center) + (0, 0.1) arc (90:270:0.1);
\fill[white] (center)+ (0, -0.1) arc (270:450:0.1);
\draw[black] (center)+ (0, -0.1) arc (270:450:0.1);
\draw[black] (center)+ (0, 0.1) arc (90:270:0.1);
\coordinate (center) at (-3,4);
\fill[white] (center) + (0, 0.1) arc (90:270:0.1);
\fill[white] (center)+ (0, -0.1) arc (270:450:0.1);
\draw[black] (center)+ (0, -0.1) arc (270:450:0.1);
\draw[black] (center)+ (0, 0.1) arc (90:270:0.1);
\coordinate (center) at (-3,5);
\fill[white] (center) + (0, 0.1) arc (90:270:0.1);
\fill[white] (center)+ (0, -0.1) arc (270:450:0.1);
\draw[black] (center)+ (0, -0.1) arc (270:450:0.1);
\draw[black] (center)+ (0, 0.1) arc (90:270:0.1);
\coordinate (center) at (-4,1);
\fill[white] (center) + (0, 0.1) arc (90:270:0.1);
\fill[white] (center)+ (0, -0.1) arc (270:450:0.1);
\draw[black] (center)+ (0, -0.1) arc (270:450:0.1);
\draw[black] (center)+ (0, 0.1) arc (90:270:0.1);
\coordinate (center) at (-4,2);
\fill[white] (center) + (0, 0.1) arc (90:270:0.1);
\fill[white] (center)+ (0, -0.1) arc (270:450:0.1);
\draw[black] (center)+ (0, -0.1) arc (270:450:0.1);
\draw[black] (center)+ (0, 0.1) arc (90:270:0.1);
\coordinate (center) at (-4,3);
\fill[white] (center) + (0, 0.1) arc (90:270:0.1);
\fill[white] (center)+ (0, -0.1) arc (270:450:0.1);
\draw[black] (center)+ (0, -0.1) arc (270:450:0.1);
\draw[black] (center)+ (0, 0.1) arc (90:270:0.1);
\coordinate (center) at (-4,4);
\fill[white] (center) + (0, 0.1) arc (90:270:0.1);
\fill[white] (center)+ (0, -0.1) arc (270:450:0.1);
\draw[black] (center)+ (0, -0.1) arc (270:450:0.1);
\draw[black] (center)+ (0, 0.1) arc (90:270:0.1);
\coordinate (center) at (-4,5);
\fill[white] (center) + (0, 0.1) arc (90:270:0.1);
\fill[white] (center)+ (0, -0.1) arc (270:450:0.1);
\draw[black] (center)+ (0, -0.1) arc (270:450:0.1);
\draw[black] (center)+ (0, 0.1) arc (90:270:0.1);
\coordinate (center) at (-5,1);
\fill[white] (center) + (0, 0.1) arc (90:270:0.1);
\fill[white] (center)+ (0, -0.1) arc (270:450:0.1);
\draw[black] (center)+ (0, -0.1) arc (270:450:0.1);
\draw[black] (center)+ (0, 0.1) arc (90:270:0.1);
\coordinate (center) at (-5,2);
\fill[white] (center) + (0, 0.1) arc (90:270:0.1);
\fill[white] (center)+ (0, -0.1) arc (270:450:0.1);
\draw[black] (center)+ (0, -0.1) arc (270:450:0.1);
\draw[black] (center)+ (0, 0.1) arc (90:270:0.1);
\coordinate (center) at (-5,3);
\fill[white] (center) + (0, 0.1) arc (90:270:0.1);
\fill[white] (center)+ (0, -0.1) arc (270:450:0.1);
\draw[black] (center)+ (0, -0.1) arc (270:450:0.1);
\draw[black] (center)+ (0, 0.1) arc (90:270:0.1);
\coordinate (center) at (-5,4);
\fill[white] (center) + (0, 0.1) arc (90:270:0.1);
\fill[white] (center)+ (0, -0.1) arc (270:450:0.1);
\draw[black] (center)+ (0, -0.1) arc (270:450:0.1);
\draw[black] (center)+ (0, 0.1) arc (90:270:0.1);
\coordinate (center) at (-5,5);
\fill[white] (center) + (0, 0.1) arc (90:270:0.1);
\fill[white] (center)+ (0, -0.1) arc (270:450:0.1);
\draw[black] (center)+ (0, -0.1) arc (270:450:0.1);
\draw[black] (center)+ (0, 0.1) arc (90:270:0.1);
\coordinate (center) at (-6,1);
\fill[white] (center) + (0, 0.1) arc (90:270:0.1);
\fill[white] (center)+ (0, -0.1) arc (270:450:0.1);
\draw[black] (center)+ (0, -0.1) arc (270:450:0.1);
\draw[black] (center)+ (0, 0.1) arc (90:270:0.1);
\coordinate (center) at (-6,2);
\fill[white] (center) + (0, 0.1) arc (90:270:0.1);
\fill[white] (center)+ (0, -0.1) arc (270:450:0.1);
\draw[black] (center)+ (0, -0.1) arc (270:450:0.1);
\draw[black] (center)+ (0, 0.1) arc (90:270:0.1);
\coordinate (center) at (-6,3);
\fill[white] (center) + (0, 0.1) arc (90:270:0.1);
\fill[white] (center)+ (0, -0.1) arc (270:450:0.1);
\draw[black] (center)+ (0, -0.1) arc (270:450:0.1);
\draw[black] (center)+ (0, 0.1) arc (90:270:0.1);
\coordinate (center) at (-6,4);
\fill[white] (center) + (0, 0.1) arc (90:270:0.1);
\fill[white] (center)+ (0, -0.1) arc (270:450:0.1);
\draw[black] (center)+ (0, -0.1) arc (270:450:0.1);
\draw[black] (center)+ (0, 0.1) arc (90:270:0.1);
\coordinate (center) at (-6,5);
\fill[white] (center) + (0, 0.1) arc (90:270:0.1);
\fill[white] (center)+ (0, -0.1) arc (270:450:0.1);
\draw[black] (center)+ (0, -0.1) arc (270:450:0.1);
\draw[black] (center)+ (0, 0.1) arc (90:270:0.1);

\coordinate (center) at (0,-1);
\fill[white] (center) + (0, 0.1) arc (90:270:0.1);
\fill[white] (center)+ (0, -0.1) arc (270:450:0.1);
\draw[black] (center)+ (0, -0.1) arc (270:450:0.1);
\draw[black] (center)+ (0, 0.1) arc (90:270:0.1);
\coordinate (center) at (-1,0);
\fill[white] (center) + (0, 0.1) arc (90:270:0.1);
\fill[white] (center)+ (0, -0.1) arc (270:450:0.1);
\draw[black] (center)+ (0, -0.1) arc (270:450:0.1);
\draw[black] (center)+ (0, 0.1) arc (90:270:0.1);

\end{tikzpicture}
\end{figure}
We now recall from \cite{K4} the following Remark about conformal duality.
\begin{rem}
\label{cantacasellikacremarksize}
By the main result in \cite{cantacasellikac}, the conformal dual of a Verma module $M(m,n,\mu_{t},\mu_{C})$ is $M(m,n,-\mu_{t}+a,-\mu_{C}+b)$, with
\[
a=\text{str}(\text{ad}(t)_{|\g_{<0}})=2
\] 
and 
\[
b=\text{str}(\text{ad}(C)_{|\g_{<0}})=0,
\]
where $\g=\mathcal A(K'_4)$, 'str' denotes supertrace, and 'ad' denotes the adjoint representation. In particular in Figure \ref{figura} the duality is obtained with the rotation by 180 degrees of the whole picture.
\end{rem}
We introduce the following $\g_{0}-$modules:
\begin{align*}
V_{A}&=\C \left[x_{1},x_{2},y_{1},y_{2}\right], &&V_{B}=\C \left[\partial_{x_{1}},\partial_{x_{2}},y_{1},y_{2}\right]_{\left[1,-1\right]},\\
V_{C}&=\C \left[\partial_{x_{1}},\partial_{x_{2}},\partial_{y_{1}},\partial_{y_{2}}\right]_{\left[2,0\right]}, &&V_{D}=\C \left[x_{1},x_{2},\partial_{y_{1}},\partial_{y_{2}}\right]_{\left[1,1\right]}.
\end{align*}
The subscripts $[i,j]$ mean that $t$ acts on $V_{X}$, for $X=A,B,C,D$, as $-\frac{1}{2}(x_{1} \partial_{x_{1}}+x_{2}\partial_{x_{2}}+y_{1} \partial_{y_{1}}+y_{2}\partial_{y_{2}})$ plus $i \Id$  and $C$ acts on $V_{X}$, for $X=A,B,C,D$, as $\frac{1}{2}(x_{1} \partial_{x_{1}}+x_{2}\partial_{x_{2}}) -\frac{1}{2}(y_{1} \partial_{y_{1}}+y_{2}\partial_{y_{2}})$ plus $j \Id$; the subscript $[i,j]$ is assumed to be $[0,0]$ when it is omitted, i.e. for $X=A$.\\
The elements of $\g_{0}^{ss}$ act on $V_{X}$, for $X=A,B,C,D$, in the standard way: 
\begin{align*}
&x_{i}\partial_{x_{j}}x_{k}=\bigchi_{j=k}x_{i}, \quad &&x_{i}\partial_{x_{j}}.\partial_{x_{k}}=-\bigchi_{i=k}\partial_{x_{j}},\quad  &&x_{i}\partial_{x_{j}}y_{k}=0, \quad  &&x_{i}\partial_{x_{j}}.\partial_{y_{k}}=0; \\
&y_{i}\partial_{y_{j}}y_{k}=\bigchi_{j=k}y_{i}, \quad &&y_{i}\partial_{y_{j}}.\partial_{y_{k}}=-\bigchi_{i=k}\partial_{y_{j}},\quad  &&y_{i}\partial_{y_{j}}x_{k}=0, \quad  &&y_{i}\partial_{y_{j}}.\partial_{x_{k}}=0. 
\end{align*}
We introduce the following bigrading:
\begin{align}
\label{bigrading}
V_{X}^{m,n}:=\left\{f \in V_{X} \, : \, (x_{1} \partial_{x_{1}}+x_{2} \partial_{x_{2}}).f=mf \, \, \text{and} \, \, (y_{1} \partial_{y_{1}}+y_{2} \partial_{y_{2}}).f=nf \right\}.
\end{align}
The $V_{X}^{m,n}$'s are irreducible $\g_{0}-$modules. 
We point out that  for $m,n \in \Z_{\geq 0}$:
\begin{align}
\label{identificazioneVX}
V_{A}^{m,n} &\cong F \Big(m,n, -\frac{m+n}{2},\frac{m-n}{2} \Big),  &&V_{B}^{-m,n} \cong F \Big(m,n, 1+\frac{m-n}{2},-\frac{m+n}{2}-1 \Big),\\ \nonumber
V_{C}^{-m,-n} &\cong F \Big(m,n, \frac{m+n}{2}+2,\frac{n-m}{2} \Big), &&V_{D}^{m,-n} \cong F \Big(m,n, 1+\frac{n-m}{2},\frac{m+n}{2}+1 \Big).
\end{align}
Hence for $m,n \in \Z_{\geq 0}$, $V_{A}^{m,n}$ is the irreducible $\g_{0}-$module determined by coordinates $(m,n)$ in quadrant $\textbf{A}$ of Figure \ref{figura}, $V_{B}^{-m,n}$ is the irreducible $\g_{0}-$module determined by coordinates $(m,n)$ in quadrant $\textbf{B}$, $V_{C}^{-m,-n}$ is the irreducible $\g_{0}-$module determined by coordinates $(m,n)$ in quadrant $\textbf{C}$ and $V_{D}^{m,-n}$ is the irreducible $\g_{0}-$module determined by coordinates $(m,n)$ in quadrant $\textbf{D}$.
We have that $V_{X}=\oplus_{m,n}V_{X}^{m,n}$ is the direct sum of all the irreducible $\g_{0}-$modules in quadrant $\textbf{X}$. \\
We denote by $M_{X}^{m,n}=U(\g_{<0}) \otimes V_{X}^{m,n}$; we point out that, for $m,n \in \Z_{\geq 0}$, $M_{A}^{m,n}$ is the finite Verma module represented in Figure \ref{figura} in quadrant $\textbf{A}$ with coordinates $(m,n)$, $M_{B}^{-m,n}$ is the finite Verma module represented in quadrant $\textbf{B}$ with coordinates $(m,n)$, $M_{C}^{-m,-n}$ is the finite Verma module represented in quadrant $\textbf{C}$ with coordinates $(m,n)$, $M_{D}^{m,-n}$ is the finite Verma module represented in quadrant $\textbf{D}$ with coordinates $(m,n)$. We call $M_{X} =\oplus_{m,n \in \Z} M_{X}^{m,n}$ the direct sum of all finite Verma modules in the quadrant $\textbf{X}$ of Figure \ref{figura}.
We now recall the classification of highest weight singular vectors found in \cite{K4}, using the notation of the $V_{X}$'s for $X=A,B,C,D$ and \eqref{identificazioneVX}.
\begin{thm}[\cite{K4}]
		\label{sing1}
		Let $F$ be an irreducible finite$-$dimensional $\g_{0}-$module, with highest weight $\mu$. A vector $ \vec{m} \in \Ind (F) $ is a non trivial highest weight singular vector of degree 1 if and only if $\vec{m}$ is (up to a scalar) one of the following vectors:
		
		\begin{description}
			\item [a] $\mu=(m,n,-\frac{m+n}{2},\frac{m-n}{2}) $ with $m,n \in \Z_{\geq 0}$,
			$$\vec{m}_{1a}=w_{11}\otimes x_{1}^{m}y_{1}^{n};$$
			\item [b] $\mu=(m,n,1+\frac{m-n}{2},-1-\frac{m+n}{2})$, with $m  \in \Z_{>0}$, $n \in \Z_{\geq 0}$,
			$$\vec{m}_{1b}=w_{21}\otimes   \partial_{x_{2}}^{m}y_{1}^{n}+w_{11}\otimes  \partial_{x_{1}} \partial_{x_{2}}^{m-1}y_{1}^{n};$$
			\item [c] $\mu=(m,n,2+\frac{m+n}{2},\frac{n-m}{2})$, with $m,n \in \Z_{>0}$,
			$$\vec{m}_{1c}=w_{22}\otimes \partial_{x_{2}}^{m}\partial_{y_{2}}^{n}+w_{12}\otimes \partial_{x_{1}}\partial_{x_{2}}^{m-1}\partial_{y_{2}}^{n}+w_{21}\otimes \partial_{x_{2}}^{m}\partial_{y_{1}}\partial_{y_{2}}^{n-1}+w_{11}\otimes \partial_{x_{1}}\partial_{x_{2}}^{m-1}\partial_{y_{1}}\partial_{y_{2}}^{n-1};$$
			\item [d] $\mu=(m,n,1+\frac{n-m}{2},1+\frac{m+n}{2})$, with $m \in \Z_{\geq 0}$, $n \in \Z_{>0}$,
			$$\vec{m}_{1d}=w_{12}\otimes x_{1}^{m}\partial_{y_{2}}^{n}+w_{11} \otimes  x_{1}^{m} \partial_{y_{1}}\partial_{y_{2}}^{n-1}.$$
		 \end{description}
		\end{thm}
		\begin{thm}[\cite{K4}]
		\label{sing2}
		Let $F$ be an irreducible finite$-$dimensional $\g_{0}-$module, with highest weight $\mu$. A vector $\vec{m} \in \Ind (F) $ is a non trivial highest weight singular vector of degree 2 if and only if $\vec{m}$ is (up to a scalar) one of the following vectors:
		
		\begin{description}
			\item [a] $\mu=(0,n,1-\frac{n}{2},-1-\frac{n}{2}) $ with $n \in \Z_{\geq 0}$,
			$$\vec{m}_{2a}=w_{11}w_{21}\otimes y_{1}^{n};$$
			\item [b] $\mu=(m,0,1-\frac{m}{2},1+\frac{m}{2}) $ with $m \in \Z_{\geq 0}$,
			$$\vec{m}_{2b}=w_{11}w_{12}\otimes  x_{1}^{m};$$
			\item [c] $\mu=(m,0,2+\frac{m}{2},-\frac{m}{2}) $ with $m \in \Z_{> 1}$,
			$$\vec{m}_{2c}=w_{21}w_{22}\otimes \partial_{x_{2}}^{m}+(w_{11}w_{22}+w_{21}w_{12})\otimes \partial_{x_{1}}\partial_{x_{2}}^{m-1}+w_{11}w_{12}\otimes \partial_{x_{1}}^{2}\partial_{x_{2}}^{m-2};$$
			\item [d] $\mu=(0,n,2+\frac{n}{2},\frac{n}{2}) $ with $n \in \Z_{> 1}$,
			$$\vec{m}_{2d}=w_{12}w_{22}\otimes \partial_{y_{2}}^{n}+(w_{11}w_{22}+w_{12}w_{21})\otimes \partial_{y_{1}}\partial_{y_{2}}^{n-1}+w_{11}w_{21}\otimes \partial_{y_{1}}^{2}\partial_{y_{2}}^{n-2}.$$
		 \end{description}
		\end{thm}
\begin{thm}[\cite{K4}]
		\label{sing3}
		Let $F$ be an irreducible finite$-$dimensional $\g_{0}-$module, with highest weight $\mu$. A vector $\vec{m} \in \Ind (F) $ is a non trivial highest weight singular vector of degree 3 if and only if $\vec{m}$ is (up to a scalar) one of the following vectors:
		
		\begin{description}
			\item [a] $\mu=(1,0,\frac{5}{2},-\frac{1}{2})$,
			$$\vec{m}_{3a}=w_{11}w_{22}w_{21}\otimes \partial_{x_{2}}-w_{21}w_{12}w_{11}\otimes \partial_{x_{1}};$$
			\item [b] $\mu=(0,1,\frac{5}{2},\frac{1}{2})$,
			$$\vec{m}_{3b}=w_{11}w_{22}w_{12}\otimes \partial_{y_{2}}-w_{12}w_{21}w_{11}\otimes \partial_{y_{1}}.$$
		 \end{description}
		\end{thm}
		\begin{thm}[\cite{K4}]
		\label{greaterthan3}
		There are no singular vectors of degree greater than 3. 
		\end{thm}
		\begin{rem}
		\label{cambio notazione}
		We point out that the highest weight singular vectors of Theorems \ref{sing1}, \ref{sing2} and \ref{sing1} are written differently from \cite{K4}. Indeed in \cite{K4} the irreducible $\g_{0}^{ss}-$module of highest weight $(m,n)$ with respect to $h_{x},h_{y}$ is identified  with the space of bihomogeneous polynomials in the four variables $x_1,x_2,y_1,y_2$ of degree $m$ in the variables $x_{1},x_{2}$, and of degree $n$ in the variables $y_{1},y_{2}$. We use instead the notation of the $V_{X}$'s and \eqref{identificazioneVX} because it is convenient for the explicit description of the morphisms in Figure \ref{figura}. 
		\end{rem}
From Theorems \ref{sing1}, \ref{sing2}, \ref{sing3} and \ref{greaterthan3} it follows that the module $M(0,0,2,0)$ does not contain non trivial singular vectors, hence it is irreducible due to Theorem \ref{keythmsingular}.
\begin{prop}[\cite{K4}]
\label{M(0,0,2,0)}
The module $M(0,0,2,0)$ is irreducible and it is isomorphic to the coadjoint representation of $K(1,4)_{+}$, where we consider the restricted dual, i.e. $K(1,4)_{+}^{*}=\oplus_{j \in \Z}({K(1,4)_{+}}_{j})^{*}$.
\end{prop}

\section{The morphisms}
\label{sezmorfismi}
In this section we find an explicit form for the morphisms that occur in Figure \ref{figura}. We follow the notation in \cite{kacrudakovE36} and define, for every $u \in U(\g_{<0}) $ and $\phi \in \Hom(V_{X},V_{Y})$, the map $u \otimes \phi :  M_{X}\longrightarrow M_{Y}$ by:
\begin{align}
\label{uuprimo}
(u \otimes \phi)(u' \otimes v)= u' \,u \otimes \phi(v),
\end{align}
for every $ u' \otimes v \in U(\g_{<0}) \otimes V_{X}$. From this definition it is clear that the map $u \otimes \phi$ commutes with the action of $\g_{<0}$. The following is straightforward.
\begin{lem}
\label{commutacong0basiduali}
Let $u \otimes \phi$ be a map as in \eqref{uuprimo}. Let us suppose that $u \otimes \phi=\sum_{i}u_{i} \otimes \phi_{i}$ where $\left\{u_{i}\right\}_{i}$ and $\left\{\phi_{i}\right\}_{i}$ are bases of dual $\g_{0}-$modules and $u_{i}$ is the dual of $\phi_{i}$ for all $i$. Then $u \otimes \phi$ commutes with $\g_{0}$. 
\end{lem}

\begin{lem}
\label{g0commuta}
Let us consider a map $u \otimes \phi \in U(\g_{<0}) \otimes \Hom(V_{X},V_{Y}) $. In order to show that $u \otimes \phi$ commutes with $\g_{0}$, it is sufficient to show that $w u \otimes \phi(v)=u \otimes \phi(w.v)$ for all $v \in V_{X}$, $w \in \g_{0}$.
\end{lem}
\begin{proof}
Let $w \in \g_{0}$. For every $u_{i_{1}}u_{i_{2}} \dots u_{i_{k}} \otimes v \in U(\g_{<0}) \otimes V_{X}$: 
\begin{align*}
w.(u_{i_{1}}u_{i_{2}} \dots u_{i_{k}} \otimes v)=u_{i_{1}}u_{i_{2}} \dots u_{i_{k}} \otimes w.v+\sum \widetilde{u}_{i_{1}}\widetilde{u}_{i_{2}} \dots \widetilde{u}_{i_{k}} \otimes v.
\end{align*}
Hence for a map $u \otimes \phi \in U(\g_{<0}) \otimes  \Hom(V_{X},V_{Y})$:
\begin{align*}
(u \otimes \phi)(w.(u_{i_{1}}u_{i_{2}} \dots u_{i_{k}} \otimes v))=u_{i_{1}}u_{i_{2}} \dots u_{i_{k}}u \otimes \phi(w.v)+\sum \widetilde{u}_{i_{1}}\widetilde{u}_{i_{2}} \dots \widetilde{u}_{i_{k}} u \otimes \phi( v).
\end{align*}
On the other hand we have:
\begin{align*}
w.(u \otimes \phi)(u_{i_{1}}u_{i_{2}} \dots u_{i_{k}} \otimes v)&=w.(u_{i_{1}}u_{i_{2}} \dots u_{i_{k}} u\otimes \phi( v))\\
&=u_{i_{1}}u_{i_{2}} \dots u_{i_{k}}w u \otimes \phi(v)+\sum \widetilde{u}_{i_{1}}\widetilde{u}_{i_{2}} \dots \widetilde{u}_{i_{k}} u \otimes \phi( v).
\end{align*}
Therefore, in order to show that $u \otimes \phi$ commutes with $\g_{0}$, it is sufficient to show that $w u \otimes \phi(v)=u \otimes \phi(w.v)$ for all $v \in V_{X}$, $w \in \g_{0}$.
\end{proof}
\begin{lem}
\label{nablag0g+}
Let $\Phi:M_{X}\rightarrow M_{Y}$ be a linear map. Let us suppose that $\Phi$ commutes with $\g_{\leq 0}$ and that $\Phi(v)$ is a singular vector for every $v$ highest weight vector in $V_{X}^{m,n}$ and for all $m,n \in \Z$. Then $\Phi$ is a morphism of $\g-$modules.
\end{lem}
\begin{proof}
Due to Lemma 2.3 in \cite{kacrudakovE36}, it is sufficient to show that $\g_{>0}\Phi(w)=0$ for every $w\in V_{X}$, in order to prove that $\Phi$ commutes with $\g_{>0}$. \\
We know that $\g_{>0}\Phi(v)=0$ for every $v$ highest weight vector in $V_{X}^{m,n}$ for all $m,n \in \Z$. Let $v$ be the highest weight vector in $V_{X}^{m,n}$, $f$ one among $f_{x},f_{y}$ and $g_{+}\in \g_{>0}$.
We have that:
\begin{align*}
g_{+}. \Phi(f.v)=g_{+}.(f. \Phi(v))=f.(g_{+}. \Phi(v))+[g_{+},f]. \Phi(v)=0.  
\end{align*}
This can be iterated and we obtain that $\g_{>0}.\Phi(w)=0$ for all $ w \in V_{X}^{m,n}$. Hence $\g_{>0}.\Phi(w)=0$ for all $ w \in V_{X}$.
\end{proof}

We consider, for $j=1,2$, the map $\partial_{x_{j}}: V_{X}\longrightarrow V_{X}$ that is the derivation by $x_{j}$ for $X=A,D$ and the multiplication by $\partial_{x_{j}}$ for $X=B,C$. We define analogously, for $j=1,2$, the map $\partial_{y_{j}}: V_{X}\longrightarrow V_{X}$, that is the derivation by $y_{j}$ for $X=A,B$ and the multiplication by $\partial_{y_{j}}$ for $X=C,D$. We will often write, by abuse of notation, $\partial_{x_{j}}$ instead of $1\otimes \partial_{x_{j}}: M_{X}\longrightarrow M_{X}$.\\
We define the maps $\Delta^{+}:M_{X} \longrightarrow M_{X}$, $\Delta^{-}:M_{X} \longrightarrow M_{X}$, $\nabla:M_{X} \longrightarrow M_{X}$ as follows:
\begin{gather}
\Delta^{+}=w_{11}\otimes \partial_{x_{1}}+w_{21}\otimes \partial_{x_{2}}, \label{nabla+-} \quad
\Delta^{-}=w_{12}\otimes \partial_{x_{1}}+w_{22}\otimes \partial_{x_{2}}, \\
\nabla= \Delta^{+}\partial_{y_{1}}+\Delta^{-}\partial_{y_{2}}=w_{11}\otimes \partial_{x_{1}}\partial_{y_{1}}+w_{21}\otimes \partial_{x_{2}}\partial_{y_{1}}+   w_{12}\otimes \partial_{x_{1}}\partial_{y_{2}}+w_{22}\otimes \partial_{x_{2}}\partial_{y_{2}}. \label{nabla}
\end{gather}
We point out that $\nabla_{|M_{X}^{m,n} }: M_{X}^{m,n} \longrightarrow M_{X}^{m-1,n-1}$ for $X=A,B,C,D$; by abuse of notation we will write $\nabla$ instead of $\nabla_{|M_{X}^{m,n} }$. 
\begin{rem}
\label{nabla+nabla-}
By \eqref{bracketwcaselli} and \eqref{quadratideiwcaselli} it is straightforward that $(\Delta^{+})^{2}=0$, $(\Delta^{-})^{2}=0$ and $\Delta^{+}\Delta^{-}+\Delta^{-}\Delta^{+}=0$.
\end{rem}
\begin{prop}
\label{osser}
 The map $\nabla$ is the explicit expression of the $\g-$morphisms of degree 1 in Figure \ref{figura} and $\nabla^{2}=0$.
\end{prop}
\begin{proof}
It is a straightforward verification that $\nabla: M_{X}^{m,n} \longrightarrow M_{X}^{m-1,n-1}$ is constructed so that $\nabla(v)$, for $v$ highest weight vector in $V_{X}^{m,n}$, is the highest weight singular vector of degree 1 in $M_{X}^{m-1,n-1}$, classified in Theorem \ref{sing1}. Indeed for $m,n \geq 0$:
\begin{itemize}
	\item [\textbf{a:}] let $\nabla: M_{A}^{m,n} \longrightarrow M_{A}^{m-1,n-1}$. The highest weight vector in $V_{A}^{m,n}$ is $x_{1}^{m} y_{1}^{n}$. By direct computation, $\nabla(x_{1}^{m}y_{1}^{n})=mn \, \vec{m}_{1a}$, where $\vec{m}_{1a}$ is the highest weight singular vector of $M(m-1,n-1, -\frac{m+n-2}{2},\frac{m-n}{2} )$ found in Theorem \ref{sing1}.
		\item [\textbf{b:}] Let $\nabla: M_{B}^{-m,n} \longrightarrow M_{B}^{-m-1,n-1}$. The highest weight vector in $V_{B}^{-m,n}$ is $\partial_{x_{2}}^{m} y_{1}^{n}$. By direct computation, $\nabla(\partial_{x_{2}}^{m}y_{1}^{n})=n \, \vec{m}_{1b}$, where $\vec{m}_{1b}$ is the highest weight singular vector of $M(m+1,n-1, 1+\frac{m-n+2}{2},-\frac{m+n}{2}-1 )$ found in Theorem \ref{sing1}.
		\item [\textbf{c:}] Let $\nabla: M_{C}^{-m,-n} \longrightarrow M_{C}^{-m-1,-n-1}$. The highest weight vector in $V_{C}^{-m,-n}$ is $\partial_{x_{2}}^{m} \partial_{y_{2}}^{n}$. By direct computation, $\nabla(\partial_{x_{2}}^{m} \partial_{y_{2}}^{n})=\vec{m}_{1c}$, where $\vec{m}_{1c}$ is the highest weight singular vector of $M(m+1,n+1, \frac{m+n+2}{2}+2,\frac{n-m}{2} )$ found in Theorem \ref{sing1}.
		%
		\item [\textbf{d:}] Let $\nabla: M_{D}^{m,-n} \longrightarrow M_{D}^{m-1,-n-1}$. The highest weight vector in $V_{D}^{m,-n}$ is $x_{1}^{m} \partial_{y_{2}}^{n}$. 
		By direct computation, $\nabla(x_{1}^{m} \partial_{y_{2}}^{n})=m \, \vec{m}_{1d}$, where $\vec{m}_{1d}$ is the highest weight singular vector of $M(m-1,n+1, 1+\frac{n-m+2}{2},\frac{m+n}{2}+1  )$ found in Theorem \ref{sing1}.
\end{itemize}
The map $\nabla: M_{X} \rightarrow M_{X}$ commutes with $\g_{<0}$ by \eqref{uuprimo}.
By Lemmas \ref{commutacong0basiduali}, \ref{nablag0g+} it follows that $\nabla$ is a morphism of $\g-$modules. The property $\nabla^{2}=0$ follows from the fact that $\nabla$ is a map between Verma modules that contain only highest weight singular vectors of degree 1, by Theorems \ref{sing1}, \ref{sing2}, \ref{sing3}.
\end{proof}
By Proposition \ref{osser}, it follows that for all $m,n \in\Z_{\geq 0}$:
\begin{enumerate}
	\item[i:] the maps $\nabla: M_{A}^{m,n} \longrightarrow M_{A}^{m-1,n-1}$ are the morphisms represented in Figure \ref{figura} in quadrant \textbf{A};
	\item[ii:] the maps $\nabla: M_{B}^{-m,n} \longrightarrow M_{B}^{-m-1,n-1}$ are the morphisms  represented in Figure \ref{figura} in quadrant \textbf{B};
	\item[iii:] the maps $\nabla: M_{C}^{-m,-n} \longrightarrow M_{C}^{-m-1,-n-1}$ are the morphisms  represented in Figure \ref{figura} in quadrant \textbf{C};
	\item[iv:] the maps $\nabla: M_{D}^{m,-n} \longrightarrow M_{D}^{m-1,-n-1}$ are the morphisms represented in Figure \ref{figura} in quadrant \textbf{D}.
\end{enumerate}
We introduce the following notation:
\begin{align*}
V_{A^{'}}&=\oplus_{m \in \Z}V^{m,0}_{A}=\C \left[x_{1},x_{2} \right], &&V_{B^{'}}=\oplus_{m \in \Z}V^{m,0}_{B}=\C \left[\partial_{x_{1}},\partial_{x_{2}}\right]_{\left[1,-1\right]},\\
V_{C^{'}}&=\oplus_{m \in \Z}V^{m,0}_{C}=\C \left[\partial_{x_{1}},\partial_{x_{2}} \right]_{\left[2,0\right]}, &&V_{D^{'}}=\oplus_{m\in \Z}V^{m,0}_{D}=\C \left[x_{1},x_{2}\right]_{\left[1,1\right]}.
\end{align*}
We denote $M_{X^{'}}=  U(\g_{<0}) \otimes V_{X^{'}}$. We point out that $M_{X^{'}}$ is the direct sum of Verma modules of Figure \ref{figura} in quadrant $\textbf{X}$ that lie on the axis $n=0$.
We consider the map $\tau_{1}:M_{A^{'}}\longrightarrow M_{D^{'}}$ that is the identity. We have that:
\begin{align}
\label{tau1}
&[t, \tau_{1}]=\tau_{1}, \quad[C, \tau_{1}]=\tau_{1}.
\end{align}
We call $\nabla_{2}:M_{A^{'}}\longrightarrow M_{D^{'}}$ the map $$\Delta^{-}\Delta^{+}\tau_{1}=w_{11}w_{12}\otimes \partial_{x_{1}}^{2}+w_{11}w_{22}\otimes \partial_{x_{1}}\partial_{x_{2}}+w_{21}w_{12}\otimes \partial_{x_{1}}\partial_{x_{2}}+w_{21}w_{22}\otimes \partial_{x_{2}}^{2}.$$
We consider the map $\tau_{2}:M_{B^{'}}\longrightarrow M_{C^{'}}$ that is the identity. We have that:
\begin{align}
\label{tau2}
&[t, \tau_{2}]=\tau_{2},\quad [C, \tau_{2}]=\tau_{2}.
\end{align}
By abuse of notation, we also call $\nabla_{2}:M_{B^{'}}\longrightarrow M_{C^{'}}$ the map $$\Delta^{-}\Delta^{+}\tau_{2}=w_{11}w_{12}\otimes \partial_{x_{1}}^{2}+w_{11}w_{22}\otimes \partial_{x_{1}}\partial_{x_{2}}+w_{21}w_{12}\otimes \partial_{x_{1}}\partial_{x_{2}}+w_{21}w_{22}\otimes \partial_{x_{2}}^{2}.$$
We observe that $M_{X^{'}}=\oplus_{m \in \Z} M_{X}^{m,0}$ for $X=A,B,C,D$. We will denote $M_{X^{'}}^{m} =M_{X}^{m,0}$.\\
We point out that ${\nabla_{2}}_{|M_{A^{'}}^{m}}:M_{A^{'}}^{m}\longrightarrow M_{D^{'}}^{m-2}$ for every $m\geq 2$ and ${\nabla_{2}}_{|M_{B^{'}}^{-m}}:M_{B^{'}}^{-m}\longrightarrow M_{C^{'}}^{-m-2}$ for every $m\geq 0$. By abuse of notation we will also write $\nabla_{2}$ instead of ${\nabla_{2}}_{|M_{A^{'}}^{m}}$ and ${\nabla_{2}}_{|M_{B^{'}}^{-m}}$.
\begin{prop}
\label{morfismonabla2}
The map $\nabla_{2}$ is the explicit expression of the morphisms of degree 2 in Figure \ref{figura} from the quadrant $\textbf{A}$ to the quadrant $\textbf{D}$ and from the quadrant $\textbf{B}$ to the quadrant $\textbf{C}$; $\nabla_{2}\nabla=\nabla \nabla_{2}=0$.
\end{prop}
\begin{proof}
We first point out that:
\begin{itemize}
	\item[i:] for $m\geq 2$, the map $\nabla_{2}:M_{A^{'}}^{m}\longrightarrow M_{D^{'}}^{m-2}$ is constructed so that $\nabla_{2}(v)$, for $v$ highest weight vector in $V_{A^{'}}^{m}$, is  the highest weight singular vector of degree 2 in $M_{D^{'}}^{m-2}$, classified in Theorem \ref{sing2}. Indeed, the highest weight vector in $V_{A^{'}}^{m}$ is $x_{1}^{m}$ and, by direct computation, $\nabla_{2}(x_{1}^{m})= m(m-1) \,\vec{m}_{2b}$, where $\vec{m}_{2b}$ is the highest weight singular vector of $M(m-2,0, 1-\frac{m-2}{2},1+\frac{m-2}{2} )$ found in Theorem \ref{sing2}.
	\item[ii:] For $m\geq 0$, the map $\nabla_{2}:M_{B^{'}}^{-m}\longrightarrow M_{C^{'}}^{-m-2}$ is constructed so that $\nabla_{2}(v)$, for $v$ highest weight vector in $V_{B^{'}}^{-m}$, is  the highest weight singular vector of degree 2 in $M_{C^{'}}^{-m-2}$, classified in Theorem \ref{sing2}. Indeed, the highest weight vector in $V_{B^{'}}^{-m}$ is $\partial_{x_{2}}^{m}$ and, by direct computation, $\nabla_{2}(\partial_{x_{2}}^{m})=\vec{m}_{2c}$, where $\vec{m}_{2c}$ is the  highest weight singular vector of $M(m+2,0, 2+\frac{m+2}{2},-\frac{m+2}{2} )$ found in Theorem \ref{sing2}.
		\end{itemize}
 The map $\nabla_{2}$ commutes with $\g_{<0}$ by \eqref{uuprimo}.
	By Lemmas \ref{commutacong0basiduali}, \ref{nablag0g+} it follows that $\nabla_{2}$ is a morphism of $\g-$modules. Finally,  $\nabla_{2}\nabla=\nabla \nabla_{2}=0$ follows from the fact that due to Theorem \ref{sing3}, there are no highest weight singular vectors of degree 3 in the codomain of $\nabla_{2}\nabla$ and $\nabla \nabla_{2}$.
\end{proof}
By Proposition \ref{morfismonabla2}, it follows that, for every $m\geq 2$, the maps $\nabla_{2}:M_{A^{'}}^{m}\longrightarrow M_{D^{'}}^{m-2}$ are the morphisms represented in Figure \ref{figura} from the quadrant \textbf{A} to the quadrant \textbf{D} and, for every $m\geq 0$, the maps $\nabla_{2}:M_{B^{'}}^{-m}\longrightarrow M_{C^{'}}^{-m-2}$ are the morphisms from the quadrant \textbf{B} to the quadrant \textbf{C}.\\
We now define the map $\tau_{3}:V_{A}^{0,0}\longrightarrow V_{C}^{0,0}$ that is the identity. We have that:
\begin{align}
\label{tau3}
&[t,\tau_{3}]=2\tau_{3}, \quad [C, \tau_{3}]=0.
\end{align}
We define the map $\nabla_{3}:M^{0,1}_{A}\longrightarrow M^{-1,0}_{C}$ as follows, using definition \eqref{uuprimo}, for every $m \in M^{0,1}_{A}$:
\begin{align*}
\nabla_{3}(m)=\Delta^{-}   \circ (w_{11}w_{21}\otimes \tau_{3})\circ(1 \otimes  \partial_{y_{1}}) (m)
+\Delta^{-}   \circ ((w_{12}w_{21}+w_{11}w_{22})\otimes \tau_{3})\circ (1 \otimes  \partial_{y_{2}})(m).
\end{align*}
\begin{prop}
\label{morfismonabla3}
The map $\nabla_{3}:M^{0,1}_{A}\longrightarrow M^{-1,0}_{C}$ is the explicit form for the morphism of $\g-$modules of degree 3 from quadrant \textbf{A} to quadrant \textbf{C} and $\nabla_{3}\nabla=\nabla \nabla_{3}=0$.
\end{prop}
\begin{proof}
The map $\nabla_{3}:M^{0,1}_{A}\longrightarrow M^{-1,0}_{C}$ is constructed so that $\nabla_{3}(v)$, for $v$ highest weight vector in $V_{A}^{0,1}$, is  the highest weight singular vector of degree 3 in $M_{C}^{-1,0}$, classified in Theorem \ref{sing3}. Indeed, the highest weight vector in $V_{A}^{0,1}$ is $y_{1}$ and, by direct computation, $\nabla_{3}(y_{1})=-\vec{m}_{3a}$, that is the highest weight singular vector of $M(1,0,\frac{5}{2},-\frac{1}{2} )$ found in Theorem \ref{sing3}. $\nabla_{3}$ commutes with $\g_{<0}$ due to \eqref{uuprimo}.
 By a straightforward computation and \eqref{tau3}, $w u \otimes \nabla_{3}(v)=u \otimes \nabla_{3}(w.v)$ for every $w \in \g_{0}$, $v \in V_{A}^{0,1}$. Therefore it is a morphism of $\g-$modules due to Lemmas \ref{g0commuta} and \ref{nablag0g+}. Finally $\nabla_{3}\nabla=\nabla \nabla_{3}=0$ since there are no singular vectors of degree 4 due to Theorem \ref{greaterthan3}.
\end{proof}
Let us define the maps $\widetilde{\Delta}^{+}:M_{X}\longrightarrow M_{X}$ and $\widetilde{\Delta}^{-}:M_{X}\longrightarrow M_{X}$ as follows:
\begin{align}
\label{widetildenabla+widetildenabla-}
\widetilde{\Delta}^{+}&=w_{11}\otimes \partial_{y_{1}}+w_{12}\otimes \partial_{y_{2}}, \quad \widetilde{\Delta}^{-}=w_{21}\otimes \partial_{y_{1}}+w_{22}\otimes \partial_{y_{2}}.
\end{align}
We point out that the morphism $\nabla$, defined in \eqref{nabla}, can be expressed also by $\nabla= \widetilde{\Delta}^{+}\partial_{x_{1}}+\widetilde{\Delta}^{-}\partial_{x_{2}}$.
\begin{rem}
\label{widetildenabla+widetildenabla-}
By \eqref{bracketwcaselli} and \eqref{quadratideiwcaselli} it is straightforward that $(\widetilde{\Delta}^{+})^{2}=0$, $(\widetilde{\Delta}^{-})^{2}=0$ and $\widetilde{\Delta}^{+}\widetilde{\Delta}^{-}+\widetilde{\Delta}^{-}\widetilde{\Delta}^{+}=0$.
\end{rem}
We introduce the following notation:
\begin{align*}
V_{A^{''}}&=\oplus_{n \in \Z} V_{A}^{0,n}=\C \left[y_{1},y_{2} \right],&&V_{B^{''}}=\oplus_{n \in \Z} V_{B}^{0,n}=\C \left[y_{1},y_{2} \right]_{\left[1,-1\right]},\\
V_{C^{''}}&=\oplus_{n \in \Z} V_{C}^{0,n}=\C \left[\partial_{y_{1}},\partial_{y_{2}} \right]_{\left[2,0\right]}, &&V_{D^{''}}=\oplus_{n \in \Z} V_{D}^{0,n}=\C \left[\partial_{y_{1}},\partial_{y_{2}} \right]_{\left[1,1\right]}.
\end{align*}
We denote $M_{X^{''}}=  U(\g_{<0}) \otimes V_{X^{''}}$. We point out that $M_{X^{''}}$ is the direct sum of Verma modules of Figure \ref{figura} in quadrant $\textbf{X}$ that lie on the axis $m=0$.
We consider the map $\widetilde{\tau}_{1}:M_{A^{''}}\longrightarrow M_{B^{''}}$ that is the identity. We have that:
\begin{align}
\label{widetildetau1}
&[t, \widetilde{\tau}_{1}]=\widetilde{\tau}_{1},\quad [C, \widetilde{\tau}_{1}]=-\widetilde{\tau}_{1}.
\end{align}
We call $\widetilde{\nabla}_{2}:M_{A^{''}}\longrightarrow M_{B^{''}}$ the map $$\widetilde{\Delta}^{-}\widetilde{\Delta}^{+}\widetilde{\tau}_{1}=w_{11}w_{21}\otimes\partial_{y_{1}}^{2}+w_{12}w_{21}\otimes\partial_{y_{1}}\partial_{y_{2}}+w_{11}w_{22}\otimes\partial_{y_{1}}\partial_{y_{2}}+w_{12}w_{22}\otimes\partial_{y_{2}}^{2}.$$
We consider the map $\widetilde{\tau}_{2}:M_{D^{''}}\longrightarrow M_{C^{''}}$ that is the identity. We have that:
\begin{align}
\label{widetildetau2}
&[t, \widetilde{\tau}_{2}]=\widetilde{\tau}_{2},\quad [C, \widetilde{\tau}_{2}]=-\widetilde{\tau}_{2}.
\end{align}
By abuse of notation, we also call $\widetilde{\nabla}_{2}:M_{D^{''}}\longrightarrow M_{C^{''}}$ the map $$\widetilde{\Delta}^{-}\widetilde{\Delta}^{+}\widetilde{\tau}_{2}=w_{11}w_{21}\otimes\partial_{y_{1}}^{2}+w_{12}w_{21}\otimes\partial_{y_{1}}\partial_{y_{2}}+w_{11}w_{22}\otimes\partial_{y_{1}}\partial_{y_{2}}+w_{12}w_{22}\otimes\partial_{y_{2}}^{2}.$$
We point out that $\widetilde{\nabla}_{2} \,_{|M_{A^{''}}^{n}}:M_{A^{''}}^{n}\longrightarrow M_{B^{''}}^{n-2}$ for every $n\geq 2$ and $\widetilde{\nabla}_{2} \,_{|M_{D^{''}}^{-n}}:M_{D^{''}}^{-n}\longrightarrow M_{C^{''}}^{-n-2}$ for every $n\geq 0$. By abuse of notation we will also write $\widetilde{\nabla}_{2}$ instead of $\widetilde{\nabla}_{2} \,_{|M_{A^{''}}^{n}}$ and $\widetilde{\nabla}_{2} \,_{|M_{D^{''}}^{-n}}$.
\begin{prop}
\label{morfismowidetildenabla2}
The map $\widetilde{\nabla}_{2}$ is the explicit expression of the morphisms of $\g-$modules of degree 2 in Figure \ref{figura} from the quadrant $\textbf{A}$ to the quadrant $\textbf{B}$ and from the quadrant $\textbf{D}$ to the quadrant $\textbf{C}$ and $\widetilde{\nabla}_{2}\nabla=\nabla \widetilde{\nabla}_{2}=0$.
\end{prop}
\begin{proof}
We first point out that:
\begin{itemize}
	\item[i:] for $n\geq 2$, the map $\widetilde{\nabla}_{2}:M_{A^{''}}^{n}\longrightarrow M_{B^{''}}^{n-2}$ is constructed so that $\widetilde{\nabla}_{2}(v)$, for $v$ highest weight vector in $V_{A^{''}}^{n}$, is  the highest weight singular vector of degree 2 in $M_{B^{''}}^{n-2}$, classified in Theorem \ref{sing2}. Indeed, the highest weight vector in $V_{A^{''}}^{n}$ is $y_{1}^{n}$ and, by direct computation, $\widetilde{\nabla}_{2}(y_{1}^{n})= n(n-1) \, \vec{m}_{2a}$, where $\vec{m}_{2a}$ is the highest weight singular vector of $M(0,n-2, 1-\frac{n-2}{2},-1-\frac{n-2}{2} )$ found in Theorem \ref{sing2}.
	\item[ii:] For $n\geq 0$, the map $\widetilde{\nabla}_{2}:M_{D^{''}}^{-n}\longrightarrow M_{C^{''}}^{-n-2}$ is constructed so that $\widetilde{\nabla}_{2}(v)$, for $v$ highest weight vector in $V_{D^{''}}^{-n}$, is  the highest weight singular vector of degree 2 in $M_{C^{''}}^{-n-2}$, classified in Theorem \ref{sing2}. Indeed, the highest weight vector in $V_{D^{''}}^{-n}$ is $\partial_{y_{2}}^{n}$ and, by direct computation, $\widetilde{\nabla}_{2}(\partial_{y_{2}}^{n})=\vec{m}_{2d}$, where $\vec{m}_{2d}$ is the  highest weight singular vector of $M(0,n+2, 2+\frac{n+2}{2},-\frac{n+2}{2} )$ found in Theorem \ref{sing2}.
		\end{itemize}
 The map $\widetilde{\nabla}_{2}$ commutes with $\g_{<0}$ by \eqref{uuprimo}.
	By Lemmas \ref{commutacong0basiduali}, \ref{nablag0g+} it follows that $\widetilde{\nabla}_{2}$ is a morphism of $\g-$modules. Finally,  $\widetilde{\nabla}_{2}\nabla=\nabla \widetilde{\nabla}_{2}=0$ follows from the fact that due to Theorem \ref{sing3}, there are no highest weight singular vectors of degree 3 in the codomain of $\widetilde{\nabla}_{2}\nabla$ and $\nabla \widetilde{\nabla}_{2}$.
\end{proof}
By Proposition \ref{morfismowidetildenabla2}, it follows that, for every for every $n\geq 2$, the maps $\widetilde{\nabla}_{2}:M_{A^{''}}^{n}\longrightarrow M_{B^{''}}^{n-2}$ are the morphisms represented in Figure \ref{figura} from the quadrant \textbf{A} to the quadrant \textbf{B} and, for every $n\geq 0$, the maps $\widetilde{\nabla}_{2} \,_{|M_{D^{''}}^{-n}}:M_{D^{''}}^{-n}\longrightarrow M_{C^{''}}^{-n-2}$ are the morphisms from the quadrant \textbf{D} to the quadrant \textbf{C}.\\
We define the map $\widetilde{\nabla}_{3}:M^{1,0}_{A}\longrightarrow M^{0,-1}_{C}$ as follows, using definition \eqref{uuprimo}, for every $m \in M^{1,0}_{A}$:
\begin{align*}
\widetilde{\nabla}_{3}(m)=\widetilde{\Delta}^{-}   \circ (w_{11}w_{12}\otimes \tau_{3})\circ(1 \otimes  \partial_{x_{1}}) (m)
+\widetilde{\Delta}^{-}     \circ ((w_{21}w_{12}+w_{11}w_{22})\otimes \tau_{3})\circ (1 \otimes  \partial_{x_{2}})(m).
\end{align*}
\begin{prop}
\label{morfismonabla3}
The map $\widetilde{\nabla}_{3}:M^{1,0}_{A}\longrightarrow M^{0,-1}_{C}$ is the explicit form of the morphism of $\g-$modules of degree 3 represented in figure \ref{figura} from quadrant \textbf{A} to quadrant \textbf{C} and $\widetilde{\nabla}_{3}\nabla=\nabla \widetilde{\nabla}_{3}=0$.
\end{prop}
\begin{proof}
The map $\widetilde{\nabla}_{3}:M^{1,0}_{A}\longrightarrow M^{0,-1}_{C}$ is constructed so that $\widetilde{\nabla}_{3}(v)$, for $v$ highest weight vector in $V_{A}^{1,0}$, is  the highest weight singular vector of degree 3 in $M_{C}^{0,-1}$, classified in Theorem \ref{sing3}. Indeed, the highest weight vector in $V_{A}^{1,0}$ is $x_{1}$ and, by direct computation, $\widetilde{\nabla}_{3}(x_{1})=-\vec{m}_{3b}$, that is the highest weight singular vector of $M(0,1,\frac{5}{2},\frac{1}{2} )$ found in Theorem \ref{sing3}. $\widetilde{\nabla}_{3}$ commutes with $\g_{<0}$ due to \eqref{uuprimo}. By a straightforward computation and \eqref{tau3}, $w u \otimes \widetilde{\nabla}_{3}(v)=u \otimes \widetilde{\nabla}_{3}(w.v)$ for every $w \in \g_{0}$, $v \in V_{A}^{1,0}$. Therefore it is a morphism of $\g-$modules by Lemmas \ref{g0commuta} and \ref{nablag0g+}. Finally $\widetilde{\nabla}_{3}\nabla=\nabla \widetilde{\nabla}_{3}=0$ since there are no singular vectors of degree 4 due to Theorem \ref{greaterthan3}.
\end{proof}
The following sections are dedicated to the computation of the homology of complexes in Figure \ref{figura}. We will use techniques of spectral sequences, that we briefly recall for the reader's convenience in the next section.
\section{Preliminaries on spectral sequences}
In this section we recall some notions about spectral sequences; for further details see \cite[Appendix]{kacrudakovE36} and \cite[Chapter XI]{maclane}. We follow the notation used in \cite{kacrudakovE36}.\\
Let $A$ be a module with a filtration:
\begin{align}
\label{filtration}
...\subset F_{p-1}A \subset F_{p}A \subset F_{p+1}A\subset...  ,
\end{align}
where $p \in \Z$. A filtration is called \textit{convergent above} if $A=\cup_{p}\, F_{p}A$. 
We suppose that $A$ is endowed with a differential $d:A\longrightarrow A$ such that:
\begin{align}
\label{differential}
 d^{2}=0 \quad \text{and} \quad d( F_{p}A)\subset F_{p-s+1}A,
\end{align}
 for fixed $s$ and all $p$ in $\Z$. The classical case studied in \cite[Chapter XI, Section 3]{maclane} corresponds to $s=1$. We will need the case $s=0$. \\
The filtration \eqref{filtration} induces a filtration on the module $H(A)$ of the homology spaces of $A$: indeed, for every $p \in \Z$, $F_{p}H(A) $ is defined as the image of $H(F_{p}A) $ under the injection $F_{p}A \longrightarrow A$.
\begin{defi}
\label{famigliacondiff}
Let $E=\left\{E_{p}\right\}_{p\in \Z}$ be a family of modules. A differential $d:E\longrightarrow E$ of degree $-r\in \Z$ is a family of homorphisms $\left\{d_{p}: E_{p}\longrightarrow E_{p-r}\right\}_{p \in \Z}$ such that $d_{p}\circ d_{p+r}=0$ for all $p \in \Z$. We denote by $H(E)=H(E,d)$ the homology of $E$ under the differential $d$ that is the family $\left\{H_{p}(E,d)\right\}_{p\in \Z}$, where:
\begin{align*}
H_{p}(E,d)=\frac{\Ker(d_{p}: E_{p}\longrightarrow E_{p-r})}{\Ima (d_{p+r}: E_{p+r}\longrightarrow E_{p})}.
\end{align*}
\end{defi}
\begin{defi}[Spectral sequence]
\label{defispec}
A \textit{spectral sequence} $E=\left\{(E^{r},d^{r})\right\}_{r\in \Z}$ is a sequence of families of modules with differential $(E^{r},d^{r})$ as in definition \ref{famigliacondiff}, such that, for all $r$, $d^{r}$ has degree $-r$ and:
\begin{align*}
H(E^{r},d^{r}) \cong E^{r+1}.
\end{align*}
\end{defi}
\begin{prop}
\label{spectral1}
Let $A$ be a module with a filtration as in \eqref{filtration} and differential as in \eqref{differential}. Therefore it is uniquely determined a spectral sequence, as in definition \ref{defispec}, $E=\left\{(E^{r},d^{r})\right\}_{r \in \Z}$ such that:
\begin{gather}
 H(E^{r},d^{r}) \cong E^{r+1}, \label{spectral1prima} \\
E^{r}_{p} \cong F_{p}A / F_{p-1}A \quad \text{for}  \quad r\leq s-1, \label{spectral1seconda} \\
d^{r}=0 \quad \text{for}  \quad r <s-1, \label{spectral1terza}\\
d^{s-1}=\Gr d, \label{spectral1quarta}\\
E^{s}_{p} \cong H(F_{p}A / F_{p-1}A) . \label{grado0graded} 
\end{gather}
\end{prop}
\begin{proof}
For the proof see \cite[Appendix]{kacrudakovE36}.
\end{proof}
We point out that \eqref{grado0graded} states that $E^{s}$ is isomorphic to the homology of the module $\Gr A$ with respect to the differential induced by $d$. 
\begin{rem}
Let $\left\{(E^{r},d^{r})\right\}_{r \in \Z}$ be a spectral sequence as in definition \ref{defispec}. We know that $E^{1}_{p}\cong H_{p}(E^{0},d^{0})$. We denote $E^{1}_{p}\cong C^{0}_{p}/B^{0}_{p}$, where $C^{0}_{p}=\Ker d^{0}_{p}$ and $B^{0}_{p}=\Ima d^{0}_{p+r}$. Analogously $E^{2}_{p}\cong H_{p}(E^{1},d^{1})$ and $E^{2}_{p}\cong C^{1}_{p}/B^{1}_{p}$, where $C^{1}_{p}/B^{0}_{p}=\Ker d^{1}_{p}$, $B^{1}_{p}/B^{0}_{p}=\Ima d^{1}_{p+r}$ and $B^{1}_{p} \subset C^{1}_{p}$. Thus, by iteration we obtain the following inclusions:
\begin{align*}
B^{0}_{p}\subset B^{1}_{p} \subset B^{2}_{p} \subset... \subset C^{2}_{p} \subset C^{1}_{p}\subset C^{0}_{p}.
\end{align*}
\end{rem}
\begin{defi}
Let $A$ be a module with a filtration as in \eqref{filtration} and differential as in \eqref{differential}. Let $\left\{(E^{r},d^{r})\right\}_{r \in \Z}$ be the spectral sequence determined by Proposition \ref{spectral1}.
We define $E^{\infty}_{p}$ as:
\begin{align*}
E^{\infty}_{p}=\frac{\bigcap_{r}C^{r}_{p}}{\bigcup _{r}B^{r}_{p}}.
\end{align*}
Let $B$ be a module with a filtration as in \eqref{filtration}. We say that the spectral sequence converges to $B$ if, for all $p$:
\begin{align*}
E^{\infty}_{p}\cong F_{p}B/F_{p-1}B.
\end{align*}
\end{defi}
\begin{prop}
\label{spectral2}
Let $A$ be a module with a filtration as in \eqref{filtration} and differential as in \eqref{differential}. Let us suppose that the filtration is convergent above and, for some $N$, $F_{-N}A=0$. Then the spectral sequence converges to the homology of $A$, that is:
$$E^{\infty}_{p} \cong F_{p} H(A) / F_{p-1}H(A).$$
\end{prop}
\begin{proof}
For the proof see \cite[Appendix]{kacrudakovE36}.
\end{proof}
\begin{rem}
\label{spectralgraduato}
Let $A$ be a module with a filtration as in \eqref{filtration} and differential as in \eqref{differential}. We moreover suppose that $A=\oplus_{n\in \Z}A_{n}$ is a $\Z-$graded module and $d:A_{n}\longrightarrow A_{n-1}$ for all $n\in \Z$. Therefore the filtration \eqref{filtration} induces a filtration on each $A_{n}$. The family $\left\{F_{p}A_{n}\right\}_{p,n \in \Z}$ is indexed by $(p,n)$. It is customary to write the indices as $(p,q)$, where $p$ is the degree of the filtration and $q=n-p$ is the complementary degree. The filtration is called \textit{bounded below} if, for all $n\in \Z$, there exists a $s=s(n)$ such that $F_{s}A_{n}=0$.\\
In this case the spectral sequence $E=\left\{(E^{r},d^{r})\right\}_{r \in \Z}$, determined as in Proposition \ref{spectral1}, is a family of modules $E^{r}=\left\{E^{r}_{p,q}\right\}_{p ,q\in \Z}$ indexed by $(p,q)$, where $E^{r}_{p}=\sum_{p,q \in \Z}E^{r}_{p,q}$, with the differential $d^{r}=\left\{d^{r}_{p,q}: E_{p,q}\longrightarrow E_{p-r,q+r-1}\right\}_{p ,q\in \Z}$ of bidegree $(-r,r-1)$ such that $d_{p,q}\circ d_{p+r,q-r+1}=0$ for all $p,q \in \Z$.
 Equations \eqref{spectral1prima}, \eqref{spectral1seconda} ,\eqref{spectral1terza}, \eqref{spectral1quarta} and \eqref{grado0graded} can be written so that the role of $q$ is explicit. For instance, Equation \eqref{spectral1prima} can be written as:
\begin{align*}
H_{p,q}(E^{r},d^{r})=\frac{\Ker(d^{r}_{p,q}: E^{r}_{p,q}\longrightarrow E^{r}_{p-r,q+r-1})}{\Ima (d^{r}_{p+r,q-r+1}: E^{r}_{p+r,q-r+1}\longrightarrow E^{r}_{p,q})}\cong E^{r+1}_{p,q}.
\end{align*}
for all $p,q \in \Z$.
Equation \eqref{grado0graded} can be written as $E^{s}_{p,q} \cong H(F_{p}A_{p+q} / F_{p-1}A_{p+q})$ for all $p,q \in \Z$.
\end{rem}
We now recall some results on spectral sequences of bicomplexes; for further details see \cite{kacrudakovE36} and \cite[Chapter XI, Section 6]{maclane}.
\begin{defi}[Bicomplex]
\label{definizionebicomplesso}
A bicomplex $K$ is a family $\left\{K_{p,q}\right\}_{p,q \in \Z}$ of modules endowed with two families of differentials, defined for all integers $p,q$, $d'$ and $d''$ such that
\begin{align*}
d':K_{p,q}\longrightarrow K_{p-1,q} , \quad d'':K_{p,q}\longrightarrow K_{p,q-1}
\end{align*}
and ${d'}^{2}={d''}^{2}=d'd''+d''d'=0$.
\end{defi}
We can also  think $K$ as a $\Z-$bigraded module where $K=\sum_{p,q \in \Z}K_{p,q}$.
A bicomplex $K$ as in Definition \ref{definizionebicomplesso} can be represented by the following commutative diagram:
\begin{align}
 \label{diagrammabicomplesso}
\xymatrix{
&                          &  \ar[d]^{d''}   &  \ar[d]^{d''}   & \ar[d]^{d''}  & \\
 &\quad \cdots  \ar[r]^{d'} &K_{p+1,q+1} \ar[d]^{d''} \ar[r]^{d'} &K_{p,q+1}\ar[d]^{d''} \ar[r]^{d'}  &K_{p-1,q+1} \ar[d]^{d''} \ar[r]^{d'} & \cdots \quad   \\
 &\quad \cdots  \ar[r]^{d'} &K_{p+1,q} \ar[d]^{d''} \ar[r]^{d'} &K_{p,q}\ar[d]^{d''} \ar[r]^{d'}  &K_{p-1,q} \ar[d]^{d''} \ar[r]^{d'} &\cdots \quad   \\
 &\quad \cdots  \ar[r]^{d'} &K_{p+1,q-1}\ar[d]^{d''}  \ar[r]^{d'} &K_{p,q-1} \ar[d]^{d''} \ar[r]^{d'}  &K_{p-1,q-1} \ar[d]^{d''} \ar[r]^{d'} &\cdots \quad \\ 
 &                          &                         &                       &                         & .}
\end{align}
\begin{defi}[Second homology]
Let $K$ be a bicomplex. The \textit{second homology} of $K$ is the homology computed with respect to $d''$, i.e.:
\begin{align*}
H''_{p,q}(K)=\frac{\Ker(d'':K_{p,q}\longrightarrow K_{p,q-1})}{d''(K_{p,q+1})}.
\end{align*}
The second homology of $K$ is a bigraded complex with differential $d':H''_{p,q}(K)\longrightarrow H''_{p-1,q}(K)$ induced by the original $d'$. Its homology is defined as:
\begin{align*}
H'_{p}H''_{q}(K)=\frac{\Ker(d':H''_{p,q}(K)\longrightarrow H''_{p-1,q})}{d'(H''_{p+1,q}(K))},
\end{align*}
and it is a bigraded module.
\end{defi}
\begin{defi}[First homology]
Let $K$ be a bicomplex. The \textit{first homology} of $K$ is the homology computed with respect to $d'$, i.e.:
\begin{align*}
H'_{p,q}(K)=\frac{\Ker(d':K_{p,q}\longrightarrow K_{p-1,q})}{d'(K_{p+1,q})}.
\end{align*}
The first homology of $K$ is a bigraded complex with differential $d'':H'_{p,q}(K)\longrightarrow H'_{p,q-1}(K)$ induced by the original $d''$. Its homology is defined as:
\begin{align*}
H''_{q}H'_{p}(K)=\frac{\Ker(d'':H'_{p,q}(K)\longrightarrow H'_{p,q-1})}{d''(H'_{p,q+1}(K))},
\end{align*}
and it is a bigraded module. 
\end{defi}
\begin{defi}[Total complex]
A bicomplex $K$ defines a single complex $T=Tot(K)$:
\begin{align*}
T_{n}=\sum_{p+q=n}K_{p,q}, \quad d=d'+d'':T_{n}\longrightarrow T_{n-1}.
\end{align*}
From the properties of $d'$ and $d''$, it follows that $d^{2}=0$. 
\end{defi}
We point out that $T_{n}$ is the sum of the modules of the secondary diagonal in diagram \eqref{diagrammabicomplesso}. We have that:
\begin{align*}
\cdots \xrightarrow[]{d} T_{n+1}\xrightarrow[]{d} T_{n}\xrightarrow[]{d} T_{n-1} \xrightarrow[]{d} \cdots.
\end{align*}
The first filtration $F'$ of $T=Tot(K)$ is defined as:
\begin{align*}
(F'_{p}T)_{n}=\sum_{h\leq p}K_{h,n-h}.
\end{align*}
The associated spectral sequence $E'$ is called \textit{first spectral sequence}. Analogously we can define the second filtration and the \textit{second spectral sequence}.
\begin{prop}
\label{spectralbicomplex}
Let $(K,d',d'')$ be a bicomplex with total differential $d$. The first spectral sequence $E'=\left\{(E'^{r},d^{r})\right\}$, $E'^{r}=\sum_{p,q}E'^{r}_{p,q}$ has the property:
\begin{align*}
(E'^{0},d^{0})\cong (K,d'') , \quad (E'^{1},d^{1})\cong (H(K,d''),d'), \quad E'^{2}_{p,q}\cong H'_{p}H''_{q}(K).
\end{align*}
The second spectral sequence $E''=\left\{(E''^{r},\delta^{r})\right\}$, $E''^{r}=\sum_{p,q}E''^{r}_{p,q}$ has the property:
\begin{align*}
(E''^{0},\delta^{0})\cong (K,d') , \quad (E''^{1},\delta^{1})\cong (H(K,d'),d''), \quad E''^{2}_{p,q}\cong H''_{q}H'_{p}(K).
\end{align*}
If the first filtration is bounded below and convergent above, then the first spectral sequence converges to the homology of $T$ with respect to the total differential $d$.\\
If the second filtration is bounded below and convergent above, then the second spectral sequence converges to the homology of $T$ with respect to the total differential $d$.
\end{prop}
\begin{proof}
See \cite[Chapter XI]{maclane}.
\end{proof}
\section{Homology}
\label{omologia}
In this section we compute the homology of the complexes in Figure \ref{figura}. The main result is the following theorem.
\begin{thm}
\label{thmomologia}
The complexes in Figure \ref{figura} are exact in each module except for $M(0,0,0,0)$ and $M(1,1,3,0)$. The homology spaces in $M(0,0,0,0)$ and $M(1,1,3,0)$ are isomorphic to the trivial representation.
\end{thm}
\begin{lem}
\label{lemmarealize}
Let $\nabla :M(\mu_{1},\mu_{2},\mu_{3},\mu_{4})\longrightarrow M(\widetilde{\mu}_{1},\widetilde{\mu}_{2},\widetilde{\mu}_{3},\widetilde{\mu}_{4})$ be a morphism represented in Figure \ref{figura} and constructed as in Remark \ref{costruzionemorfismi}. Then $\Ima \nabla$ is an irreducible $\g-$submodule of $M(\widetilde{\mu}_{1},\widetilde{\mu}_{2},\widetilde{\mu}_{3},\widetilde{\mu}_{4})$.
\end{lem}
\begin{proof}
By Theorems \ref{sing1}, \ref{sing2}, \ref{sing3} and Remark \ref{costruzionemorfismi}, we know that $M(\widetilde{\mu}_{1},\widetilde{\mu}_{2},\widetilde{\mu}_{3},\widetilde{\mu}_{4})$ contains a unique, up to scalars, highest weight nontrivial singular vector, that we call $\vec{m}$. By construction of $\nabla$, $\Ima \nabla$ is the $\g-$submodule of $M(\widetilde{\mu}_{1},\widetilde{\mu}_{2},\widetilde{\mu}_{3},\widetilde{\mu}_{4})$ generated by $\vec{m}$. In particular it is straightforward that $\g_{0}\vec{m}$ is an irreducible finite$-$dimensional $\g_{0}-$module on which $\g_{>0}$ acts trivially, since $\vec{m}$ is singular. The $\g-$module $\Ima \nabla=\g\vec{m}$ is therefore isomorphic to $\Ind(\g_{0}\vec{m})$. Hence, due to Theorem  \ref{keythmsingular}, $\Ima \nabla$ is an irreducible $\g-$module since in $M(\widetilde{\mu}_{1},\widetilde{\mu}_{2},\widetilde{\mu}_{3},\widetilde{\mu}_{4})$ there is only the highest weight nontrivial singular vector $\vec{m}$ that is trivial for $\Ind(\g_{0}\vec{m})$.
\end{proof}
\begin{rem}
We are now able to show, using Theorem \ref{thmomologia}, that all the irreducible quotients of finite degenerate Verma modules occur among cokernels, kernels and images of the complexes in Figure \ref{figura}. 
Indeed if $(\mu_{1},\mu_{2},\mu_{3},\mu_{4})$ is not among the weights that occur in Theorems \ref{sing1}, \ref{sing2}, \ref{sing3}, then $M(\mu_{1},\mu_{2},\mu_{3},\mu_{4})$ is irreducible, due to Theorem \ref{keythmsingular}, since it does not contain nontrivial singular vectors. Let us now suppose that $(\mu_{1},\mu_{2},\mu_{3},\mu_{4})$ is among the weights that occur in Theorems \ref{sing1}, \ref{sing2}, \ref{sing3}. By Theorem \ref{keythmsingular}, $M(\mu_{1},\mu_{2},\mu_{3},\mu_{4})$ is degenerate. We denote its irreducible quotient by $I(\mu_{1},\mu_{2},\mu_{3},\mu_{4})$. By Figure \ref{figura} we know that if $(\mu_{1},\mu_{2},\mu_{3},\mu_{4}) \neq (0,0,0,0)$ there exist two morphisms $\nabla :M(\mu_{1},\mu_{2},\mu_{3},\mu_{4})\rightarrow M(\widetilde{\mu}_{1},\widetilde{\mu}_{2},\widetilde{\mu}_{3},\widetilde{\mu}_{4})$ and $\widehat{\nabla} :M(\hat{\mu}_{1},\hat{\mu}_{2},\hat{\mu}_{3},\hat{\mu}_{4})\rightarrow M(\mu_{1},\mu_{2},\mu_{3},\mu_{4})$ constructed as in Remark \ref{costruzionemorfismi}. Due to Lemma \ref{lemmarealize}, $\Ker \nabla$ is the maximal submodule of $M(\mu_{1},\mu_{2},\mu_{3},\mu_{4})$ because $M(\mu_{1},\mu_{2},\mu_{3},\mu_{4})/\Ker \nabla \cong \Ima\nabla $ is irreducible and $\Ima\nabla \cong I(\mu_{1},\mu_{2},\mu_{3},\mu_{4})$. Moreover if $(\mu_{1},\mu_{2},\mu_{3},\mu_{4}) \neq (1,1,3,0)$, by Theorem \ref{thmomologia}, then $\Ker \nabla=\Ima \widehat{\nabla}$ is irreducible and it is the unique nontrivial submodule of $M(\mu_{1},\mu_{2},\mu_{3},\mu_{4})$; in this case $I(\mu_{1},\mu_{2},\mu_{3},\mu_{4})$ can be realized also as the cokernel of the map $\widehat{\nabla}$. We point out that $M(1,1,3,0)$ can be realized as $M(1,1,3,0)/\Ker \nabla \cong \Ima\nabla$.
	Finally, in the case of $M(0,0,0,0)$, by Remark \ref{costruzionemorfismi}, we have a morphism $\nabla:M(1,1,-1,0) \rightarrow M(0,0,0,0)$. By Theorem \ref{thmomologia}, $ M(0,0,0,0)/ \Ima \nabla$ is irreducible and therefore $I(0,0,0,0)\cong  M(0,0,0,0)/ \Ima \nabla$.
\end{rem}
The aim of this section is to prove Theorem \ref{thmomologia}.
Following \cite{kacrudakovE36}, we consider the following filtration on $U(\g_{<0})$: for all $i\geq 0$, $ F_{i}U(\g_{<0})$ is defined as the subspace of $U(\g_{<0})$ spanned by elements with at most $i$ terms of $\g_{<0}$. Namely:
\begin{align*}
\C=F_{0}U(\g_{<0}) \subset F_{1}U(\g_{<0})\subset ...\subset F_{i-1}U(\g_{<0})\subset F_{i}U(\g_{<0})\subset... \, ,
\end{align*}
where $F_{i}U(\g_{<0})=\g_{<0}F_{i-1}U(\g_{<0})+F_{i-1}U(\g_{<0})$.
We call $F_{i}M_{X}=F_{i}U(\g_{<0})\otimes V_{X}$. By \eqref{nabla}, it follows that $\nabla F_{i}M_{X} \subset F_{i+1}M_{X} $. Therefore $M_{X}$ is a filtered complex with the bigrading induced by \eqref{bigrading} and differential $\nabla$; moreover $M_{X}=\cup_{i }F_{i}M_{X}$ and $F_{-1}M_{X}=0$. Hence we can apply Propositions \ref{spectral1} and \ref{spectral2} to our complex $(M_{X}, \nabla)$ and obtain a spectral sequence $\left\{(E^{i}, \nabla^{i})\right\}$ such that $E^{0}=H(\Gr M_{X})$, $E^{i+1} \cong H(E^{i}, \nabla^{i})$ and $E^{\infty} \cong \Gr H(M_{X}) $. Thus we start by studying the homology of $\Gr M_{X}$.
\begin{rem}
We observe that $\g$ contains a copy of $W(1,0)=\langle p(t) \partial_{t}\rangle$ via the injective Lie superalgebras morphism:
\begin{gather*}
W(1,0)\longrightarrow \g \\
p(t) \partial_{t}\longrightarrow \frac{p(t)}{2} .
\end{gather*}
In particular, we point out that $\g_{-2}$ is contained in this copy of $W(1,0)$. 
\end{rem}
We consider the standard filtration on $W(1,0)=L_{-1}^{W}\supset L_{0}^{W} \supset L_{1}^{W} ...\,$. 
\begin{lem}
\label{filtrazioneWlemma}
For all $i\geq 0$ and $j\geq -1$: 
\begin{align}
\label{filtrazioneW}
L_{j}^{W} F_{i}M_{X} \subset F_{i-j}M_{X}.
\end{align}
\end{lem}
\begin{proof}
We point out that $L_{j}^{W} \subseteq\bigoplus_{k\geq j} \g_{2k}$, since $p(t)\partial_{t} \in L^{W}_{\degr(p(t))-1}$ corresponds to $\frac{p(t)}{2} \in \g$ and  $\degr(\frac{p(t)}{2})=2\degr(p(t))-2$.\\
Let us fix $j$ and show the thesis by induction on $i$. It is clear that $L_{j}^{W} F_{0}M_{X} \subset F_{-j}M_{X}$. Indeed let $w_{j} \in L_{j}^{W} $, $v \in F_{0}M_{X}$, then:
\begin{align*}
w_{j}.v \in
\begin{cases}
 F_{0}M_{X} \quad \text{if} \quad  j\geq0;\\
F_{1}M_{X} \quad \text{if} \quad  j=-1.
\end{cases}
\end{align*}
We now suppose that the thesis holds for $i$. Let $w_{j} \in L_{j}^{W} $ and $u_{1}u_{2}...u_{r} \otimes v \in F_{i+1}M_{X}$, with $r\leq i+1$ and $u_{1},u_{2},...,u_{r} \in \g_{<0}$. We moreover suppose that, for some $N$, $u_{s}= \Theta$ for all $s\leq N$ and $u_{s} \in \g_{-1}$ for all $s> N$ . We obtain that:
\begin{align*}
w_{j}u_{1}u_{2}...u_{r} \otimes v =(-1)^{p(w_{j})p(u_{1})}u_{1}w_{j}u_{2}...u_{r} \otimes v+[w_{j},u_{1}]u_{2}...u_{r} \otimes v.
\end{align*}
By hypothesis of induction, $u_{1}w_{j}u_{2}...u_{r} \otimes v \in F_{r-j}M_{X}\subset F_{i+1-j}M_{X}$. Let us focus on $[w_{j},u_{1}]u_{2}...u_{r} \otimes v$. We have two possibilities: $[w_{j},u_{1}] \in \oplus_{k\geq j} \g_{2k-2}$ if $u_{1}=\Theta$ or $[w_{j},u_{1}] \in \oplus_{k\geq j} \g_{2k-1}$ if $u_{1} \in \g_{-1}$.\\
If $u_{1}=\Theta$, then $[w_{j},u_{1}] \in L^{W}_{j-1}$ and, by hypothesis of induction, $[w_{j},u_{1}]u_{2}...u_{r} \otimes v \in  F_{r-j}M_{X}\subset F_{i+1-j}M_{X}$. If $u_{1} \in \g_{-1}$, then $\degr([w_{j},u_{1}]u_{2}...u_{r})\geq 2j-1-r+1$ and, by our assumption, $u_{2},...,u_{r} \in \g_{-1}$. Therefore $[w_{j},u_{1}]u_{2}...u_{r} \otimes v \in F_{r-2j}M_{X} \subset F_{i+1-j}M_{X}$.
\end{proof}
By \eqref{filtrazioneW} and the fact that $W(1,0) \cong \Gr W(1,0)$, we have that the action of $W(1,0)$ on $M_{X}$ descends on $\Gr M_{X}$.\\
 By the Poincar\'e$-$Birkhoff$-$Witt Theorem, $\Gr U(\g_{<0}) \cong S(\g_{-2}) \otimes \inlinewedge(\g_{-1})$; indeed, in $U(\g_{<0})$, $\eta^{2}_{i}=\Theta$ for all $i \in \left\{1,2,3,4\right\}$. 
Therefore:
\begin{align*}
\Gr M_{X}=\Gr U(\g_{<0}) \otimes V_{X}\cong S(\g_{-2}) \otimes \displaywedge(\g_{-1})\otimes V_{X}.
\end{align*}
We define:
\begin{align*}
\mathcal{W}=W(1,0)+\g_{0}=W(1,0)\oplus \g_{0}^{ss} \oplus \C C,
\end{align*}
that is a Lie subalgebra of $\g$. On $\mathcal{W}$ we consider the filtration $\mathcal{W}=L_{-1}^{\mathcal{W}}\supset L_{0}^{\mathcal{W}} \supset L_{1}^{\mathcal{W}} ...\,$, where $L_{0}^{\mathcal{W}}=L_{0}^{W} \oplus \g_{0}^{ss} \oplus \C C$ and $L_{k}^{\mathcal{W}}=L_{k}^{W}$ for all $k>0$.
From \eqref{filtrazioneW}, it follows that $L_{1}^{\mathcal{W}}=L_{1}^{W}$ annihilates $G_{X}:=\inlinewedge(\g_{-1}) \otimes V_{X}$. Therefore, as $\mathcal{W}-$modules:
\begin{align*}
\Gr M_{X} \cong S(\g_{-2}) \otimes (\displaywedge(\g_{-1}) \otimes V_{X}) \cong \Ind ^{\mathcal{W}}_{L^{\mathcal{W}}_{0}}(\displaywedge(\g_{-1}) \otimes V_{X}).
\end{align*}
We observe that $\Gr M_{X}$ is a complex with the morphism induced by $\nabla$, that we still call $\nabla$.
Indeed $\nabla F_{i}M_{X} \subset F_{i+1}M_{X} $ for all $i$ and therefore it is well defined the induced morphism
\begin{align*}
\nabla: \Gr_{i }M_{X}= F_{i}M_{X}/F_{i-1}M_{X}   \longrightarrow \Gr_{i+1 }M_{X}= F_{i+1}M_{X}/F_{i}M_{X},
\end{align*}
that has the same expression as $\nabla$ defined in \eqref{nabla}, except that the multiplication by the $w$'s must be seen as multiplication in $\Gr U(\g_{<0})$ instead of $ U(\g_{<0})$. Thus $(G_{X}, \nabla)$ is a subcomplex of $(\Gr M_{X}, \nabla)$: indeed it is sufficient to restrict $\nabla$ to $G_{X}$; the complex $(\Gr M_{X}, \nabla)$ is obtained from $(G_{X}, \nabla)$ extending the coefficients to $S(\g_{-2})$.\\
We point out that also the homology spaces $H^{m,n}(G_{X})$ are annihilated by $L_{1}^{\mathcal{W}}$. Therefore, as $\mathcal{W}-$modules:
\begin{align}
\label{keyhomology}
H^{m,n}(\Gr M_{X}) \cong S(\g_{-2}) \otimes H^{m,n}(G_{X}) \cong \Ind ^{\mathcal{W}}_{L^{\mathcal{W}}_{0}}(H^{m,n}(G_{X}) ).
\end{align}
By \eqref{keyhomology} and Propositions \ref{spectral1}, \ref{spectral2}, we obtain the following result.
\begin{prop}
\label{keyhomologyspectral}
If $H^{m,n}(G_{X})=0$, then $H^{m,n}(\Gr M_{X})=0$ and therefore $H^{m,n}(M_{X})=0$.
\end{prop}
\subsection{Homology of the complexes $G_{X}$}
Motivated by Proposition \ref{keyhomologyspectral}, in this subsection we study the homology of the complexes $G_{X}$'s. We will use arguments similar to the one used in \cite{kacrudakovE36} for $E(3,6)$. We point out that, in order to compute the homology of the $M_{X}$'s, we will compute the homology of the $G_{X}$'s for $X=A,C,D$ and we will use arguments of conformal duality, in particular Remark \ref{cantacasellikacremarksize} and Proposition \ref{esattezzafuntoreduale}, for the case $X=B$.\\
 We denote by $G_{X'}:=\inlinewedge(\g_{-1}) \otimes V_{X'}$.
 Let us consider the evaluation maps from $V_{X}$ to $V_{X^{'}}$ that map $y_{1},y_{2},\partial_{y_{1}},\partial_{y_{2}}$ to zero and are the identity on all the other elements. We consider the corresponding evaluation maps from $G_{X}$ to $G_{X^{'}}$. We can compose these maps with $\nabla_{2}$ when $X=A,B$ and obtain new maps, that we still call $\nabla_{2}$, from $G_{A}$ to $G_{D^{'}}$ and from $G_{B}$ to $G_{C^{'}}$ respectively. \\
We consider also the map from $G_{A^{'}}$ to $G_{D}$ (resp. from $G_{B^{'}}$ to $G_{C}$) that is the composition of $\nabla_{2}:G_{A^{'}} \longrightarrow G_{D^{'}}$ (resp. $\nabla_{2}:G_{B^{'}} \longrightarrow G_{C^{'}}$) and the inclusion of $G_{D^{'}}$ into $G_{D}$ (resp. $G_{C^{'}}$ into $G_{C}$); we will call also this composition $\nabla_{2}$.
 We let:
\begin{align*}
&G_{A^{\circ}}=\Ker (\nabla_{2}:G_{A} \longrightarrow G_{D^{'}} ) , \quad G_{D^{\circ}}=\CoKer (\nabla_{2}:G_{A^{'}} \longrightarrow G_{D} ),\\
&G_{B^{\circ}}=\Ker (\nabla_{2}:G_{B} \longrightarrow G_{C^{'}} ) ,\quad G_{C^{\circ}}=\CoKer (\nabla_{2}:G_{B^{'}} \longrightarrow G_{C} ).
\end{align*}
\begin{rem}
\label{rem1Gcirc}
The map $\nabla$ is still defined on $G_{X^{\circ}}$ since $\nabla \nabla_{2}= \nabla_{2}\nabla=0$.\\
The bigrading \eqref{bigrading} induces a bigrading also on the $G_{X^{\circ}}$'s. We point out that $G_{A}^{m,n}=G^{m,n}_{A^{\circ}} $ for $n>0$, $G_{D}^{m,n}=G^{m,n}_{D^{\circ}} $ for $n<0$, $G_{B}^{m,n}=G^{m,n}_{B^{\circ} }$ for $n>0$ and $G_{C}^{m,n}=G^{m,n}_{C^{\circ}} $ for $n<0$.\\
The complexes $(G_{X^{\circ}},\nabla)$ start or end at the axes of Figure \ref{figura}; thus for $m,n \in \Z$:
\begin{align*}
H^{m,n}(G_{A^{\circ}})&=\frac{G^{m,n}_{A^{\circ}}}{\Ima (\nabla:G^{m+1,n+1}_{A^{\circ}}\longrightarrow  G^{m,n}_{A^{\circ}})}             \quad  \text{for} \,\,\, m=0 \,\, \text{or} \,\, n=0;\\
H^{m,n}(G_{D^{\circ}})&=
\begin{cases}
\frac{G^{m,n}_{D^{\circ}}}{\Ima (\nabla:G^{m+1,n+1}_{D^{\circ}}\longrightarrow  G^{m,n}_{D^{\circ}})}             \quad  \text{for} \,\,\, m=0;\\
\Ker (\nabla:G^{m,n}_{D^{\circ}}\longrightarrow  G^{m-1,n-1}_{D^{\circ}})            \quad  \text{for} \,\,\, n=0;
\end{cases}\\
H^{m,n}(G_{B^{\circ}})&=
\begin{cases}
\Ker (\nabla:G^{m,n}_{B^{\circ}}\longrightarrow  G^{m-1,n-1}_{B^{\circ}})            \quad  \text{for} \,\,\, m=0;\\
\frac{G^{m,n}_{B^{\circ}}}{\Ima (\nabla:G^{m+1,n+1}_{B^{\circ}}\longrightarrow  G^{m,n}_{B^{\circ}})} \quad  \text{for} \,\,\, n=0;
\end{cases}\\
H^{m,n}(G_{C^{\circ}})&=\Ker (\nabla:G^{m,n}_{C^{\circ}}\longrightarrow  G^{m-1,n-1}_{C^{\circ}})   \quad  \text{for} \,\,\, m=0 \,\, \text{or} \,\, n=0.
\end{align*}
\end{rem}
\begin{rem}
\label{rem2Gcirc}
The following relations are a direct consequence of the definition of the $G_{X^{\circ}}$'s and Remark \ref{rem1Gcirc}:
\begin{align*}
H^{m,n}(G_{A})&=H^{m,n}(G_{A^{\circ}}) \quad \text{for} \,\,\, m>0,n\geq0;\\
H^{m,n}(G_{D})&=H^{m,n}(G_{D^{\circ}}) \quad \text{for} \,\,\, m>0,n \leq 0;\\
H^{m,n}(G_{B})&=H^{m,n}(G_{B^{\circ}})\quad \text{for} \,\,\, m<0,n \geq 0;\\
H^{m,n}(G_{C})&=H^{m,n}(G_{C^{\circ}}) \quad \text{for} \,\,\, m<0,n \leq 0.
\end{align*}
\end{rem}
Motivated by Remark \ref{rem2Gcirc} and Proposition \ref{keyhomologyspectral}, we study the homology of the complexes $G_{X^{\circ}}$'s.\\
We therefore introduce an additional bigrading as follows:
\begin{align}
\label{bigradingpq}
&(V_{X})_{[p,q]}=\left\{f \in V_{X} \, : \, y_{1}\partial_{y_{1}}f=pf \, \, \text{and} \, \, y_{2}\partial_{y_{2}}f=qf\right\},\\ \nonumber
&(G_{X})_{[p,q]}=\displaywedge(\g_{-1}) \otimes (V_{X})_{[p,q]}.
\end{align}
The definition can be extended also to $G_{X^{\circ}}$. We point out that this new bigrading is related to the bigrading \eqref{bigrading} by the equation $p+q=n$.\\
We have that $d':=\Delta^{+}\partial_{y_{1}}:(G_{X})_{[p,q]}\longrightarrow  (G_{X})_{[p-1,q]}$, $d'':=\Delta^{-}\partial_{y_{2}}:(G_{X})_{[p,q]}\longrightarrow  (G_{X})_{[p,q-1]}$ and $d=d'+d''=\nabla:\oplus_{m}G^{m,n}_{X} \longrightarrow \oplus_{m}G^{m,n-1}_{X}$. \\
We know, by Remark \ref{nabla+nabla-}, that ${d'}^{2}={d''}^{2}=d'd''+d''d'=0$. Therefore $\oplus_{m}G^{m}_{X}$ and $\oplus_{m}G^{m}_{X^{\circ}}$, with the bigrading \eqref{bigradingpq}, are bicomplexes with differentials  $d'$, $d''$ and total differential $\nabla=d'+d''$.\\
Now let:
\begin{align*}
\displaywedge_{+}^{i}=\displaywedge^{i}\langle w_{11},w_{21} \rangle \quad \text{and} \quad  \displaywedge_{-}^{i}=\displaywedge^{i}\langle w_{12},,w_{22}\rangle.
\end{align*}
We point out that $\inlinewedge_{+}^{i}$ and $\inlinewedge_{-}^{i}$ are isomorphic as $\langle x_{1}\partial_{x_{1}}-x_{2}\partial_{x_{2}}, x_{1}\partial_{x_{2}},x_{2}\partial_{x_{1}}\rangle-$modules; therefore, in the following, we will often write $\inlinewedge^{i}$ when we are speaking of the $\langle x_{1}\partial_{x_{1}}-x_{2}\partial_{x_{2}}, x_{1}\partial_{x_{2}},x_{2}\partial_{x_{1}}\rangle-$ module isomorphic to $\inlinewedge_{+}^{i}$ and $\inlinewedge_{-}^{i}$.\\
We introduce the following notation, for all $a,b \in \Z$ and $p,q\geq 0$:
\begin{align*}
&G_{A}(a,b)_{[p,q]}=\displaywedge_{+}^{a-p}\displaywedge_{-}^{b-q} \otimes \C \left[x_{1},x_{2}\right]y_{1}^{p}y_{2}^{q}, \quad \quad \,\,\,\, &&G_{B}(a,b)_{[p,q]}=\displaywedge_{+}^{a-p}\displaywedge_{-}^{b-q} \otimes \C \left[\partial_{x_{1}},\partial_{x_{2}}\right]y_{1}^{p}y_{2}^{q},
\end{align*}
and for all $a,b \in \Z$, $p,q\leq 0$:
\begin{align*}
&G_{D}(a,b)_{[p,q]}=\displaywedge_{+}^{a-p}\displaywedge_{-}^{b-q} \otimes \C \left[x_{1},x_{2}\right]\partial_{y_{1}}^{-p}\partial_{y_{2}}^{-q}, \quad &&G_{C}(a,b)_{[p,q]}=\displaywedge_{+}^{a-p}\displaywedge_{-}^{b-q}\otimes \C \left[\partial_{x_{1}},\partial_{x_{2}}\right]\partial_{y_{1}}^{-p}\partial_{y_{2}}^{-q} .
\end{align*}
From now on we will use the notation $\inlinewedge_{\pm}^{i}[x_{1},x_{2}]$ (resp. $\inlinewedge_{\pm}^{i}[\partial_{x_{1}},\partial_{x_{2}}]$) for $\inlinewedge_{\pm}^{i} \otimes \C [x_{1},x_{2}]$ (resp. $\inlinewedge_{\pm}^{i} \otimes \C [\partial_{x_{1}},\partial_{x_{2}}]$).\\
We point out that, as $\langle x_{1}\partial_{x_{1}}-x_{2}\partial_{x_{2}}, x_{1}\partial_{x_{2}},x_{2}\partial_{x_{1}}\rangle-$modules, $G_{X}=\oplus_{a,b}G_{X}(a,b)$, where $G_{X}(a,b)=\oplus _{p,q}G_{X}(a,b)_{[p,q]}$.\\ Similarly we define $G_{X^{\circ}}(a,b)_{[p,q]}$ and, as $\langle x_{1}\partial_{x_{1}}-x_{2}\partial_{x_{2}}, x_{1}\partial_{x_{2}},x_{2}\partial_{x_{1}}\rangle-$modules, $G_{X^{\circ}}=\oplus_{a,b}G_{X^{\circ}}(a,b)$, where $G_{X^{\circ}}(a,b)=\oplus _{p,q}G_{X^{\circ}}(a,b)_{[p,q]}$. \\
We observe that $\nabla: G_{X}(a,b)\rightarrow G_{X}(a,b)$ (resp. $\nabla: G_{X^{\circ}}(a,b) \rightarrow G_{X^{\circ}}(a,b)$) and therefore $G_{X}(a,b)$ (resp. $G_{X^{\circ}}(a,b)$) is a subcomplex of $G_{X}$ (resp. $G_{X^{\circ}}$); the $G_{X}(a,b)$'s and $G_{X^{\circ}}(a,b)$'s are bicomplexes, with the bigrading induced by \eqref{bigradingpq} and differentials $d'=\Delta^{+}\partial_{y_{1}}$ and $d''=\Delta^{-}\partial_{y_{2}}$. The computation of homology spaces of $G_{X}$ and $G_{X^{\circ}}$ can be reduced to the computation for $G_{X}(a,b)$ and $G_{X^{\circ}}(a,b)$.
In the following lemmas we compute the homology of the $G_{X^{\circ}}(a,b)$'s. We start with the homology of the $G_{X^{\circ}}(a,b)$'s when either $a$ or $b$ do not lie in $\left\{0,1,2\right\}$. To prove the following results we will use Proposition \ref{spectralbicomplex}.
\begin{lem} 
\label{lemmi 4.3 4.4 4.5}
Let us suppose that $a>2$ or $b>2$. Let $k=max(a,b)$.\\
Then as $\langle x_{1}\partial_{x_{1}}-x_{2}\partial_{x_{2}}, x_{1}\partial_{x_{2}},x_{2}\partial_{x_{1}}\rangle-$modules:
\begin{align*}
H^{m,n}(G_{A^{\circ}}(a,b))&=H^{m,n}(G_{A}(a,b))\cong
\begin{cases}
0 \, \,\,\,\, &\text{if} \, \, \, m>0   \, \,\, \text{or} \, \,\, (m=0 ,\, \,\, n< k), \\
\inlinewedge^{a+b-n}  \,\, \,\,\,  &\text{if} \, \,\, m=0 ,\, \,\, n\geq k .
\end{cases}
\end{align*}
Let us suppose that $a<0$ or $b<0$. Let $k=min(a,b)$.\\
Then, as $\langle x_{1}\partial_{x_{1}}-x_{2}\partial_{x_{2}}, x_{1}\partial_{x_{2}},x_{2}\partial_{x_{1}} \rangle-$modules:
\begin{align*}
H^{m,n}(G_{D^{\circ}}(a,b))&=H^{m,n}(G_{D}(a,b))\cong
\begin{cases}
0 \, \,\,\, \,&\text{if} \, \, \, m>0 \, \,\, \text{or} \, \,\, (m=0, \, \, \, n> k),\\
\inlinewedge^{a+b-n}  \, \,\,\,\,  &\text{if} \, \,\, m=0, \, \, \, n\leq k;
\end{cases}\\
H^{m,n}(G_{C^{\circ}}(a,b))&=H^{m,n}(G_{C}(a,b))\cong
\begin{cases}
0 \, \, &\text{if} \, \, \, m<0 \, \,\,\text{or} \, \,\, (m=0, \, \,\, n> k -2),\\
\inlinewedge^{a+b-n-2}  \, \,  &\text{if} \, \,\, m=0, \, \,\, n\leq k -2.
\end{cases}
\end{align*}
\begin{proof}
We first observe that if $a>2$ or $b>2$ (resp. $ a<0$ or $ b<0$), then $G_{X^{\circ}}(a,b)=G_{X}(a,b)$ for $X=A$ (resp. $X=C,D$); indeed they differ only when $p+q=0$, that does not occur in this case. We prove the thesis in the case $b>2$ for $X=A$ and $b<0$ for $X=C,D$; the case $a>2$ for $X=A$ and $a<0$ for $X=C,D$ can be proved analogously using the second spectral sequence instead of the first one.\\

\textbf{Case A)} Let us consider $G_{A^{\circ}}(a,b)$ with the differential $d''=\Delta^{-}\partial_{y_{2}}$:
\begin{align*}
...\xleftarrow[]{\Delta^{-}\partial_{y_{2}}}\displaywedge_{+}^{a-p}\displaywedge_{-}^{b-q+1}\left[x_{1},x_{2}\right]y_{1}^{p}y_{2}^{q-1} \xleftarrow[]{\Delta^{-}\partial_{y_{2}}} &\displaywedge_{+}^{a-p}\displaywedge_{-}^{b-q}\left[x_{1},x_{2}\right]y_{1}^{p}y_{2}^{q} \\
&\xleftarrow[]{\Delta^{-}\partial_{y_{2}}} \displaywedge_{+}^{a-p}\displaywedge_{-}^{b-q-1}\left[x_{1},x_{2}\right]y_{1}^{p}y_{2}^{q+1}  \xleftarrow[]{\Delta^{-}\partial_{y_{2}}}... \,.
\end{align*}
This complex is the tensor product of $\inlinewedge_{+}^{a-p}y_{1}^{p}$ and the following complex, since $\inlinewedge_{+}^{a-p}y_{1}^{p}$ is not involved by $d''$:
\begin{align*}
0 \xleftarrow[]{\Delta^{-}\partial_{y_{2}}}\displaywedge_{-}^{2}\left[x_{1},x_{2}\right]y_{2}^{b -2} \xleftarrow[]{\Delta^{-}\partial_{y_{2}}} \displaywedge_{-}^{1}\left[x_{1},x_{2}\right]y_{2}^{b-1} \xleftarrow[]{\Delta^{-}\partial_{y_{2}}} \displaywedge_{-}^{0}\left[x_{1},x_{2}\right]y_{2}^{b}  \xleftarrow[]{\Delta^{-}\partial_{y_{2}}} 0.
\end{align*}
We now show that this complex is exact except for the right end, in which the homology space is $\C y_{2}^{b}$. Indeed:
\begin{description}[leftmargin=0cm]
	\item[i] let us consider the map $\Delta^{-}\partial_{y_{2}}:\inlinewedge_{-}^{0}\left[x_{1},x_{2}\right]y_{2}^{b}  \longrightarrow \inlinewedge_{-}^{1}\left[x_{1},x_{2}\right]y_{2}^{b-1}  $. We compute the kernel. Let $p(x_{1},x_{2})y_{2}^{b} \in \inlinewedge_{-}^{0}\left[x_{1},x_{2}\right]y_{2}^{b}  $. Then
\begin{align*}
\Delta^{-}\partial_{y_{2}}(p(x_{1},x_{2})y_{2}^{b})=w_{12}\otimes \partial_{x_{1}}p(x_{1},x_{2})b y_{2}^{b-1}+w_{22}\otimes \partial_{x_{2}}p(x_{1},x_{2})b y_{2}^{b-1}
\end{align*}
is zero if and only if $p$ is constant. Hence the kernel is $\C y_{2}^{b} $.
\item[ii] Let us consider the map $\Delta^{-}\partial_{y_{2}}:\inlinewedge_{-}^{1}\left[x_{1},x_{2}\right]y_{2}^{b-1}  \longrightarrow \inlinewedge_{-}^{2}\left[x_{1},x_{2}\right]y_{2}^{b-2}$. We compute the kernel. Let $w_{12}\otimes p_{1}(x_{1},x_{2})y_{2}^{b-1}+w_{22}\otimes p_{2}(x_{1},x_{2})y_{2}^{b-1}\in \inlinewedge_{-}^{1}\left[x_{1},x_{2}\right]y_{2}^{b-1}  $. Then
\begin{align*}
&\Delta^{-}\partial_{y_{2}}(w_{12}\otimes p_{1}(x_{1},x_{2})y_{2}^{b-1}+w_{22}\otimes p_{2}(x_{1},x_{2})y_{2}^{b-1})\\
&\quad =w_{12}w_{22}\otimes \partial_{x_{2}}p_{1}(x_{1},x_{2})(b-1) y_{2}^{b-2}+w_{22}w_{12}\otimes \partial_{x_{1}}p_{2}(x_{1},x_{2})(b-1) y_{2}^{b-2}
\end{align*}
is zero if and only if $\partial_{x_{2}}p_{1}(x_{1},x_{2})=\partial_{x_{1}}p_{2}(x_{1},x_{2})$, that implies $p_{1}(x_{1},x_{2})=\int \partial_{x_{1}}p_{2}(x_{1},x_{2})dx_{2}$. Hence an element of the kernel is such that:
\begin{align*}
w_{12}\otimes \Big(\int\partial_{x_{1}}p_{2}(x_{1},x_{2})dx_{2}\Big)y_{2}^{b-1}+w_{22}\otimes p_{2}(x_{1},x_{2})y_{2}^{b-1} =\Delta^{-}\partial_{y_{2}}\Big(\Big(\int p_{2}(x_{1},x_{2})dx_{2}\Big)\frac{y_{2}^{b}}{b}\Big).
\end{align*}
Thus at this point the sequence is exact.
\item[iii] We consider the map $\Delta^{-}\partial_{y_{2}}:\inlinewedge_{-}^{2}\left[x_{1},x_{2}\right]y_{2}^{b-2}  \longrightarrow 0$.
We have that:
\begin{align*}
w_{12}w_{22}\otimes p(x_{1},x_{2})y_{2}^{b-2}= \Delta^{-}\partial_{y_{2}} \Big(w_{12}\otimes \Big(\int p(x_{1},x_{2})dx_{2} \Big)\frac{y_{2}^{b-1}}{b-1} \Big).
\end{align*}
Thus at this point the sequence is exact.
\end{description}
Since the original complex was the tensor product with $\inlinewedge_{+}^{a-p}y_{1}^{p}$, we have that the non zero homology space is $\inlinewedge_{+}^{a-p}y_{1}^{p}y_{2}^{b}$ and $E^{'1}_{p,q}(G_{A^{\circ}}(a,b))$ survives only for $q=b$. Now we should compute its homology with respect to $d'$, but the $E^{'1}_{p,q}(G_{A^{\circ}}(a,b))$'s do not involve $x_{1},x_{2}$, so the differentials $d'$'s are zero and we have $E^{'2}=E^{'1}$. Moreover, for a one$-$row spectral sequence, we know that $E^{'2}=...=E^{'\infty}$ since, for all $r\geq 2$ and all $p\in \Z$, $d^{r}_{p,b}$ has bidegree $(-r,r-1)$, i.e. $d^{r}_{p,b}:E^{r}_{p,b}\longrightarrow E^{r}_{p-r,b+r-1}=0$, $d^{r}_{p+r,b-r+1}:E^{r}_{p+r,b-r+1}=0\longrightarrow E^{r}_{p,b}$. Therefore:
\begin{align*}
E^{'\infty}_{p,q}(G_{A}(a,b))=
\begin{cases}
0 \quad & \text{if} \,\,\, q \neq b,\\
\inlinewedge_{+}^{a-p}y_{1}^{p}y_{2}^{b} \,\,\,\, & \text{if} \,\,\, q=b.
\end{cases}
\end{align*}
We observe that the first filtration $(F'_{p} (G_{A}(a,b)))_{n}=\sum_{h\leq p}(G_{A}(a,b))_{[h,n-h]}$ is bounded below, since $F'_{-1}=0$, and it is convergent above. Therefore by Proposition \ref{spectralbicomplex}:
\begin{align*}
\sum_{m}H^{m,n}(G_{A}(a,b))\cong \sum_{p+q=n}E^{'\infty}_{p,q}(G_{A}(a,b))=E^{'\infty}_{n-b,b}(G_{A}(a,b))\cong \displaywedge_{+}^{a+b-n}y_{1}^{n-b}y_{2}^{b}.  
\end{align*}
Since there are no $x_{1}$'s and $x_{2}$'s involved, this means that $H^{m,n}(G_{A}(a,b))=0$ if $m \neq 0$ and $H^{0,n}(G_{A}(a,b))=\inlinewedge_{+}^{a+b-n}y_{1}^{n-b}y_{2}^{b} \cong\inlinewedge^{a+b-n} $ as $\langle x_{1}\partial_{x_{1}}-x_{2}\partial_{x_{2}}, x_{1}\partial_{x_{2}},x_{2}\partial_{x_{1}}\rangle-$modules.\\

\textbf{Case D)} In the case of $G_{D}(a,b)$, using the same argument, when $ b <0$ we obtain:
\begin{align*}
E^{'\infty}_{p,q}(G_{D}(a,b))=
\begin{cases}
0 \quad & \text{if} \,\,\, q \neq b,\\
\inlinewedge_{+}^{a-p}\partial_{y_{1}}^{-p}\partial_{y_{2}}^{-b} \,\,\,\, & \text{if} \,\,\, q=b.
\end{cases}
\end{align*}
Therefore:
\begin{align*}
\sum_{m}H^{m,n}(G_{D}(a,b))\cong \sum_{p+q=n}E^{'\infty}_{p,q}(G_{D}(a,b))=E^{'\infty}_{n-b,b}(G_{D}(a,b))\cong \displaywedge_{+}^{a+b-n}\partial_{y_{1}}^{-n+b}\partial_{y_{2}}^{-b} . 
\end{align*}
Since there are no $x_{1}$'s and $x_{2}$'s involved, this means that $H^{m,n}(G_{D}(a,b))=0$ if $m \neq 0$ and $H^{0,n}(G_{D}(a,b))=\inlinewedge_{+}^{a+b-n}\partial_{y_{1}}^{-n+b}\partial_{y_{2}}^{-b}    \cong\inlinewedge^{a+b-n} $ as $\langle x_{1}\partial_{x_{1}}-x_{2}\partial_{x_{2}}, x_{1}\partial_{x_{2}},x_{2}\partial_{x_{1}}\rangle-$modules.\\
\textbf{Case C)} 
 In the case of $G_{C}(a,b)$ when $ b <0$ we have the following complex with the differential $d''=\Delta^{-}\partial_{y_{2}}$:
\begin{align*}
\leftarrow \displaywedge_{+}^{a-p}\displaywedge_{-}^{b-q+1}\left[\partial_{x_{1}},\partial_{x_{2}}\right]\partial_{y_{1}}^{-p}\partial_{y_{2}}^{-q+1} \xleftarrow[]{\Delta^{-}\partial_{y_{2}}}& \displaywedge_{+}^{a-p}\displaywedge_{-}^{b-q}\left[\partial_{x_{1}},\partial_{x_{2}}\right]\partial_{y_{1}}^{-p}\partial_{y_{2}}^{-q}\\
& \xleftarrow[]{\Delta^{-}\partial_{y_{2}}} \displaywedge_{+}^{a-p} \displaywedge_{-}^{b-q-1}\left[\partial_{x_{1}},\partial_{x_{2}}\right]\partial_{y_{1}}^{-p}\partial_{y_{2}}^{-q-1}\leftarrow  .
\end{align*}
This complex is the tensor product of $\inlinewedge_{+}^{a-p}\partial_{y_{1}}^{-p}$ and the following complex, since $\inlinewedge_{+}^{a-p}\partial_{y_{1}}^{-p}$ is not involved by $d''$:
\begin{align*}
0 \xleftarrow[]{\Delta^{-}\partial_{y_{2}}}\displaywedge_{-}^{2}\left[\partial_{x_{1}},\partial_{x_{2}}\right]\partial_{y_{2}}^{-b +2} \xleftarrow[]{\Delta^{-}\partial_{y_{2}}} \displaywedge_{-}^{1}\left[\partial_{x_{1}},\partial_{x_{2}}\right]\partial_{y_{2}}^{-b+1} \xleftarrow[]{\Delta^{-}\partial_{y_{2}}} \displaywedge_{-}^{0}\left[\partial_{x_{1}},\partial_{x_{2}}\right]\partial_{y_{2}}^{-b}  \xleftarrow[]{\Delta^{-}\partial_{y_{2}}} 0.
\end{align*}
We now show that this complex is exact except for the left end, in which the homology space is $\C \inlinewedge_{-}^{2} \partial_{y_{2}}^{-b+2}$. Indeed:
\begin{description}[leftmargin=0cm]
	\item[i] let us consider the map $\Delta^{-}\partial_{y_{2}}:\inlinewedge_{-}^{0}\left[\partial_{x_{1}},\partial_{x_{2}}\right]\partial_{y_{2}}^{-b}  \longrightarrow \inlinewedge_{-}^{1}\left[\partial_{x_{1}},\partial_{x_{2}}\right]\partial_{y_{2}}^{-b+1}  $. We compute the kernel. Let $p(\partial_{x_{1}},\partial_{x_{2}})\partial_{y_{2}}^{-b} \in \inlinewedge_{-}^{0}\left[\partial_{x_{1}},\partial_{x_{2}}\right]\partial_{y_{2}}^{-b}  $. Then
\begin{align*}
\Delta^{-}\partial_{y_{2}}(p(\partial_{x_{1}},\partial_{x_{2}})\partial_{y_{2}}^{-b})=w_{12}\otimes \partial_{x_{1}}p(\partial_{x_{1}},\partial_{x_{2}})\partial_{y_{2}}^{-b+1}+w_{22}\otimes \partial_{x_{2}}p(\partial_{x_{1}},\partial_{x_{2}}) \partial_{y_{2}}^{-b+1}
\end{align*}
is zero if and only if $\partial_{x_{1}}p(\partial_{x_{1}},\partial_{x_{2}})=\partial_{x_{2}}p(\partial_{x_{1}},\partial_{x_{2}})=0$, that is $p=0$. Hence the kernel is $0 $.
\item[ii] Let us consider the map $\Delta^{-}\partial_{y_{2}}:\inlinewedge_{-}^{1}\left[\partial_{x_{1}},\partial_{x_{2}}\right]\partial_{y_{2}}^{-b+1}  \longrightarrow \inlinewedge_{-}^{2}\left[\partial_{x_{1}},\partial_{x_{2}}\right]\partial_{y_{2}}^{-b+2}$. We compute the kernel. Let $w_{12}\otimes p_{1}(\partial_{x_{1}},\partial_{x_{2}})\partial_{y_{2}}^{-b+1}+w_{22}\otimes p_{2}(\partial_{x_{1}},\partial_{x_{2}})\partial_{y_{2}}^{-b+1}\in \inlinewedge_{-}^{1}\left[\partial_{x_{1}},\partial_{x_{2}}\right]\partial_{y_{2}}^{-b+1}  $. Then
\begin{align*}
\Delta^{-}\partial_{y_{2}}(w_{12}\otimes p_{1}\partial_{y_{2}}^{-b+1}+&w_{22}\otimes p_{2}\partial_{y_{2}}^{-b+1})=w_{12}w_{22}\otimes \partial_{x_{2}}p_{1} \partial_{y_{2}}^{-b+2}+w_{22}w_{12}\otimes \partial_{x_{1}}p_{2} \partial_{y_{2}}^{-b+2}
\end{align*}
is zero if and only if $\partial_{x_{2}}p_{1}(\partial_{x_{1}},\partial_{x_{2}})=\partial_{x_{1}}p_{2}(\partial_{x_{1}},\partial_{x_{2}})$, that is $p_{1}(\partial_{x_{1}},\partial_{x_{2}})=\frac{\partial_{x_{1}}p_{2}(\partial_{x_{1}},\partial_{x_{2}})}{\partial_{x_{2}}}$ (in particular $p_{2}$ has positive degree in $\partial_{x_{2}}$). Therefore an element of the kernel is such that:
\begin{align*}
w_{12}\otimes \frac{\partial_{x_{1}}p_{2}(\partial_{x_{1}},\partial_{x_{2}})}{\partial_{x_{2}}} \partial_{y_{2}}^{-b+1}+w_{22}\otimes p_{2}(\partial_{x_{1}},\partial_{x_{2}})\partial_{y_{2}}^{-b+1}=\Delta^{-}\partial_{y_{2}}\Big(\frac{p_{2}(\partial_{x_{1}},\partial_{x_{2}})}{\partial_{x_{2}}}\partial_{y_{2}}^{-b} \Big).
\end{align*}
Thus at this point the sequence is exact.
\item[iii] Let us consider the map $\Delta^{-}\partial_{y_{2}}:\inlinewedge_{-}^{2}\left[\partial_{x_{1}},\partial_{x_{2}}\right]\partial_{y_{2}}^{-b+2}  \longrightarrow 0$. Let $w_{12}w_{22}\otimes p(\partial_{x_{1}},\partial_{x_{2}})\partial_{y_{2}}^{-b+2}\in \inlinewedge_{-}^{2}\left[\partial_{x_{1}},\partial_{x_{2}}\right]\partial_{y_{2}}^{-b+2}  $. If $p$ has positive degree in $\partial_{x_{1}}$, then:
\begin{align*}
w_{12}w_{22}\otimes p(\partial_{x_{1}},\partial_{x_{2}})\partial_{y_{2}}^{-b+2}=\Delta^{-}\partial_{y_{2}}\Big(-w_{22}\otimes\frac{ p(\partial_{x_{1}},\partial_{x_{2}})}{\partial_{x_{1}}}\partial_{y_{2}}^{-b+1}\Big).
\end{align*}
If $p$ has positive degree in $\partial_{x_{2}}$, then:
\begin{align*}
w_{12}w_{22}\otimes p(\partial_{x_{1}},\partial_{x_{2}})\partial_{y_{2}}^{-b+2}=\Delta^{-}\partial_{y_{2}}\Big(w_{12}\otimes\frac{ p(\partial_{x_{1}},\partial_{x_{2}})}{\partial_{x_{2}}}\partial_{y_{2}}^{-b+1}\Big).
\end{align*}
If $p$ is constant, it does not belong to the image of $\Delta^{-}\partial_{y_{2}}$. Therefore the homology space is isomorphic to $\C \inlinewedge_{-}^{2}\partial_{y_{2}}^{-b+2}$.
\end{description}
Since the original complex was the tensor product with $\inlinewedge_{+}^{a-p}\partial_{y_{1}}^{-p}$, then the non zero homology space is $\inlinewedge_{+}^{a-p}\inlinewedge_{-}^{2}\partial_{y_{1}}^{-p}\partial_{y_{2}}^{-b+2}$ and $E^{'1}_{p,q}(G_{C}(a,b))$ survives only for $q=b-2$. 
Moreover $E^{'1}_{p,q}\cong E^{'2}_{p,q}$ because the map $d'$ is 0 on  the $E^{'1}_{p,q}$'s (the image of the map $d'$ always involves elements of positive degree in $\partial_{x_{1}}$ or $\partial_{x_{2}}$ that are 0 in $E^{'1}_{p,q}$ for the previous computations).\\
Since we have a one row spectral sequence, then $E^{'2}=...=E^{'\infty}$. Therefore:
\begin{align*}
E^{'\infty}_{p,q}(G_{C}(a,b))=
\begin{cases}
0 \quad & \text{if} \,\,\, q \neq b-2,\\
\inlinewedge_{+}^{a-p}\inlinewedge_{-}^{2}\partial_{y_{1}}^{-p}\partial_{y_{2}}^{-b+2} \,\,\,\, & \text{if} \,\,\, q=b-2.
\end{cases}
\end{align*}
We observe that the first filtration $(F'_{p} (G_{C}(a,b)))_{n}=\sum_{h\leq p}(G_{C}(a,b))_{[h,n-h]}$ is bounded below, since $F'_{n-1}=0$, and it is convergent above. Therefore by Proposition \ref{spectralbicomplex}:
\begin{align*}
\sum_{m}H^{m,n}(G_{C}(a,b))\cong \sum_{p+q=n}E^{'\infty}_{p,q}(G_{C}(a,b))=E^{'\infty}_{n-b+2,b-2}(G_{C}(a,b))\cong \displaywedge_{+}^{a+b-n-2} \displaywedge_{-}^{2}\partial_{y_{1}}^{-n+b-2}\partial_{y_{2}}^{-b+2} . 
\end{align*}
Since there are no $\partial_{x_{1}}$'s and $\partial_{x_{2}}$'s involved, this means that $H^{m,n}(G_{C}(a,b))=0$ if $m \neq 0$ and $H^{0,n}(G_{C}(a,b))=\inlinewedge_{+}^{a+b-n-2}\inlinewedge_{-}^{2}\partial_{y_{1}}^{-n+b-2}\partial_{y_{2}}^{-b+2} \cong  \inlinewedge^{a+b-n-2} $ as $\langle x_{1}\partial_{x_{1}}-x_{2}\partial_{x_{2}}, x_{1}\partial_{x_{2}},x_{2}\partial_{x_{1}}\rangle-$modules.
\end{proof}
\end{lem}
In Lemma \ref{lemmi 4.3 4.4 4.5} we computed the homology of the $G_{X^{\circ}}(a,b)$'s in the case that either $a$ or $b$ do not belong to $\left\{0,1,2\right\}$. In order to compute the homology of the $G_{X^{\circ}}(a,b)$'s in the case that both $a$ and $b$ belong to $\left\{0,1,2\right\}$, we need the following remark and lemmas.
\begin{rem}
\label{appoggio_lemma4.7}
We introduce some notation that will be used in the following lemmas. Let $0<b \leq2$.
Let us define:
\begin{align*}
\widetilde{G}_{A}(a, b)_{[p,q]}=
\begin{cases}
\inlinewedge_{+}^{a-p}\inlinewedge_{-}^{b-q}[x_{1},x_{2}] \,\,\,&\text{if} \,\,\, p\geq 0, q\geq 0,\\
0 \,\,\, &\text{otherwise}.
\end{cases}
\end{align*}
We have an isomorphism of bicomplexes $\gamma: G_{A}(a, b)_{[p,q]} \longrightarrow \widetilde{G}_{A}(a, b)_{[p,q]} $ which is the valuating map that values $y_{1}$ and $y_{2}$ in 1 and is the identity on all the other elements.
We consider on $\widetilde{G}_{A}(a, b)$ the differentials $d'=\Delta^{+}$ and $d''=\Delta^{-}$ induced by $\Delta^{+}\partial_{y_{1}}$ and $\Delta^{-}\partial_{y_{2}}$ for $G_{A}(a, b)$.
We also define:
\begin{align*}
G_{D^{'}}(a, b)_{[p,q]}=
\begin{cases}
\inlinewedge_{+}^{a+1}\inlinewedge_{-}^{b+1}[x_{1},x_{2}] \,\,\,&\text{if} \,\,\, p=q= 0,\\
0 \,\,\,&\text{otherwise}.
\end{cases}
\end{align*}
The following is a commutative diagram:\\
\begin{center}
\begin{tikzpicture}
\node[black] at (-1,0) {$G_{A}(a, b)$};
\draw[->,black] (0,0) -- (2,0); 
\node[black] at (3,0) {$G_{D^{'}}(a, b)$};
\draw[->,black] (-1,-0.5) -- (-1,-1.5);
\node[black] at (-1.5,-1) {$\gamma$};
\draw[->,black](3,-0.5) -- (3,-1.5);
\draw[->,black] (0,-2) -- (2,-2);
\node[black] at (3.5,-1) {$id$};
\node[black] at (1,0.5) {$\nabla_{2}$};
\node[black] at (1,-1.5) {$\Delta^{-}\Delta^{+}\tau_{1}$};
\node[black] at (-1,-2) {$\widetilde{G}_{A}(a, b)$};
\node[black] at (3,-2) {$G_{D^{'}}(a, b)$.};
\end{tikzpicture}
\end{center}
We have that $\widetilde{G}_{A^{\circ}}(a, b):=\Ker(\Delta^{-}\Delta^{+}\tau_{1}:\widetilde{G}_{A}(a, b) \longrightarrow G_{D^{'}}(a, b))$ is isomorphic, as a bicomplex, to $G_{A^{\circ}}(a, b)$. Its diagram is the same of $\widetilde{G}_{A}(a, b)$ except for $p=q=0$. The diagram of $\widetilde{G}_{A}(a, b)$ is the following, respectively for $a=0$, $a=1$, $a\geq 2$:
\begin{center}
\begin{tikzpicture}
\node[black] at (-7.5,0.5) {$a=0$};
\node[black] at (-7.5,0) {$\inlinewedge_{+}^{0}\inlinewedge_{-}^{0}[x_{1},x_{2}] $};
\node[black] at (-7.5,-2) {$\inlinewedge_{+}^{0}\inlinewedge_{-}^{b}[x_{1},x_{2}] $,};
\node[black] at (-7.5,-1){$\cdots$};
\draw[->,black] (-7.5,-0.4) -- (-7.5,-0.6);
\draw[->,black] (-7.5,-1.4) -- (-7.5,-1.6);
\node[black] at (-4.5,0) {$\inlinewedge_{+}^{1}\inlinewedge_{-}^{0}[x_{1},x_{2}] $};
\node[black] at (-2,0) {$\inlinewedge_{+}^{0}\inlinewedge_{-}^{0}[x_{1},x_{2}] $};
\draw[->,black] (-3.15,0) -- (-3.35,0);
\node[black] at (-4.5,-1){$\cdots$};
\node[black] at (-2,-1){$\cdots$};
\draw[->,black] (-4.5,-0.4) -- (-4.5,-0.6);
\draw[->,black] (-4.5,-1.4) -- (-4.5,-1.6);
\draw[->,black] (-2,-0.4) -- (-2,-0.6);
\draw[->,black] (-2,-1.4) -- (-2,-1.6);
\node[black] at (-3.25,0.5) {$a=1$};
\node[black] at (-4.5,-2) {$\inlinewedge_{+}^{1}\inlinewedge_{-}^{b}[x_{1},x_{2}] $};
\node[black] at (-2,-2) {$\inlinewedge_{+}^{0}\inlinewedge_{-}^{b}[x_{1},x_{2}] $,};
\draw[->,black] (-3.15,-2) -- (-3.35,-2);
\node[black] at (3.5,0.5) {$a \geq 2$};
\node[black] at (1,0) {$\inlinewedge_{+}^{2}\inlinewedge_{-}^{0}[x_{1},x_{2}] $};
\node[black] at (3.5,0){$\inlinewedge_{+}^{1}\inlinewedge_{-}^{0}[x_{1},x_{2}] $};
\node[black] at (6,0) {$\inlinewedge_{+}^{0}\inlinewedge_{-}^{0}[x_{1},x_{2}] $};
\draw[->,black] (2.35,0) -- (2.15,0);
\draw[->,black] (4.85,0) -- (4.65,0);
\node[black] at (1,-1){$\cdots$};
\node[black] at (3.5,-1){$\cdots$};
\node[black] at (6,-1){$\cdots$};
\draw[->,black] (1,-0.4) -- (1,-0.6);
\draw[->,black] (1,-1.4) -- (1,-1.6);
\draw[->,black] (3.5,-0.4) -- (3.5,-0.6);
\draw[->,black] (3.5,-1.4) -- (3.5,-1.6);
\draw[->,black] (6,-0.4) -- (6,-0.6);
\draw[->,black] (6,-1.4) -- (6,-1.6);
\node[black] at (1,-2) {$\inlinewedge_{+}^{2}\inlinewedge_{-}^{b}[x_{1},x_{2}] $};
\node[black] at (3.5,-2){$\inlinewedge_{+}^{1}\inlinewedge_{-}^{b}[x_{1},x_{2}] $};
\node[black] at (6,-2) {$\inlinewedge_{+}^{0}\inlinewedge_{-}^{b}[x_{1},x_{2}] $,};
\draw[->,black] (2.35,-2) -- (2.15,-2);
\draw[->,black] (4.85,-2) -- (4.65,-2);
\end{tikzpicture}
\end{center}
where the horizontal maps are $d'$ and the vertical maps are $d''$. The diagram of $\widetilde{G}_{A^{\circ}}(a, b)$ is analogous to this, except for $p=q=0$, where $\inlinewedge_{+}^{a}\inlinewedge_{-}^{b}[x_{1},x_{2}] $ is substituted by $\Ker(\Delta^{-}\Delta^{+}:\inlinewedge_{+}^{a}\inlinewedge_{-}^{b}[x_{1},x_{2}]  \longrightarrow \inlinewedge_{+}^{a+1}\inlinewedge_{-}^{b+1}[x_{1},x_{2}] )$, that we shortly call $\Ker(\Delta^{-}\Delta^{+})$ in the next diagram.\\ 
The $E'^{1}$ spectral sequence of $\widetilde{G}_{A^{\circ}}(a, b)$, i.e. the homology with respect to $\Delta^{-}$, is the following, respectively for $a=0$, $a=1$, $a\geq 2$, $b=1$ and $a=0$, $a=1$, $a\geq 2$, $b=2$ (the computation is analogous to Lemma \ref{lemmi 4.3 4.4 4.5}):
\begin{center}
\begin{tikzpicture}
\node[black] at (-7.5,0.7) {$a=0, b=1$};
\node[black] at (-7.5,0) {$\inlinewedge_{+}^{0}$};
\node[black] at (-7.5,-1) {$\frac{\Ker(\Delta^{-}\Delta^{+}) }{\Ima (\Delta^{-})} $,};
\draw[->,black] (-7.5,-0.3) -- (-7.5,-0.5);
\node[black] at (-3.25,0.7) {$a=1, b=1$};
\node[black] at (-4.5,0) {$\inlinewedge_{+}^{1}$};
\node[black] at (-2,0) {$\inlinewedge_{+}^{0} $};
\node[black] at (-4.5,-1) {$\frac{\Ker(\Delta^{-}\Delta^{+}) }{\Ima (\Delta^{-})} $};
\node[black] at (-2,-1) {$\frac{\inlinewedge_{+}^{0}\inlinewedge_{-}^{1}[x_{1},x_{2}]}{\Ima (\Delta^{-})},$};
\draw[->,black] (-3.15,0) -- (-3.35,0);
\draw[->,black] (-3.15,-1) -- (-3.35,-1);
\draw[->,black] (-4.5,-0.3) -- (-4.5,-0.5);
\draw[->,black] (-2,-0.3) -- (-2,-0.5);
\node[black] at (3.5,0.7) {$a \geq 2, b=1 $};
\node[black] at (1,0) {$\inlinewedge_{+}^{2}$};
\node[black] at (3.5,0) {$\inlinewedge_{+}^{1} $};
\node[black] at (6,0) {$\inlinewedge_{+}^{0} $};
\draw[->,black] (2.35,0) -- (2.15,0);
\draw[->,black] (4.85,0) -- (4.65,0);
\draw[->,black] (2.35,-1) -- (2.15,-1);
\draw[->,black] (4.85,-1) -- (4.65,-1);
\draw[->,black] (1,-0.3) -- (1,-0.5);
\draw[->,black] (3.5,-0.3) -- (3.5,-0.5);
\draw[->,black] (6,-0.3) -- (6,-0.5);
\node[black] at (1,-1) {$\frac{\inlinewedge_{+}^{2}\inlinewedge_{-}^{1}[x_{1},x_{2}] }{\Ima (\Delta^{-})} $};
\node[black] at (3.5,-1) {$\frac{\inlinewedge_{+}^{1}\inlinewedge_{-}^{1}[x_{1},x_{2}]}{\Ima (\Delta^{-})}$};
\node[black] at (6,-1) {$\frac{\inlinewedge_{+}^{0}\inlinewedge_{-}^{1}[x_{1},x_{2}]}{\Ima (\Delta^{-})}.$};
\end{tikzpicture}
\end{center}
\begin{center}
\begin{tikzpicture}
\node[black] at (-7.5,0.7) {$a=0, b=2$};
\node[black] at (-7.5,0) {$\inlinewedge_{+}^{0}$};
\node[black] at (-7.5,-2) {$\frac{\Ker(\Delta^{-}\Delta^{+}) }{\Ima (\Delta^{-})} $,};
\node[black] at (-7.5,-1){$0$};
\draw[->,black] (-7.5,-0.4) -- (-7.5,-0.6);
\draw[->,black] (-7.5,-1.3) -- (-7.5,-1.5);
\node[black] at (-3.25,0.7) {$a=1, b=2$};
\node[black] at (-4.5,0) {$\inlinewedge_{+}^{1}$};
\node[black] at (-2,0) {$\inlinewedge_{+}^{0} $};
\node[black] at (-4.5,-2) {$\frac{\Ker(\Delta^{-}\Delta^{+}) }{\Ima (\Delta^{-})} $};
\node[black] at (-2,-2) {$\frac{\inlinewedge_{+}^{0}\inlinewedge_{-}^{2}[x_{1},x_{2}]}{\Ima (\Delta^{-})},$};
\draw[->,black] (-3.15,0) -- (-3.35,0);
\draw[->,black] (-3.15,-2) -- (-3.35,-2);
\draw[->,black] (-3.15,-1) -- (-3.35,-1);
\node[black] at (-4.5,-1){$0$};
\node[black] at (-2,-1){$0$};
\draw[->,black] (-4.5,-0.4) -- (-4.5,-0.6);
\draw[->,black] (-4.5,-1.3) -- (-4.5,-1.5);
\draw[->,black] (-2,-0.4) -- (-2,-0.6);
\draw[->,black] (-2,-1.3) -- (-2,-1.5);
\node[black] at (3.5,0.7) {$a \geq 2, b=2 $};
\node[black] at (1,0) {$\inlinewedge_{+}^{2}$};
\node[black] at (3.5,0) {$\inlinewedge_{+}^{1} $};
\node[black] at (6,0) {$\inlinewedge_{+}^{0} $};
\draw[->,black] (2.35,0) -- (2.15,0);
\draw[->,black] (4.85,0) -- (4.65,0);
\node[black] at (1,-1){$0$};
\node[black] at (3.5,-1){$0$};
\node[black] at (6,-1){$0$};
\draw[->,black] (2.35,-1) -- (2.15,-1);
\draw[->,black] (4.85,-1) -- (4.65,-1);
\draw[->,black] (1,-0.4) -- (1,-0.6);
\draw[->,black] (1,-1.3) -- (1,-1.5);
\draw[->,black] (3.5,-0.4) -- (3.5,-0.6);
\draw[->,black] (3.5,-1.3) -- (3.5,-1.5);
\draw[->,black] (6,-0.4) -- (6,-0.6);
\draw[->,black] (6,-1.3) -- (6,-1.5);
\node[black] at (1,-2) {$\frac{\inlinewedge_{+}^{2}\inlinewedge_{-}^{2}[x_{1},x_{2}] }{\Ima (\Delta^{-})} $};
\node[black] at (3.5,-2) {$\frac{\inlinewedge_{+}^{1}\inlinewedge_{-}^{2}[x_{1},x_{2}]}{\Ima (\Delta^{-})}$};
\node[black] at (6,-2) {$\frac{\inlinewedge_{+}^{0}\inlinewedge_{-}^{2}[x_{1},x_{2}]}{\Ima (\Delta^{-})}.$};
\draw[->,black] (2.35,-2) -- (2.15,-2);
\draw[->,black] (4.85,-2) -- (4.65,-2);
\end{tikzpicture}
\end{center}
We have that, in the diagram of the $E'^{1}$ spectral sequence, only the rows for $q=0$ and $q=b$ are different from 0. The previous diagram will be the first step in Lemma \ref{4.6} for the computation of the homology of the $\widetilde{G}_{A^{\circ}}(a, b)$'s when $a,b \in \left\{0,1,2\right\}$.\\ 
Analogously we define,  for $0 \leq b <2$:
\begin{align*}
\widetilde{G}_{C}(a, b)_{[p,q]}=
\begin{cases}
\inlinewedge_{+}^{a-p}\inlinewedge_{-}^{b-q}[\partial_{x_{1}},\partial_{x_{2}}] \,\,\, &\text{if} \,\,\, p\leq 0, q\leq 0,\\
0 \,\,\, &\text{otherwise}.
\end{cases}
\end{align*}
We have an isomorphism of bicomplexes $\gamma: G_{C}(a, b)_{[p,q]} \longrightarrow \widetilde{G}_{C}(a, b)_{[p,q]} $ which is the valuating map that values $\partial_{y_{1}}$ and $\partial_{y_{2}}$ in 1 and is the identity on all the other elements.
We consider on $\widetilde{G}_{C}(a, b)$ the differentials $d'=\Delta^{+}$ and $d''=\Delta^{-}$ induced by $\Delta^{+}\partial_{y_{1}}$ and $\Delta^{-}\partial_{y_{2}}$ for $G_{C}(a, b)$.
We also define:
\begin{align*}
G_{B^{'}}(a, b)_{[p,q]}=
\begin{cases}
\inlinewedge_{+}^{a-1}\inlinewedge_{-}^{b-1}[\partial_{x_{1}},\partial_{x_{2}}] \,\,\, &\text{if} \,\,\, p=q= 0,\\
0 \,\,\, &\text{otherwise}.
\end{cases}
\end{align*}
We have the following commutative diagram:\\
\begin{center}
\begin{tikzpicture}
\node[black] at (-1,0) {$G_{B^{'}}(a, b)$};
\draw[->,black] (0,0) -- (2,0); 
\node[black] at (3,0) {$G_{C}(a, b)$};
\draw[->,black] (-1,-0.5) -- (-1,-1.5);
\node[black] at (-1.5,-1) {$id$};
\draw[->,black](3,-0.5) -- (3,-1.5);
\draw[->,black] (0,-2) -- (2,-2);
\node[black] at (1,0.5) {$\nabla_{2}$};
\node[black] at (3.5,-1) {$\gamma$};
\node[black] at (1,-1.5) {$\Delta^{-}\Delta^{+}\tau_{2}$};
\node[black] at (-1,-2) {$G_{B^{'}}(a, b)$};
\node[black] at (3,-2) {$\widetilde{G}_{C}(a, b)$.};
\end{tikzpicture}
\end{center}

We have that $\widetilde{G}_{C^{\circ}}(a, b):=\CoKer(\Delta^{-}\Delta^{+}\tau_{2}:G_{B^{'}}(a, b) \longrightarrow \widetilde{G}_{C}(a, b))$ is isomorphic, as a bicomplex, to $G_{C^{\circ}}(a, b)$. Its diagram is the same of $\widetilde{G}_{C}(a, b)$ except for $p=q=0$. In the following diagram we shortly write $\CoKer(\Delta^{-}\Delta^{+})$ for:
\begin{align*}
\CoKer(\Delta^{-}\Delta^{+}:\displaywedge_{+}^{a-1}\displaywedge_{-}^{b-1}[\partial_{x_{1}},\partial_{x_{2}}]\longrightarrow \displaywedge_{+}^{a}\displaywedge_{-}^{b}[\partial_{x_{1}},\partial_{x_{2}}]).
\end{align*}
The diagram of the bicomplex $\widetilde{G}_{C^{\circ}}(a, b)$ is the following, respectively for $a=2$, $a=1$ and $a \leq 0$:
\begin{center}
\begin{tikzpicture}
\node[black] at (-8,0.5) {$a = 2$};
\node[black] at (-8,0) {$\CoKer(\Delta^{-}\Delta^{+}),$};
\node[black] at (-8,-1){$\cdots$};
\draw[->,black] (-8,-0.4) -- (-8,-0.6);
\draw[->,black] (-8,-1.4) -- (-8,-1.6);
\node[black] at (-8,-2) {$\inlinewedge_{+}^{2}\inlinewedge_{-}^{2}[\partial_{x_{1}},\partial_{x_{2}}],$};
\node[black] at (-3.5,0.5) {$a = 1$};
\node[black] at (-5,0) {$\inlinewedge_{+}^{2}\inlinewedge_{-}^{b}[\partial_{x_{1}},\partial_{x_{2}}]$};
\node[black] at (-2,0) {$\CoKer(\Delta^{-}\Delta^{+}) ,$};
\draw[->,black] (-3.4,0) -- (-3.6,0);
\draw[->,black] (-3.4,-2) -- (-3.6,-2);
\node[black] at (-2,-1){$\cdots$};
\node[black] at (-5,-1){$\cdots$};
\draw[->,black] (-2,-0.4) -- (-2,-0.6);
\draw[->,black] (-2,-1.4) -- (-2,-1.6);
\draw[->,black] (-5,-0.4) -- (-5,-0.6);
\draw[->,black] (-5,-1.4) -- (-5,-1.6);
\node[black] at (-5,-2) {$\inlinewedge_{+}^{2}\inlinewedge_{-}^{2}[\partial_{x_{1}},\partial_{x_{2}}]$};
\node[black] at (-2,-2) {$\inlinewedge_{+}^{1}\inlinewedge_{-}^{2}[\partial_{x_{1}},\partial_{x_{2}}],$};

\node[black] at (4,0.5) {$a \leq 0$};
\node[black] at (1,0) {$\inlinewedge_{+}^{2}\inlinewedge_{-}^{b}[\partial_{x_{1}},\partial_{x_{2}}]$};
\node[black] at (4,0) {$\inlinewedge_{+}^{1}\inlinewedge_{-}^{b}[\partial_{x_{1}},\partial_{x_{2}}]$};
\node[black] at (7,0) {$\inlinewedge_{+}^{0}\inlinewedge_{-}^{b}[\partial_{x_{1}},\partial_{x_{2}}] $};
\draw[->,black] (2.6,0) -- (2.4,0);
\draw[->,black] (5.6,0) -- (5.4,0);
\node[black] at (1,-1){$\cdots$};
\node[black] at (4,-1){$\cdots$};
\node[black] at (7,-1){$\cdots$};
\draw[->,black] (1,-0.4) -- (1,-0.6);
\draw[->,black] (1,-1.4) -- (1,-1.6);
\draw[->,black] (4,-0.4) -- (4,-0.6);
\draw[->,black] (4,-1.4) -- (4,-1.6);
\draw[->,black] (7,-0.4) -- (7,-0.6);
\draw[->,black] (7,-1.4) -- (7,-1.6);
\node[black] at (1,-2) {$\inlinewedge_{+}^{2}\inlinewedge_{-}^{2}[\partial_{x_{1}},\partial_{x_{2}}]$};
\node[black] at (4,-2) {$\inlinewedge_{+}^{1}\inlinewedge_{-}^{2}[\partial_{x_{1}},\partial_{x_{2}}],$};
\node[black] at (7,-2) {$\inlinewedge_{+}^{0}\inlinewedge_{-}^{2}[\partial_{x_{1}},\partial_{x_{2}}],$};
\draw[->,black] (2.6,-2) -- (2.4,-2);
\draw[->,black] (5.6,-2) -- (5.4,-2);
\end{tikzpicture}
\end{center}
where the horizontal maps are $d'$ and the vertical maps are $d''$.\\
In the following diagram we shortly write $\Ker(\Delta^{-})_{i,j}$ for:
\begin{align*}
\Ker(\Delta^{-}:\displaywedge_{+}^{i}\displaywedge_{-}^{j}[\partial_{x_{1}},\partial_{x_{2}}] \longrightarrow \displaywedge_{+}^{i}\displaywedge_{-}^{j+1}[\partial_{x_{1}},\partial_{x_{2}}] ),
\end{align*}
and we shortly write $\frac{\Ker(\Delta^{-}) }{\Ima (\Delta^{-}\Delta^{+})}$ for:
\begin{align*}
\frac{\Ker(\Delta^{-}:\inlinewedge_{+}^{a}\inlinewedge_{-}^{b}[\partial_{x_{1}},\partial_{x_{2}}]\longrightarrow   \inlinewedge_{+}^{a}\inlinewedge_{-}^{b+1}[\partial_{x_{1}},\partial_{x_{2}}])}{\Delta^{-}\Delta^{+}:\inlinewedge_{+}^{a-1}\inlinewedge_{-}^{b-1}[\partial_{x_{1}},\partial_{x_{2}}]\longrightarrow \inlinewedge_{+}^{a}\inlinewedge_{-}^{b}[\partial_{x_{1}},\partial_{x_{2}}]}.
\end{align*}
The $E'^{1}$ spectral sequence of $\widetilde{G}_{C^{\circ}}(a, b)$ is the following, respectively for $a=2$, $a=1$, $a \leq 0$, $b=1$ and $a=2$, $a=1$, $a \leq 0$, $b=0$ (the computation is analogous to Lemma \ref{lemmi 4.3 4.4 4.5}):
\begin{center}
\begin{tikzpicture}
\node[black] at (-8,0.7) {$a =2, b=1$};
\node[black] at (-8,0) {$\frac{\Ker(\Delta^{-}) }{\Ima (\Delta^{-}\Delta^{+})} $};
\draw[->,black] (-8,-0.4) -- (-8,-0.6);
\node[black] at (-8,-1) {$\inlinewedge_{+}^{2}\inlinewedge_{-}^{2}$,};
\node[black] at (-3.5,0.7) {$a =1, b=1$};
\node[black] at (-5,0) {$\Ker(\Delta^{-})_{2,1}$};
\node[black] at (-2,0) {$\frac{\Ker(\Delta^{-}) }{\Ima (\Delta^{-}\Delta^{+})} $};
\draw[->,black] (-3.4,0) -- (-3.6,0);
\draw[->,black] (-3.4,-1) -- (-3.6,-1);
\draw[->,black] (-2,-0.4) -- (-2,-0.6);
\draw[->,black] (-5,-0.4) -- (-5,-0.6);
\node[black] at (-5,-1) {$\inlinewedge_{+}^{2}\inlinewedge_{-}^{2}$};
\node[black] at (-2,-1) {$\inlinewedge_{+}^{1}\inlinewedge_{-}^{2},$};
\node[black] at (4,0.7) {$a \leq 0, b=1$};
\node[black] at (1,0) {$\Ker(\Delta^{-})_{2,1}$};
\node[black] at (4,0) {$\Ker(\Delta^{-})_{1,1} $};
\node[black] at (7,0) {$\Ker(\Delta^{-})_{0,1} $};
\draw[->,black] (2.6,0) -- (2.4,0);
\draw[->,black] (5.6,0) -- (5.4,0);
\draw[->,black] (2.6,-1) -- (2.4,-1);
\draw[->,black] (5.6,-1) -- (5.4,-1);
\draw[->,black] (1,-0.4) -- (1,-0.6);
\draw[->,black] (4,-0.4) -- (4,-0.6);
\draw[->,black] (7,-0.4) -- (7,-0.6);
\node[black] at (1,-1) {$\inlinewedge_{+}^{2}\inlinewedge_{-}^{2}$};
\node[black] at (4,-1) {$\inlinewedge_{+}^{1}\inlinewedge_{-}^{2}$};
\node[black] at (7,-1) {$\inlinewedge_{+}^{0}\inlinewedge_{-}^{2}$.};
\end{tikzpicture}
\end{center}
\begin{center}
\begin{tikzpicture}
\node[black] at (-8,0.7) {$a =2, b=0$};
\node[black] at (-8,0) {$\frac{\Ker(\Delta^{-}) }{\Ima (\Delta^{-}\Delta^{+})} $};
\node[black] at (-8,-1){$0$};
\draw[->,black] (-8,-0.4) -- (-8,-0.6);
\draw[->,black] (-8,-1.4) -- (-8,-1.6);
\node[black] at (-8,-2) {$\inlinewedge_{+}^{2}\inlinewedge_{-}^{2}$};
\node[black] at (-3.5,0.7) {$a =1, b=0$};
\node[black] at (-5,0) {$\Ker(\Delta^{-})_{2,0}$};
\node[black] at (-2,0) {$\frac{\Ker(\Delta^{-}) }{\Ima (\Delta^{-}\Delta^{+})} $};
\draw[->,black] (-3.4,0) -- (-3.6,0);
\draw[->,black] (-3.4,-1) -- (-3.6,-1);
\draw[->,black] (-3.4,-2) -- (-3.6,-2);
\node[black] at (-2,-1){$0$};
\node[black] at (-5,-1){$0$};
\draw[->,black] (-2,-0.4) -- (-2,-0.6);
\draw[->,black] (-2,-1.4) -- (-2,-1.6);
\draw[->,black] (-5,-0.4) -- (-5,-0.6);
\draw[->,black] (-5,-1.4) -- (-5,-1.6);
\node[black] at (-5,-2) {$\inlinewedge_{+}^{2}\inlinewedge_{-}^{2}$};
\node[black] at (-2,-2) {$\inlinewedge_{+}^{1}\inlinewedge_{-}^{2},$};
\node[black] at (4,0.7) {$a \leq 0, b=0$};
\node[black] at (1,0) {$\Ker(\Delta^{-})_{2,0}$};
\node[black] at (4,0) {$\Ker(\Delta^{-})_{1,0} $};
\node[black] at (7,0) {$\Ker(\Delta^{-})_{0,0} $};
\draw[->,black] (2.6,0) -- (2.4,0);
\draw[->,black] (5.6,0) -- (5.4,0);
\draw[->,black] (2.6,-1) -- (2.4,-1);
\draw[->,black] (5.6,-1) -- (5.4,-1);
\draw[->,black] (2.6,-2) -- (2.4,-2);
\draw[->,black] (5.6,-2) -- (5.4,-2);
\node[black] at (1,-1){$0$};
\node[black] at (4,-1){$0$};
\node[black] at (7,-1){$0$};
\draw[->,black] (1,-0.4) -- (1,-0.6);
\draw[->,black] (1,-1.4) -- (1,-1.6);
\draw[->,black] (4,-0.4) -- (4,-0.6);
\draw[->,black] (4,-1.4) -- (4,-1.6);
\draw[->,black] (7,-0.4) -- (7,-0.6);
\draw[->,black] (7,-1.4) -- (7,-1.6);
\node[black] at (1,-2) {$\inlinewedge_{+}^{2}\inlinewedge_{-}^{2}$};
\node[black] at (4,-2) {$\inlinewedge_{+}^{1}\inlinewedge_{-}^{2}$};
\node[black] at (7,-2) {$\inlinewedge_{+}^{0}\inlinewedge_{-}^{2}$.};
\end{tikzpicture}
\end{center}
We have that only the rows $q=0$ and $q=b-2$ are different from 0. We point out that, since $b<2$:
\begin{align*}
 \frac{\Ker(\Delta^{-})}{ \Ima (\Delta^{-}\Delta^{+})} \cong  \frac{\Delta^{-}(\inlinewedge_{+}^{a}\inlinewedge_{-}^{b-1}[\partial_{x_{1}},\partial_{x_{2}}])}{\Delta^{-}\Delta^{+}(\inlinewedge_{+}^{a-1}\inlinewedge_{-}^{b-1}[\partial_{x_{1}},\partial_{x_{2}}])}\cong \CoKer(\Delta^{-}( \displaywedge_{+}^{a-1}\displaywedge_{-}^{b-1}[\partial_{x_{1}},\partial_{x_{2}}])\xrightarrow[]{\Delta^{+}} \Delta^{-}(\displaywedge_{+}^{a}\displaywedge_{-}^{b-1}[\partial_{x_{1}},\partial_{x_{2}}]) ).
\end{align*}
The isomorphism holds because $b<2$ and we know, by Lemma \ref{lemmi 4.3 4.4 4.5}, that
\begin{align*}
0 \xrightarrow[]{\Delta^{-}}\displaywedge_{-}^{0}\left[\partial_{x_{1}},\partial_{x_{2}}\right] \xrightarrow[]{\Delta^{-}} \displaywedge_{-}^{1}\left[\partial_{x_{1}},\partial_{x_{2}}\right] \xrightarrow[]{\Delta^{-}} \displaywedge_{-}^{2}\left[\partial_{x_{1}},\partial_{x_{2}}\right]  \xrightarrow[]{\Delta^{-}} 0
\end{align*}
is exact except for the right end.\\
 The previous diagram will be the first step in Lemma \ref{4.6} for the computation of the homology of the $\widetilde{G}_{C^{\circ}}(a, b)$'s when $a,b \in \left\{0,1,2\right\}$.
\end{rem}
The following two technical lemmas will be used in the proof of Lemma \ref{4.6} for the computation of the homology of the $G_{X^{\circ}}(a, b)$'s when $a,b \in \left\{0,1,2\right\}$.
\begin{lem}
\label{4.7}
Let $0\leq b \leq 2$. Let us consider the complex $S(a,b)$ defined as follows:
\begin{align*}
S(a,b)_{a} \xleftarrow[]{\Delta^{+}}... \xleftarrow[]{\Delta^{+}} \Delta^{-}( \displaywedge_{+}^{k}\displaywedge_{-}^{b}[x_{1},x_{2}])\xleftarrow[]{\Delta^{+}}... \xleftarrow[]{\Delta^{+}} \Delta^{-}(\displaywedge_{+}^{0}\displaywedge_{-}^{b}[x_{1},x_{2}] ),
\end{align*}
where $S(a,b)_{a}=\Ker(\Delta^{-}(\inlinewedge_{+}^{a}\inlinewedge_{-}^{b}[x_{1},x_{2}] )\xrightarrow[]{\Delta^{+}} \Delta^{-}(\inlinewedge_{+}^{a+1}\inlinewedge_{-}^{b}[x_{1},x_{2}] ))$.
The homology spaces of the complex $S(a,b)$, from left to right, are respectively isomorphic to:
\begin{align*}
H_{a}(S(a,b))\cong \displaywedge^{a+b+1},\,\,...\,, \,\,  H_{k}(S(a,b))\cong \displaywedge^{k+1+b},\,\,... \, ,\,\,  H_{0}(S(a,b))\cong \displaywedge^{b+1}.
\end{align*}
\end{lem}
\begin{proof}
We first focus on $0<b \leq 2$. In order to make the proof more clear, we show the statement for $b=1$ that is more significant; the proof for $b=2$ is analogous.
We observe that, due to the definition of $S(a,1)$, $H_{i}(S(a,1))=H_{i}(S(a+1,1))$ for $0 \leq i \leq a$. Hence it is sufficient to compute it for large $a$. We take $a >2$; for sake of simplicity, we choose $a=3$.
The complex $S(3, 1)$ reduces to:
\begin{align*}
0 \xleftarrow[]{\Delta^{+}} \Delta^{-}(\displaywedge_{+}^{2}\displaywedge_{-}^{1}[x_{1},x_{2}] )\xleftarrow[]{\Delta^{+}} \Delta^{-}( \displaywedge_{+}^{1}\displaywedge_{-}^{1}[x_{1},x_{2}]) \xleftarrow[]{\Delta^{+}} \Delta^{-}(\displaywedge_{+}^{0}\displaywedge_{-}^{1}[x_{1},x_{2}] ).
\end{align*}
In this case the thesis reduces to show that:
\begin{align*}
H_{3}(S(3,1))\cong 0, \quad H_{2}(S(3,1))\cong 0, \quad  H_{1}(S(3,1))\cong 0, \quad   H_{0}(S(3,1))  \cong \displaywedge_{+}^{2}.  
\end{align*}
We point out that the complex $S(3, 1)$ is isomorphic, via $\Delta^{-}$, to the complex:
\begin{align*}
0 \xleftarrow[]{\Delta^{+}} \frac{\inlinewedge_{+}^{2}\inlinewedge_{-}^{1}[x_{1},x_{2}] }{\Ima(\Delta^{-})} \xleftarrow[]{\Delta^{+}}  \frac{\inlinewedge_{+}^{1}\inlinewedge_{-}^{1}[x_{1},x_{2}] }{\Ima (\Delta^{-})}\xleftarrow[]{\Delta^{+}}\frac{\inlinewedge_{+}^{0}\inlinewedge_{-}^{1}[x_{1},x_{2}] }{\Ima (\Delta^{-})},
\end{align*}
that is exactly the row for $q=0$ in the diagram of the $E'^{1}$ spectral sequence of $\widetilde{G}_{A^{\circ}}(3,1)$ in Remark \ref{appoggio_lemma4.7}. In particular, since $a=3$, this is the row for $q=0$ and values of $p$ respectively $0,1,2$ and $3$ from the left to the right. The isomorphism of the two complexes follows from $b=1>0$ and the fact that, by Lemma \ref{lemmi 4.3 4.4 4.5}, we know that 
\begin{align*}
0 \xrightarrow[]{\Delta^{-}}\displaywedge_{-}^{0}\left[x_{1},x_{2}\right] \xrightarrow[]{\Delta^{-}} \displaywedge_{-}^{1}\left[x_{1},x_{2}\right] \xrightarrow[]{\Delta^{-}} \displaywedge_{-}^{2}\left[x_{1},x_{2}\right]  \xrightarrow[]{\Delta^{-}} 0
\end{align*}
is exact except for the left end.  \\
Since $E'^{2}(G_{A^{\circ}}(3,1))$ has two nonzero rows for $q=0$ and $q=1$ (see the diagram in Remark \ref{appoggio_lemma4.7}), then the differentials  $d^{r}_{p,q}$ are all zero except for $r=b+1=2$, $q=0$, $1 <p \leq 3$. Indeed $1 <p \leq 3$ follows from the fact that: $$d^{2}_{p,0}:E'^{2}_{p,0}\longrightarrow E'^{2}_{p-2,1}$$ and $E'^{2}_{p-2,1}=0$ if $p-2<0$, $E'^{2}_{p,0}=0$ if $p>3$.\\
From the fact that the homology spaces of $G_{A^{\circ}}(3,1)$ and $\widetilde{G}_{A^{\circ}}(3,1)$ are isomorphic and from Lemma \ref{lemmi 4.3 4.4 4.5}, it follows that:
\begin{align}
\label{appoggiospectral}
&\sum_{p+q=n}E'^{ \infty}_{p,q}(\widetilde{G}_{A^{\circ}}(3,1))=
\begin{cases}
0 \,\,\, &\text{if} \,\,\, n<3, \\
\inlinewedge_{+}^{1} \,\,\, &\text{if} \,\,\, n=3. 
\end{cases}
\end{align}
By \eqref{appoggiospectral} we obtain that $d^{2}_{p,0}$, for $1 <p \leq 3$, must be an isomorphism. Indeed, let us first show that $d^{2}_{p,0}$, for $1 <p \leq 3$, is surjective. We point out that:
\begin{align}
\label{keylemma4.7}
d^{2}_{p,0}: E'^{ 2}_{p,0}(\widetilde{G}_{A^{\circ}}(3,1))\longrightarrow  E'^{ 2}_{p-2,1}(\widetilde{G}_{A^{\circ}}(3,1)).
\end{align}
It is possible to show that $E'^{2}_{p-2,1}(\widetilde{G}_{A^{\circ}}(3,1))\cong \inlinewedge_{+}^{a+b-p+1}=\inlinewedge_{+}^{5-p}$ using an argument similar to Lemma \ref{lemmi 4.3 4.4 4.5}. \\
By \eqref{appoggiospectral}, we know that for $n=p-1<3$: $$\sum_{\widetilde{p}+\widetilde{q}=p-1}E'^{ \infty}_{\widetilde{p},\widetilde{q}}(\widetilde{G}_{A^{\circ}}(3,1))=0 .$$ 
Moreover $d^{r}=0$ for $r>2$ and $d^{2}_{p-2,1}=0$. Therefore $d^{2}_{p,0}$ must be surjective.\\ 
Let us see that $d^{2}_{p,0}$ is injective. If $p<3$, then $E'^{ \infty}_{p,0}(\widetilde{G}_{A^{\circ}}(3,1))=0$ since it appears in the sum $$\sum_{\widetilde{p}+\widetilde{q}=p}E'^{\infty}_{\widetilde{p},\widetilde{q}}(\widetilde{G}_{A^{\circ}}(3,1))=0,$$
by \eqref{appoggiospectral}.
 Moreover 
\begin{align}
\label{appoggiokeylemma4.7}
d^{2}_{p+2,-1}:E'^{2}_{p+2,-1}(\widetilde{G}_{A^{\circ}}(3,1))=0 \longrightarrow E'^{ 2}_{p,0}(\widetilde{G}_{A^{\circ}}(3,1))
\end{align}
 is identically 0. Hence $\Ker(d^{2}_{p,0})=0$. If $p=3$, we know, by \eqref{appoggiospectral}, that 
\begin{align}
\label{appoggiospectral2}
\sum_{\widetilde{p}+\widetilde{q}=p}E'^{ \infty}_{\widetilde{p},\widetilde{q}}(\widetilde{G}_{A^{\circ}}(3,1))\cong \displaywedge_{+}^{1}
\end{align}
and $E'^{ \infty}_{p,0}(\widetilde{G}_{A^{\circ}}(3,1))$ appears in this sum. Moreover we know that $$E'^{2}_{2,1}(\widetilde{G}_{A^{\circ}}(3,1))=E'^{\infty}_{2,1}(\widetilde{G}_{A^{\circ}}(3,1))\cong \displaywedge_{+}^{1},$$ since $d^{r}=0$, when $r>2$, $d^{2}_{4,0}=d^{2}_{2,1}=0$ and $E'^{2}_{2,1}(\widetilde{G}_{A^{\circ}}(3,1))\cong \inlinewedge_{+}^{1}$ due to an argument similar to Lemma \ref{lemmi 4.3 4.4 4.5}. The space $E'^{\infty}_{2,1}(\widetilde{G}_{A^{\circ}}(3,1))$ also appears in the sum \eqref{appoggiospectral2} and therefore we conclude that $E'^{ \infty}_{p,0}(\widetilde{G}_{A^{\circ}}(3,1))=0$. Since $d^{2}_{p+2,-1}$, given by \eqref{appoggiokeylemma4.7}, is identically 0, therefore $\Ker(d^{2}_{p,0})=0$. \\
Thus, by the fact that $d^{2}_{p,0}$ is an isomorphism, we obtain that $E'^{2}_{p,0}(\widetilde{G}_{A^{\circ}}(3,1)) \cong \inlinewedge_{+}^{5-p}$. Hence:
\begin{align*}
H_{3}(S(3,1))\cong 0, \quad H_{2}(S(3,1))\cong 0, \quad  H_{1}(S(3,1))\cong 0, \quad   H_{0}(S(3,1))  \cong \displaywedge_{+}^{2}.  
\end{align*}

We now prove the statement in the case $b=0$. Due to the definition of $S(a,0)$, $H_{i}(S(a,0))=H_{i}(S(a+1,0))$ for $0 \leq i \leq a$. Hence it is sufficient to compute it for $a=2$.
The complex $S(2,0)$ reduces to:
\begin{align*}
\Delta^{-}(\displaywedge_{+}^{2}\displaywedge_{-}^{0}[x_{1},x_{2}] )\xleftarrow[]{\Delta^{+}} \Delta^{-}( \displaywedge_{+}^{1}\displaywedge_{-}^{0}[x_{1},x_{2}]) \xleftarrow[]{\Delta^{+}} \Delta^{-}(\displaywedge_{+}^{0}\displaywedge_{-}^{0}[x_{1},x_{2}] ).
\end{align*}
In this case the thesis reduces to show that:
\begin{align*}
H_{2}(S(2,0))\cong 0, \quad  H_{1}(S(2,0))\cong \displaywedge_{+}^{2}, \quad   H_{0}(S(2,0))  \cong \displaywedge_{+}^{1}.  
\end{align*}
We compute the homology spaces by direct computations.
	Let us compute $H_{0}(S(2,0))$. We take $p(x_{1},x_{2}) \in \inlinewedge_{+}^{0}\inlinewedge_{-}^{0}[x_{1},x_{2}]$; an element in $\Delta^{-}(\inlinewedge_{+}^{0}\inlinewedge_{-}^{0}[x_{1},x_{2}] )$ has the following form:
	\begin{align*}
	P:=w_{12}\otimes \partial_{x_{1}}p+w_{22}\otimes \partial_{x_{2}}p.
	\end{align*}
	Hence:
	\begin{align*}
	\Delta^{+}(P)=w_{12}w_{11}\otimes \partial^{2}_{x_{1}}p+w_{12}w_{21}\otimes \partial_{x_{1}}\partial_{x_{2}}p+w_{22}w_{11}\otimes \partial_{x_{1}}\partial_{x_{2}}p+w_{22}w_{21}\otimes \partial^{2}_{x_{2}}p.
	\end{align*}
	Therefore $P$ lies in the kernel if and only if $\partial^{2}_{x_{1}}p=\partial_{x_{1}}\partial_{x_{2}}p=\partial^{2}_{x_{2}}p=0$, that is $p= \alpha x_{1}+\beta x_{2}$, for $\alpha,\beta \in \C$. Thus $H_{0}(S(2,0))\cong \inlinewedge^{1}$.\\
	Let us now compute $H_{1}(S(2,0))$. We take $w_{11}p(x_{1},x_{2})+w_{21}q(x_{1},x_{2}) \in \inlinewedge_{+}^{1}\inlinewedge_{-}^{0}[x_{1},x_{2}]$; an element in $\Delta^{-}(\inlinewedge_{+}^{1}\inlinewedge_{-}^{0}[x_{1},x_{2}] )$ has the following form:
	\begin{align*}
	P:=w_{11}w_{12}\otimes \partial_{x_{1}}p+w_{11}w_{22}\otimes \partial_{x_{2}}p+w_{21}w_{12}\otimes \partial_{x_{1}}q+w_{21}w_{22}\otimes \partial_{x_{2}}q.
	\end{align*}
	Hence:
	\begin{align*}
	\Delta^{+}(P)=w_{11}w_{12}w_{21}\otimes \partial_{x_{1}}\partial_{x_{2}}p+w_{11}w_{22}w_{21}\otimes \partial^{2}_{x_{2}}p+w_{21}w_{12}w_{11}\otimes \partial^{2}_{x_{1}}q+w_{21}w_{22}w_{11}\otimes  \partial_{x_{1}}\partial_{x_{2}}q.
	\end{align*}
	Therefore $P$ lies in the kernel if and only if:
	\begin{align*}
	\begin{cases}
	\partial_{x_{1}}\partial_{x_{2}}p-\partial^{2}_{x_{1}}q=0,\\
	\partial^{2}_{x_{2}}p-\partial_{x_{1}}\partial_{x_{2}}q=0.
	\end{cases}
	\end{align*}
	We obtain that:
	\begin{align*}
	\begin{cases}
	\partial_{x_{1}}q=\int \partial_{x_{1}}\partial_{x_{2}}p d x_{1}= \partial_{x_{2}}p+Q_{2}(x_{2}),\\
	\partial_{x_{1}}q= \int\partial^{2}_{x_{2}}p d x_{2}=\partial_{x_{2}}p+Q_{1}(x_{1}),
	\end{cases}
	\end{align*}
	where $Q_{1}(x_{1})$ (resp. $Q_{2}(x_{2})$) is a polynomial expression costant in $x_{2}$ (resp. costant in $x_{1}$). Therefore, if $P$ lies in the kernel then $\partial_{x_{1}}q= \partial_{x_{2}}p+\alpha$, with $\alpha \in \C$. 
	Let us consider an element of the kernel, we obtain that:
	\begin{align*}
	P=&w_{11}w_{12}\otimes \partial_{x_{1}}p+w_{11}w_{22}\otimes \partial_{x_{2}}p+w_{21}w_{12}\otimes (\partial_{x_{2}}p+\alpha)+w_{21}w_{22}\otimes \int \partial^{2}_{x_{2}}p dx_{1}\\
	=&\Delta^{+}\big(-w_{12}\otimes p-w_{22}\otimes \int \partial_{x_{2}}p dx_{1} \big)+w_{21}w_{12}\otimes \alpha=\Delta^{+}\big(\Delta^{-}\big(- \int p dx_{1}\big)\big)+w_{21}w_{12}\otimes \alpha.
	\end{align*}
	We point out that $w_{21}w_{12}\otimes \alpha$ does not lie in the image of the map $\Delta^{-}(\inlinewedge_{+}^{0}\inlinewedge_{-}^{0}[x_{1},x_{2}] )\xrightarrow[]{\Delta^{+}} \Delta^{-}( \inlinewedge_{+}^{1}\inlinewedge_{-}^{0}[x_{1},x_{2}]) $, because $w_{21}w_{12}\otimes \alpha=\Delta^{+}(-w_{12}\otimes \alpha x_{2})$ but $-w_{12}\otimes \alpha x_{2} \notin \Delta^{-}(\inlinewedge_{+}^{0}\inlinewedge_{-}^{0}[x_{1},x_{2}] )$. Thus $H_{1}(S(2,0))\cong \inlinewedge^{2}$.\\
	Finally, let us compute $H_{2}(S(2,0))$. We take $w_{11}w_{21}p(x_{1},x_{2}) \in \inlinewedge_{+}^{2}\inlinewedge_{-}^{0}[x_{1},x_{2}]$; an element in $\Delta^{-}(\inlinewedge_{+}^{2}\inlinewedge_{-}^{0}[x_{1},x_{2}] )$ has the following form:
	\begin{align*}
	P:=w_{11}w_{21}w_{12}\otimes \partial_{x_{1}}p+w_{11}w_{21}w_{22}\otimes \partial_{x_{2}}p.
	\end{align*}
	We point out that:
	\begin{align*}
	P=\Delta^{+}(-w_{11}w_{12} \otimes\int \partial_{x_{1}}p dx_{2}-w_{11}w_{22}\otimes p)=\Delta^{+}(\Delta^{-}(-w_{11} \otimes\int p  dx_{2})).
	\end{align*}
	Therefore every element of $\Delta^{-}(\inlinewedge_{+}^{2}\inlinewedge_{-}^{0}[x_{1},x_{2}] )$ lies in the image of the map $\Delta^{-}(\inlinewedge_{+}^{1}\inlinewedge_{-}^{0}[x_{1},x_{2}] )\xrightarrow[]{\Delta^{+}} \Delta^{-}( \inlinewedge_{+}^{2}\inlinewedge_{-}^{0}[x_{1},x_{2}]) $.  Thus $H_{0}(S(2,0))\cong 0$.
\end{proof}
\begin{lem}
\label{4.7T}
Let $0\leq b <2$. Let us consider the complex $T(a, b)$ defined as follows:
\begin{align}
\label{complexapp2}
 \Delta^{-}(\displaywedge_{+}^{2}\displaywedge_{-}^{b-1}[\partial_{x_{1}},\partial_{x_{2}}]) &\xleftarrow[]{\Delta^{+}} ...\xleftarrow[]{\Delta^{+}} \Delta^{-}( \displaywedge_{+}^{k}\displaywedge_{-}^{b-1}[\partial_{x_{1}},\partial_{x_{2}}])\xleftarrow[]{\Delta^{+}}... \xleftarrow[]{\Delta^{+}}\\ \nonumber
&\xleftarrow[]{\Delta^{+}}  \CoKer(\Delta^{-}( \displaywedge_{+}^{a-1}\displaywedge_{-}^{b-1}[\partial_{x_{1}},\partial_{x_{2}}])\xrightarrow[]{\Delta^{+}} \Delta^{-}(\displaywedge_{+}^{a}\displaywedge_{-}^{b-1}[\partial_{x_{1}},\partial_{x_{2}}]) ).
\end{align}
The homology spaces of the complex $T(a, b)$, from left to right, are respectively isomorphic to:
\begin{align*}
H_{2}(T(a, b))\cong\displaywedge_{+}^{b-1},\,\,...\,,\,\, H_{k}(T(a, b))\cong \displaywedge_{+}^{k+b-3} ,\,\,...\,,\,\, H_{a}(T(a, b))\cong\displaywedge_{+}^{-1+a+b-2}.
\end{align*}
\end{lem}
\begin{proof}
We first point out that the statement is obvious for $b=0$ since in this case the complex is trivial and the homology spaces are obviously trivial.\\
We now focus on $b=1$. The complex $T(a, b)$, due to its construction, has the property that $H_{i}(T(a, b))=H_{i}(T(a-1, b))$ for $a \leq i \leq 2$; then we can compute the homology for small $a$. Let us take $a <0$. For sake of simplicity we focus on $a =-1$. The complex $T(-1,1)$ reduces to:
\begin{align}
\label{complexapp2}
 \Delta^{-}(\displaywedge_{+}^{2}\displaywedge_{-}^{0}[\partial_{x_{1}},\partial_{x_{2}}]) \xleftarrow[]{\Delta^{+}} \Delta^{-}( \displaywedge_{+}^{1}\displaywedge_{-}^{0}[\partial_{x_{1}},\partial_{x_{2}}])\xleftarrow[]{\Delta^{+}} \Delta^{-}( \displaywedge_{+}^{0}\displaywedge_{-}^{0}[\partial_{x_{1}},\partial_{x_{2}}])\xleftarrow[]{\Delta^{+}} 0.
\end{align}
The thesis reduces to show that:
\begin{align*}
H_{2}(T(-1,1))\cong\displaywedge_{+}^{0},\,\, H_{1}(T(-1,1))\cong 0,\,\, H_{0}(T(-1,1))\cong 0 ,\,\, H_{-1}(T(-1,1))\cong 0.
\end{align*}
In order to prove the thesis, we use that the complex $T(-1, 1)$ is isomorphic, via $\Delta^{-}$, to the row for $q=0$ in the diagram of the $E'^{1}$ spectral sequence of $\widetilde{G}_{C^{\circ}}(-1, 1)$ in Remark \ref{appoggio_lemma4.7}, that is:
\begin{align*}
\Ker(\Delta^{-})_{2,1}   \xleftarrow[]{\Delta^{+}} \Ker(\Delta^{-})_{1,1} \xleftarrow[]{\Delta^{+}} \Ker(\Delta^{-})_{0,1}  \xleftarrow[]{\Delta^{+}} 0,
\end{align*}
where we shortly write $\Ker(\Delta^{-})_{i,j}$ for:
\begin{align*}
\Ker(\Delta^{-}:\displaywedge_{+}^{i}\displaywedge_{-}^{j}[\partial_{x_{1}},\partial_{x_{2}}] \longrightarrow \displaywedge_{+}^{i}\displaywedge_{-}^{j+1}[\partial_{x_{1}},\partial_{x_{2}}] ).
\end{align*}
We point out that in this case, the spaces $\Ker(\Delta^{-})_{2,1}$, $\Ker(\Delta^{-})_{1,1}$ and $\Ker(\Delta^{-})_{0,1}$ correspond respectively to the valus of $p=-3,-2,-1$ and $q=0$ in the diagram of the $E'^{1}$ spectral sequence of $\widetilde{G}_{C^{\circ}}(-1, 1)$ (see Remark \ref{appoggio_lemma4.7}). The isomorphism between the two complexes follows from $b=1<2$ and the fact that, by Lemma \ref{lemmi 4.3 4.4 4.5}, we know that 
\begin{align*}
0 \xrightarrow[]{\Delta^{-}}\displaywedge_{-}^{0}\left[\partial_{x_{1}},\partial_{x_{2}}\right] \xrightarrow[]{\Delta^{-}} \displaywedge_{-}^{1}\left[\partial_{x_{1}},\partial_{x_{2}}\right] \xrightarrow[]{\Delta^{-}} \displaywedge_{-}^{2}\left[\partial_{x_{1}},\partial_{x_{2}}\right]  \xrightarrow[]{\Delta^{-}} 0
\end{align*}
is exact except for the right end.  In this case the complex $E'^{1}$ of $\widetilde{G}_{C^{\circ}}(-1,1)$ has two nonzero rows, for $q=0$ and $q=b-2=-1$, and therefore the differentials $d^{r}_{p,q}$ are all zero except for $r=2$, $q=-1$ and $-2<p\leq 0$. Indeed:
\begin{align*}
d^{2}_{p,-1}:E'^{2}_{p,-1}\longrightarrow E'^{2}_{p-2,0},
\end{align*}
where $E'^{2}_{p,-1}=0$ if $p> 0$ and $E'^{2}_{p-2,0}=0$ if $p-2<-3$. We know, by Lemma \ref{lemmi 4.3 4.4 4.5}, that:
\begin{align}
\label{appoggiocomplex3}
\sum_{p+q=n}E'^{\infty}_{p,q}(\widetilde{G}_{C^{\circ}}(-1,1))=
\begin{cases}
0 \,\,\, &\text{if} \,\,\, n>-3 \\
\inlinewedge_{+}^{1}  \,\,\, &\text{if} \,\,\, n=-3.
\end{cases}
\end{align}
By \eqref{appoggiocomplex3} we obtain that $d^{r}_{p,q}$ for $r=2$, $q=-1$ and $-2<p\leq 0$ must be an isomorphism. Indeed, let us first show that $d^{2}_{p,q}$ for $q=-1$ and $-2<p\leq 0$ is injective. We point out that:
\begin{align}
\label{key4.7T}
d^{2}_{p,-1}: E'^{2}_{p,-1}(\widetilde{G}_{C^{\circ}}(-1,1))\longrightarrow  E'^{2}_{p-2,0}(\widetilde{G}_{C^{\circ}}(-1,1)).
\end{align}
It is possible to show that $E'^{2}_{p,-1}(\widetilde{G}_{C^{\circ}}(-1,1))\cong \inlinewedge_{+}^{-p-1}$ using an argument similar to Lemma \ref{lemmi 4.3 4.4 4.5}. We know, by \eqref{appoggiocomplex3}, that for $n=p-1>-3$:
\begin{align*}
\sum_{\widetilde{p}+\widetilde{q}=p-1}E'^{ \infty}_{\widetilde{p},\widetilde{q}}(\widetilde{G}_{C^{\circ}}(-1,1))=0.
\end{align*} 
Hence $E'^{\infty}_{p,-1}(\widetilde{G}_{C^{\circ}}(-1,1))=0$. Moreover $d^{r}=0$ for $r> 2$ and $d^{2}_{p+2,-2}=0$ since its domain is 0. Therefore $d^{2}_{p,-1}$ must be injective. \\
Let us show that $d^{2}_{p,-1}$ is surjective. If $p-2>-3$, then $E^{'\infty}_{p-2,0}(\widetilde{G}_{C^{\circ}}(-1,1))$ appears in the sum
\begin{align*} 
\sum_{\tilde{p}+\tilde{q}=p-2}E^{' \infty}_{\tilde{p},\tilde{q}}(\widetilde{G}_{C^{\circ}}(-1,1))=0,
\end{align*} 
by \eqref{appoggiocomplex3}. Therefore $E^{'\infty}_{p-2,0}(\widetilde{G}_{C^{\circ}}(-1,1))=0$. But we know that 
\begin{align}
\label{appkey4.7T}
d^{2}_{p-2,0}:E^{'2}_{p-2,0}(\widetilde{G}_{C^{\circ}}(-1,1))\longrightarrow E^{'2}_{p-4,1}(\widetilde{G}_{C^{\circ}}(-1,1))=0
\end{align}  
is identically 0 because the codomain is 0. Hence $d^{2}_{p,-1}$ must be surjective. If $p-2=-3$, then $E^{'\infty}_{p-2,0}(\widetilde{G}_{C^{\circ}}(-1,1))$ appears in the sum 
\begin{align}
\label{appoggioT4} 
\sum_{\tilde{p}+\tilde{q}=p-2}E^{' \infty}_{\tilde{p},\tilde{q}}(\widetilde{G}_{C^{\circ}}(-1,1))=\displaywedge_{+}^{1},
\end{align} 
by \eqref{appoggiocomplex3}. We know that $E^{'2}_{-2,-1}(\widetilde{G}_{C^{\circ}}(-1,1))=E^{'\infty}_{-2,-1}(\widetilde{G}_{C^{\circ}}(-1,1))\cong \inlinewedge_{+}^{1}$, since $d^{r}=0$, when $r>2$, $d^{2}_{0,-2}=d^{2}_{-2,-1}=0$ and $E^{'2}_{-2,-1}(\widetilde{G}_{C^{\circ}}(-1,1))\cong \inlinewedge_{+}^{1}$ due to an argument similar to Lemma \ref{lemmi 4.3 4.4 4.5}.
 Since $E^{'\infty}_{-2,-1}(\widetilde{G}_{C^{\circ}}(-1,1))$ also appears in the sum \eqref{appoggioT4}, we conclude that $E^{'\infty}_{p-2,0}(\widetilde{G}_{C^{\circ}}(-1,1))=0$.\\
  Since $d^{2}_{p-2,0}$, given by \eqref{appkey4.7T}, is identically 0, therefore $d^{2}_{p,-1}$ must be surjective. Hence, by the fact that $d^{2}_{p,-1}$ is an isomorphism, we obtain that $E^{'2}_{p-2,0}(\widetilde{G}_{C^{\circ}}(-1,1))\cong \inlinewedge_{+}^{-1-p}.$ Thus $E^{'2}_{s,0}(\widetilde{G}_{C^{\circ}}(-1,1)) \cong \inlinewedge_{+}^{-s-3}$ and we obtain that:
	\begin{align*}
H_{2}(T(-1,1))\cong\displaywedge_{+}^{0},\,\, H_{1}(T(-1,1))\cong 0,\,\, H_{0}(T(-1,1))\cong 0,\,\, H_{-1}(T(-1,1))\cong 0.
\end{align*}
\end{proof} 
Now using Remark \ref{appoggio_lemma4.7} and Lemmas \ref{4.7}, \ref{4.7T}, we are able to compute the homology of the $G_{X^{\circ}}(a, b)$'s when $a,b \in \left\{0,1,2\right\}$.
\begin{lem}
\label{4.6} 
If $0 \leq a \leq b \leq 2$ then, as $\langle x_{1}\partial_{x_{1}}-x_{2}\partial_{x_{2}}, x_{1}\partial_{x_{2}},x_{2}\partial_{x_{1}}\rangle-$modules:
\begin{align*}
H^{m,n}(G_{A^{\circ}}(a,b))&\cong
\begin{cases}
\inlinewedge^{a+b-n}  \, \,\,\,  &\text{if} \, \,\, m=0, \, \, \, n \geq b,\\
\inlinewedge^{a+b-n+1}  \, \,\,\,  &\text{if} \, \,\, m=1,  \, \, \, 0 \leq n \leq a,\\
0 \, \,\,\, &\text{otherwise};
\end{cases}\\
H^{m,n}(G_{D^{\circ}}(a,b))&\cong
\begin{cases}
\inlinewedge^{a+b-n}  \, \,\,\,  &\text{if} \, \,\, m=0, \, \, \, n \leq 0,\\
0 \,\,\,\, &\text{otherwise}.
\end{cases}
\end{align*}
If $0 \leq b \leq a \leq 2$ then, as $\langle x_{1}\partial_{x_{1}}-x_{2}\partial_{x_{2}}, x_{1}\partial_{x_{2}},x_{2}\partial_{x_{1}}\rangle -$modules:
\begin{align*}
H^{m,n}(G_{C^{\circ}}(a,b))\cong
\begin{cases}
\inlinewedge^{a+b-n-2}  \, \,\,\,  &\text{if} \, \,\, m=0, \, \, \, n \leq b -2,\\
\inlinewedge^{-1+a+b-n-2}  \, \,\,\,  &\text{if} \, \,\, m=-1,  \, \, \, a-2 \leq n \leq 0,\\
0 \, \,\,\, &\text{otherwise}.
\end{cases}
\end{align*}
Analogously if $0 \leq b \leq a \leq 2$ then, as $\langle x_{1}\partial_{x_{1}}-x_{2}\partial_{x_{2}}, x_{1}\partial_{x_{2}},x_{2}\partial_{x_{1}}\rangle-$modules:
\begin{align*}
H^{m,n}(G_{A^{\circ}}(a,b))&\cong
\begin{cases}
\inlinewedge^{a+b-n}  \, \,\,\,  &\text{if} \, \,\, m=0, \, \, \, n \geq a,\\
\inlinewedge^{a+b-n+1}  \, \,\,\,  &\text{if} \, \,\, m=1,  \, \, \, 0 \leq n \leq b,\\
0 \, \,\,\, &\text{otherwise};
\end{cases}\\
H^{m,n}(G_{D^{\circ}}(a,b))&\cong
\begin{cases}
\inlinewedge^{a+b-n}  \, \,\,\,  &\text{if} \, \,\, m=0, \, \, \, n \leq 0,\\
0 \,\,\,\, &\text{otherwise}.
\end{cases}
\end{align*}
If $0 \leq a \leq b \leq 2$ then, as $\langle x_{1}\partial_{x_{1}}-x_{2}\partial_{x_{2}}, x_{1}\partial_{x_{2}},x_{2}\partial_{x_{1}}\rangle -$modules:
\begin{align*}
H^{m,n}(G_{C^{\circ}}(a,b))\cong
\begin{cases}
\inlinewedge^{a+b-n-2}  \, \,\,\,  &\text{if} \, \,\, m=0, \, \, \, n \leq a -2,\\
\inlinewedge^{-1+a+b-n-2}  \, \,\,\,  &\text{if} \, \,\, m=-1,  \, \, \, b-2 \leq n \leq 0,\\
0 \, \,\,\, &\text{otherwise}.
\end{cases}
\end{align*}
\end{lem} 

\begin{proof}
We prove the statement in the case $0 \leq a \leq b \leq 2$ for $X=A,D$ and $0 \leq b \leq a \leq 2$ for $X=C$ using the theory of spectral sequences for bicomplexes; the case $0 \leq b \leq a \leq 2$ for $X=A,D$ and $0 \leq a \leq b \leq 2$ for $X=C$ can be proved analogously using the second spectral sequence instead of the first.\\

\textbf{Case A)} Let us first consider $G_{A^{\circ}}(0,0)=\Ker(\nabla_{2}:\inlinewedge_{+}^{0}\inlinewedge_{-}^{0}[x_{1},x_{2}] \longrightarrow   \inlinewedge_{+}^{1}\inlinewedge_{-}^{1}[x_{1},x_{2}] )$. We have that $G_{A^{\circ}}(0,0)=\C+\langle x_{1},x_{2} \rangle$, since an element $p(x_{1},x_{2}) \in \inlinewedge_{+}^{0}\inlinewedge_{-}^{0}[x_{1},x_{2}] $ lies in the kernel if and only if $\partial_{x_{1}}\partial_{x_{1}}p=\partial_{x_{1}}\partial_{x_{2}}p=\partial_{x_{2}}\partial_{x_{2}}p=0$. In this case the statement is straightforward. Indeed by $a = b=0$ we deduce that $p=q=0$. Therefore $G_{A^{\circ}}^{m,n}(0,0)=0$ when $n \neq 0$, $G_{A^{\circ}}^{1,0}(0,0)=\langle x_{1},x_{2}\rangle $, $G_{A^{\circ}}^{0,0}(0,0)=\C$ and, by the fact that 
\begin{align*}
\xrightarrow[]{\nabla } G_{A^{\circ}}^{2,1}(0,0)=0 \xrightarrow[]{\nabla } G_{A^{\circ}}^{1,0}(0,0)=\langle x_{1},x_{2}\rangle \rightarrow 0,
\end{align*}
we obtain $H^{1,0}(G_{A^{\circ}}(0,0))\cong \inlinewedge^{1}$. By the sequence
\begin{align*}
\xrightarrow[]{\nabla } G_{A^{\circ}}^{1,1}(0,0)=0 \xrightarrow[]{\nabla } G_{A^{\circ}}^{0,0}(0,0)=\C \rightarrow 0,
\end{align*}
we deduce that $H^{0,0}(G_{A^{\circ}}(0,0))\cong \inlinewedge^{0}$. 
We therefore assume $b>0$. As in Remark \ref{appoggio_lemma4.7} we consider:
\begin{align*}
\widetilde{G}_{A}(a, b)_{[p,q]}=
\begin{cases}
\inlinewedge_{+}^{a-p}\inlinewedge_{-}^{b-q}[x_{1},x_{2}] \,\,\,&\text{if} \,\,\, p\geq 0, q\geq 0,\\
0 \,\,\, &\text{otherwise}.
\end{cases}
\end{align*}
We consider on this space the differentials $d'=\Delta^{+}$ and $d''=\Delta^{-}$ induced by $\Delta^{+}\partial_{y_{1}}$ and $\Delta^{-}\partial_{y_{2}}$ for $G_{A}(a, b)$.
As in Remark \ref{appoggio_lemma4.7}, the $E'^{1}$ spectral sequence of $\widetilde{G}_{A^{\circ}}(a, b)$, i.e. the homology with respect to $\Delta^{-}$, is represented in following diagram:
\begin{center}
\begin{tikzpicture}
\node[black] at (0,0) {$\inlinewedge_{+}^{a}$};
\node[black] at (4,0) {$\inlinewedge_{+}^{0} $};
\draw[->,black] (1.45,0) -- (1.2,0);
\node[black] at (2,0){$\cdots$};
\draw[->,black] (2.6,0) -- (2.35,0);
\node[black] at (0,-1){$0$};
\node[black] at (2,-1){$\cdots$};
\node[black] at (4,-1){$0$};
\draw[->,black] (0,-0.4) -- (0,-0.6);
\draw[->,black] (0,-1.4) -- (0,-1.6);
\draw[->,black] (4,-0.4) -- (4,-0.6);
\draw[->,black] (4,-1.4) -- (4,-1.6);
\node[black] at (0,-2) {$\frac{\Ker(\Delta^{-}\Delta^{+}) }{\Ima (\Delta^{-})} $};
\node[black] at (4,-2) {$\frac{\inlinewedge_{+}^{0}\inlinewedge_{-}^{b}[x_{1},x_{2}]}{\Ima (\Delta^{-})}.$};
\draw[->,black] (1.45,-2) -- (1.2,-2);
\node[black] at (2,-2){$\cdots$};
\draw[->,black] (2.6,-2) -- (2.35,-2);
\end{tikzpicture}
\end{center}
We have that only the rows for $q=0$ and $q=b$ are different from 0. We observe that $d'$ is 0 on the row $q=b$. Moreover $d^{r}_{p,q}$ is 0 for $r \geq 2$ because either the domain or the codomain of these maps are 0, since $a \leq b$. Therefore $E'^{2}=...=E'^{\infty}$.\\
We need to compute $E'^{2}$ for the row $q=0$; for this computation we apply Lemma \ref{4.7}. We point out that the isomorphism in \eqref{keylemma4.7} of Lemma \ref{4.7} was induced by $\nabla$, that decreases the degree in $x_{1},x_{2}$ by 1.  Therefore $E'^{2}_{p,0}(\widetilde{G}_{A^{\circ}}(a,b))\cong \inlinewedge_{+}^{a+b-p+1}  $ is formed by elements with representatives of degree 1 in $x_{1},x_{2}$. 
Hence, if $0\leq n \leq a< b$:
\begin{align}
\label{degree1} 
\sum_{p+q=n} E'^{\infty}_{p,q}(\widetilde{G}_{A^{\circ}}(a,b))&=E'^{\infty}_{n, 0}(\widetilde{G}_{A^{\circ}}(a,b))=E'^{2}_{n,0}(\widetilde{G}_{A^{\circ}}(a,b))\cong \displaywedge_{+}^{a+b-n+1} \,\,\, ( degree \,\,1 \, \, in \, \, x_{1},x_{2}).
\end{align}
Hence $H^{1,n}(\widetilde{G}_{A^{\circ}}(a,b))\cong \inlinewedge_{+}^{a+b-n+1}$, if $0\leq n \leq a$. If $n \geq b >a$:
\begin{align*}
\sum_{p+q=n} E'^{\infty}_{p,q}(\widetilde{G}_{A^{\circ}}(a,b))=E'^{\infty}_{n-b, b}(\widetilde{G}_{A^{\circ}}(a,b))=E'^{2}_{n-b, b}(\widetilde{G}_{A^{\circ}}(a,b))\cong \displaywedge_{+}^{a+b-n}.
\end{align*}
Indeed in this sum there is not the possibility $(p,q)=(p,0)$ with $p \leq a < b$. We have that $H^{0,n}(\widetilde{G}_{A^{\circ}}(a,b))\cong \inlinewedge_{+}^{a+b-n}$, if $n \geq b> a$. 
For $n=a =b$ the result follows similarly.

\textbf{Case D)} 
We define:
\begin{align*}
\widetilde{G}_{D}(a, b)_{[p,q]}=
\begin{cases}
\inlinewedge_{+}^{a-p}\inlinewedge_{-}^{b-q}[x_{1},x_{2}] \,\,\, &\text{if} \,\,\, p\leq 0, q\leq 0,\\
0 \,\,\, &\text{otherwise}.
\end{cases}
\end{align*}
We have an isomorphism of bicomplexes $\gamma: G_{D}(a, b)_{[p,q]} \longrightarrow \widetilde{G}_{D}(a, b)_{[p,q]} $ which is the valuating map that values $\partial_{y_{1}}$ and $\partial_{y_{2}}$ in 1 and is the identity on all the other elements. We consider on $\widetilde{G}_{D}(a, b)$ the differentials $d'=\Delta^{+}$ and $d''=\Delta^{-}$ induced by $\Delta^{+}\partial_{y_{1}}$ and $\Delta^{-}\partial_{y_{2}}$ for $G_{D}(a, b)$.
We also define:
\begin{align*}
G_{A^{'}}(a, b)_{[p,q]}=
\begin{cases}
\inlinewedge_{+}^{a-1}\inlinewedge_{-}^{b-1}[x_{1},x_{2}] \,\,\,&\text{if}\,\,\, p=q= 0,\\
0 \,\,\, &\text{otherwise}.
\end{cases}
\end{align*}
We have the following commutative diagram:\\
\begin{center}
\begin{tikzpicture}
\node[black] at (-1,0) {$G_{A^{'}}(a, b)$};
\draw[->,black] (0,0) -- (2,0); 
\node[black] at (3,0) {$G_{D}(a, b)$};
\draw[->,black] (-1,-0.5) -- (-1,-1.5);
\node[black] at (-1.5,-1) {$id$};
\draw[->,black](3,-0.5) -- (3,-1.5);
\draw[->,black] (0,-2) -- (2,-2);
\node[black] at (3.5,-1) {$\gamma$};
\node[black] at (1,0.5) {$\nabla_{2}$};
\node[black] at (1,-1.5) {$\Delta^{-}\Delta^{+}\tau_{1}$};
\node[black] at (-1,-2) {$G_{A^{'}}(a, b)$};
\node[black] at (3,-2) {$\widetilde{G}_{D}(a, b)$.};
\end{tikzpicture}
\end{center}
We have that $\widetilde{G}_{D^{\circ}}(a, b):=\CoKer(\Delta^{-}\Delta^{+}\tau_{1}:G_{A^{'}}(a, b) \longrightarrow \widetilde{G}_{D}(a, b))$ is isomorphic, as a bicomplex, to $G_{D^{\circ}}$. Its diagram is the same of $\widetilde{G}_{D}$ except for $p=q=0$ (upper right point in the following diagram), 
where instead of $\inlinewedge_{+}^{a}\inlinewedge_{-}^{b}[x_{1},x_{2}] $ there is $\CoKer(\Delta^{-}\Delta^{+}:\inlinewedge_{+}^{a-1}\inlinewedge_{-}^{b-1}[x_{1},x_{2}]\longrightarrow \inlinewedge_{+}^{a}\inlinewedge_{-}^{b}[x_{1},x_{2}])$. \\
Moreover we observe that $G_{A^{'}}(0,0)=0$,  then $G_{D^{\circ}}(0,0)=G_{D}(0,0)$ and we can use the same argument of Lemma \ref{lemmi 4.3 4.4 4.5}. We now assume $b >0$. The diagram of $G_{D}(a, b)$ is the following:
\begin{center}
\begin{tikzpicture}
\node[black] at (0,0) {$\inlinewedge_{+}^{2}\inlinewedge_{-}^{b}[x_{1},x_{2}]$};
\node[black] at (4,0) {$\inlinewedge_{+}^{a}\inlinewedge_{-}^{b}[x_{1},x_{2}] $};
\draw[->,black] (1.45,0) -- (1.2,0);
\node[black] at (2,0){$\cdots$};
\draw[->,black] (2.6,0) -- (2.35,0);
\node[black] at (0,-1){$\cdots$};
\node[black] at (2,-1){$\cdots$};
\node[black] at (4,-1){$\cdots$};
\draw[->,black] (0,-0.4) -- (0,-0.6);
\draw[->,black] (0,-1.4) -- (0,-1.6);
\draw[->,black] (4,-0.4) -- (4,-0.6);
\draw[->,black] (4,-1.4) -- (4,-1.6);
\node[black] at (0,-2) {$\inlinewedge_{+}^{2}\inlinewedge_{-}^{2}[x_{1},x_{2}] $};
\node[black] at (4,-2) {$\inlinewedge_{+}^{a}\inlinewedge_{-}^{2}[x_{1},x_{2}]$,};
\draw[->,black] (1.45,-2) -- (1.2,-2);
\node[black] at (2,-2){$\cdots$};
\draw[->,black] (2.6,-2) -- (2.35,-2);
\end{tikzpicture}
\end{center}
where the horizontal maps are $d'$ and the vertical maps are $d''$. In the following diagram we shortly write $\frac{\Ker(\Delta^{-})}{\Ima (\Delta^{-}\Delta^{+})}  $ for the space:
\begin{align*}
\frac{\Ker(\Delta^{-}:\inlinewedge_{+}^{a}\inlinewedge_{-}^{b}[x_{1},x_{2}]\longrightarrow \inlinewedge_{+}^{a}\inlinewedge_{-}^{b+1}[x_{1},x_{2}])
}{\Ima (\Delta^{-}\Delta^{+}:\inlinewedge_{+}^{a-1}\inlinewedge_{-}^{b-1}[x_{1},x_{2}]\longrightarrow \inlinewedge_{+}^{a}\inlinewedge_{-}^{b}[x_{1},x_{2}])},
\end{align*}
and $\Ker(\Delta^{-})_{i,j}$ for:
\begin{align*}
\Ker(\Delta^{-}:\displaywedge_{+}^{i}\displaywedge_{-}^{j}[x_{1},x_{2}]\longrightarrow \displaywedge_{+}^{i}\displaywedge_{-}^{j+1}[x_{1},x_{2}]).
\end{align*}
The $E'^{1}$ spectral sequence of $\widetilde{G}_{D^{\circ}}(a, b)$ is (the computation is analogous to Lemma \ref{lemmi 4.3 4.4 4.5}):
\begin{center}
\begin{tikzpicture}
\node[black] at (0,0) {$\Ker(\Delta^{-})_{2,b}$};
\node[black] at (4,0) {$\frac{\Ker(\Delta^{-})}{\Ima (\Delta^{-}\Delta^{+})} $};
\draw[->,black] (1.45,0) -- (1.2,0);
\node[black] at (2,0){$\cdots$};
\draw[->,black] (2.6,0) -- (2.35,0);
\node[black] at (0,-1){$0$};
\node[black] at (2,-1){$\cdots$};
\node[black] at (4,-1){$0$};
\draw[->,black] (0,-0.4) -- (0,-0.6);
\draw[->,black] (0,-1.4) -- (0,-1.6);
\draw[->,black] (4,-0.4) -- (4,-0.6);
\draw[->,black] (4,-1.4) -- (4,-1.6);
\node[black] at (0,-2) {$0 $};
\node[black] at (4,-2) {$0$.};
\draw[->,black] (1.45,-2) -- (1.2,-2);
\node[black] at (2,-2){$0$};
\draw[->,black] (2.6,-2) -- (2.35,-2);
\end{tikzpicture}
\end{center}
We observe that, since $b>0$:
\begin{align*}
 \frac{\Ker(\Delta^{-})}{\Ima (\Delta^{-}\Delta^{+})} \cong \frac{\Delta^{-}(\inlinewedge_{+}^{a}\inlinewedge_{-}^{b-1}[x_{1},x_{2}])}{\Delta^{-}\Delta^{+}(\inlinewedge_{+}^{a-1}\inlinewedge_{-}^{b-1}[x_{1},x_{2}])}\cong \CoKer(\Delta^{-}( \displaywedge_{+}^{a-1}\displaywedge_{-}^{b-1}[x_{1},x_{2}])\xrightarrow[]{\Delta^{+}} \Delta^{-}( \displaywedge_{+}^{a}\displaywedge_{-}^{b-1}[x_{1},x_{2}]) ).
\end{align*}
The non zero row of the previous diagram is isomorphic, via $\Delta^{-}$, to the following complex:
\begin{align}
\label{complexapp}
&\Delta^{-}(\displaywedge_{+}^{2}\displaywedge_{-}^{b-1}[x_{1},x_{2}]) \xleftarrow[]{\Delta^{+}} ... \xleftarrow[]{\Delta^{+}} (\CoKer(\Delta^{-}( \displaywedge_{+}^{a-1}\displaywedge_{-}^{b-1}[x_{1},x_{2}])\xrightarrow[]{\Delta^{+}} \Delta^{-}( \displaywedge_{+}^{a}\displaywedge_{-}^{b-1}[x_{1},x_{2}]) )).
\end{align}
The fact that the two complexes are isomorphic follows from $b>0$ and that, by Lemma \ref{lemmi 4.3 4.4 4.5}, the sequence 
\begin{align*}
0 \xrightarrow[]{\Delta^{-}}\displaywedge_{-}^{0}\left[x_{1},x_{2}\right] \xrightarrow[]{\Delta^{-}} \displaywedge_{-}^{1}\left[x_{1},x_{2}\right] \xrightarrow[]{\Delta^{-}} \displaywedge_{-}^{2}\left[x_{1},x_{2}\right]  \xrightarrow[]{\Delta^{-}} 0
\end{align*}
is exact except for the left end.  
We observe that we can compute the homology of the complex \eqref{complexapp} using the homology of $S(2,b-1)$ given by Lemma \ref{4.7}. Indeed \eqref{complexapp} is different from $S(2,b-1)$ only at the right end, because the left end of \eqref{complexapp} is $\Delta^{-}(\inlinewedge_{+}^{2}\inlinewedge_{-}^{b-1}[x_{1},x_{2}])\cong \Ker(\Delta^{-}(\inlinewedge_{+}^{2}\inlinewedge_{-}^{b-1}[x_{1},x_{2}] )\xrightarrow[]{\Delta^{+}} \Delta^{-}(\inlinewedge_{+}^{3}\inlinewedge_{-}^{b-1}[x_{1},x_{2}]=0 ))$ that is the left end of $S(2,b-1)$.\\
The homology at the right end of \eqref{complexapp} is:
\begin{align*}
\Ker\Big(\Delta^{+}\Big(\frac{\Delta^{-}(\inlinewedge_{+}^{a}\inlinewedge_{-}^{b-1}[x_{1},x_{2}])}{\Delta^{-}\Delta^{+}(\inlinewedge_{+}^{a-1}\inlinewedge_{-}^{b-1}[x_{1},x_{2}])}\Big) \Big)
 \cong \frac{\Ker(\Delta^{+}(\Delta^{-}(\inlinewedge_{+}^{a}\inlinewedge_{-}^{b-1}[x_{1},x_{2}])))}{\Delta^{-}\Delta^{+}(\inlinewedge_{+}^{a-1}\inlinewedge_{-}^{b-1}[x_{1},x_{2}])} \cong H_{a}(S(2,b-1)).
\end{align*}
Therefore we can use the homology $S(2,b-1)$ and obtain that the homology spaces for the complex \eqref{complexapp} are isomorphic, respectively from left to right, to:
\begin{align*}
\displaywedge_{+}^{2+b} \quad ... \quad \displaywedge_{+}^{a+b}.
\end{align*}
We conclude because $E^{'2}_{n,0}(\widetilde{G}_{D^{\circ}}(a, b))=E^{'\infty}_{n,0}(\widetilde{G}_{D^{\circ}}(a, b))\cong \inlinewedge_{+}^{a-n+b}$ and:
\begin{align*}
\sum_{p+q=n}E^{'\infty}_{p,q}(\widetilde{G}_{D^{\circ}}(a, b))=E^{'\infty}_{n,0}(\widetilde{G}_{D^{\circ}}(a, b)) \cong \displaywedge_{+}^{a-n+b}.
\end{align*}

\textbf{Case C)} Let us first consider $G_{C^{\circ}}(2,2)=\CoKer(\nabla_{2}:\inlinewedge_{+}^{1}\inlinewedge_{-}^{1}[\partial_{x_{1}},\partial_{x_{2}} ]\longrightarrow   \inlinewedge_{+}^{2}\inlinewedge_{-}^{2}[\partial_{x_{1}},\partial_{x_{2}}] )$.
 We have that $G_{C^{\circ}}(0,0)=\C+\langle \partial_{x_{1}},\partial_{x_{2}} \rangle $, since an element $p(\partial_{x_{1}},\partial_{x_{2}}) \in \inlinewedge_{+}^{1}\inlinewedge_{-}^{1}[\partial_{x_{1}},\partial_{x_{2}}] $ is mapped to an element with degree increased by 2 in $\partial_{x_{1}},\partial_{x_{2}}$. In this case the statement is true. Indeed by $a = b=2$, we deduce $p=q=0$. Hence $G_{C^{\circ}}^{m,n}(2,2)=0$ when $n \neq 0$, $G_{C^{\circ}}^{-1,0}(0,0)=\langle \partial_{x_{1}},\partial_{x_{2}}\rangle $, $G_{C^{\circ}}^{0,0}(2,2)=\C$. By the sequence  
\begin{align*}
\xrightarrow[]{\nabla } G_{C^{\circ}}^{0,1}(2,2)=0 \xrightarrow[]{\nabla } G_{C^{\circ}}^{-1,0}(2,2)=\langle \partial_{x_{1}},\partial_{x_{2}}\rangle \rightarrow 0,
\end{align*}
we obtain $H^{-1,0}(G_{C^{\circ}}(0,0))\cong \inlinewedge^{1}$. By the sequence
\begin{align*}
\xrightarrow[]{\nabla } G_{C^{\circ}}^{1,1}(2,2)=0 \xrightarrow[]{\nabla } G_{C^{\circ}}^{0,0}(2,2)=\C \rightarrow 0,
\end{align*}
we obtain $H^{0,0}(G_{C^{\circ}}(2,2))\cong \inlinewedge^{2}$. We therefore assume $b<2$. As in Remark \ref{appoggio_lemma4.7}, we consider:
\begin{align*}
\widetilde{G}_{C}(a, b)_{[p,q]}=
\begin{cases}
\inlinewedge_{+}^{a-p}\inlinewedge_{-}^{b-q}[\partial_{x_{1}},\partial_{x_{2}}] \,\,\, &\text{if} \,\,\, p\leq 0, q\leq 0,\\
0 \,\,\, &\text{otherwise}.
\end{cases}
\end{align*}

We consider on this space the differentials $d'=\Delta^{+}$ and $d''=\Delta^{-}$ induced by $\Delta^{+}\partial_{y_{1}}$ and $\Delta^{-}\partial_{y_{2}}$ for $G_{C}(a, b)$.
  As in Remark \ref{appoggio_lemma4.7}, the $E'^{1}$ spectral sequence of $\widetilde{G}_{C^{\circ}}(a, b)$ is represented in the following diagram:
	\begin{center}
\begin{tikzpicture}
\node[black] at (0,0) {$\Ker(\Delta^{-})_{2,b} $};
\node[black] at (4,0) {$\frac{\Ker(\Delta^{-}) }{\Ima (\Delta^{-}\Delta^{+})} $};
\draw[->,black] (1.45,0) -- (1.2,0);
\node[black] at (2,0){$\cdots$};
\draw[->,black] (2.6,0) -- (2.35,0);
\node[black] at (0,-1){$0$};
\node[black] at (2,-1){$\cdots$};
\node[black] at (4,-1){$0$};
\draw[->,black] (0,-0.4) -- (0,-0.6);
\draw[->,black] (0,-1.4) -- (0,-1.6);
\draw[->,black] (4,-0.4) -- (4,-0.6);
\draw[->,black] (4,-1.4) -- (4,-1.6);
\node[black] at (0,-2) {$\inlinewedge_{+}^{2}\inlinewedge_{-}^{2} $};
\node[black] at (4,-2) {$\inlinewedge_{+}^{a}\inlinewedge_{-}^{2}$.};
\draw[->,black] (1.45,-2) -- (1.2,-2);
\node[black] at (2,-2){$\cdots$};
\draw[->,black] (2.6,-2) -- (2.35,-2);
\end{tikzpicture}
\end{center}
We have that only the rows for $q=0$ and $q=b-2$ are different from 0. We observe that $d'$ is 0 on the row $q=b-2$. Moreover $d^{r}_{p,q}$ is 0 for $r \geq 2$ because either the domain or the codomain of these maps are 0, since $2-a \leq 2-b$. Therefore $E'^{2}=...=E'^{\infty}$. We need to compute $E'^{2}$ for the row $q=0$; for this computation we apply Lemma \ref{4.7T}. 
We observe that the isomorphism in \eqref{key4.7T} of Lemma \ref{4.7T} is induced by $\nabla$ that increases the degree in $\partial_{x_{1}},\partial_{x_{2}}$ by 1. Thus the elements of $ E^{'2}_{p,0}(\widetilde{G}_{C^{\circ}}(a,b))$ are represented by elements with degree 1 in $\partial_{x_{1}},\partial_{x_{2}}$. Therefore we have that if $0 \leq b \leq a \leq 2$ and $a-2 \leq n \leq 0$:
\begin{align*}
\sum_{m}H^{m,n}(\widetilde{G}_{C^{\circ}}(a,b))=\sum_{p+q=n}E^{' \infty}_{p,q}(\widetilde{G}_{C^{\circ}}(a,b))=&E^{' \infty}_{n,0}(\widetilde{G}_{C^{\circ}}(a,b)) \cong\\
 &\displaywedge_{+}^{-1+a+b-n-2}  \quad (degree \, \,\, 1 \,\,\, in \,\,\, \partial_{x_{1}},\partial_{x_{2}}).
\end{align*}
Then $H^{-1,n}(\widetilde{G}_{C^{\circ}}(a,b))\cong \inlinewedge_{+}^{-1+a+b-n-2}$.
If $0 \leq b \leq a \leq 2$ and $n \leq b-2$:
\begin{align*}
\sum_{m}H^{m,n}(\widetilde{G}_{C^{\circ}}(a,b))=\sum_{p+q=n}E^{' \infty}_{p,q}(\widetilde{G}_{C^{\circ}}(a,b))=E^{' \infty}_{n-b+2,b-2}(\widetilde{G}_{C^{\circ}}(a,b)) \cong \displaywedge_{+}^{a+b-n-2} . 
\end{align*}
Hence $H^{0,n}(\widetilde{G}_{C^{\circ}}(a,b))\cong \inlinewedge_{+}^{a+b-n-2}$. Finally if $0 \leq b \leq a \leq 2$ and $n = b-2=a-2$, the result follows analogously.
\end{proof}
We now sum up the information of Lemmas \ref{lemmi 4.3 4.4 4.5} and \ref{4.6} in the following result about the homology of the $G_{X^{\circ}}$'s. Following \cite{kacrudakovE36}, we introduce the notation $P(n,t,c)$ that denotes the irreducible $\langle y_{1}\partial_{y_{1}}-y_{2}\partial_{y_{2}}, y_{1}\partial_{y_{2}},y_{2}\partial_{y_{1}}\rangle\oplus \C t \oplus \C C -$module of highest weight $(n,t,c)$ with respect to $y_{1}\partial_{y_{1}}-y_{2}\partial_{y_{2}}, t, C$ when $n \in \Z_{\geq 0}$ and $P(n,t,c)=0$ when $n <0$. In the following result we will use the notation $Q(i,n,t,c)$ for the irreducible $\g_{0}-$module of highest weight $(i,n,t,c)$ with respect to $x_{1}\partial_{x_{1}}-x_{2}\partial_{x_{2}},y_{1}\partial_{y_{1}}-y_{2}\partial_{y_{2}}, t, C$ when $i,n \in \Z_{\geq 0}$ and $Q(i,n,t,c)=0$ when $n <0$ or $i<0$. Moreover, for $i\in \left\{0,1,2\right\}$, we will denote by $r_{i}$ the remainder $i$ mod $2$, that is $r_{i}=0$ for $i=0,2$ and $r_{i}=1$ for $i=1$. Using Lemmas \ref{lemmi 4.3 4.4 4.5}, \ref{4.6} and the fact that $G_{X^{\circ}} =\oplus_{a,b} G_{X^{\circ}}(a,b)$, we obtain the following result.
\begin{prop} 
\label{Proppesi}
As $\g_{0}-$modules:
\begin{align*}
H^{m,n}(G_{A^{\circ}})&\cong
\begin{cases}
\sum^{2}_{i=0}Q\big(r_{i},n-i,-i-\frac{1}{2}n,-\frac{1}{2}n \big)\, \,\,\,  &if \, \,\, m=0, \, \, \, n \geq 0,\\
\sum^{2}_{i=0}Q\big(r_{i},i-n-1,-i-\frac{1}{2}n+\frac{1}{2},-\frac{1}{2}n+\frac{1}{2}\big)\, \,\,\,  &if \, \,\, m=1,  \, \, \, 0 \leq n \leq 1,\\
0 \, \,\,\, &otherwise.
\end{cases}\\
H^{m,n}(G_{D^{\circ}})&\cong
\begin{cases}
\sum^{2}_{i=0}Q\big(r_{i},-n+i,-i-\frac{1}{2}n+1,-\frac{1}{2}n+1\big) \, \,\,\,\,\,  &if \, \,\, m=0, \, \, \, n \leq 0,\\
0 \,\,\,\, \,\,&otherwise.
\end{cases}\\
H^{m,n}(G_{C^{\circ}})&\cong
\begin{cases}
\sum^{2}_{i=0}Q\big(r_{i},-n-2+i,-i-\frac{1}{2}n,-\frac{1}{2}n\big) \, \,  &if \, \,\, m=0, \, \, \, n \leq 0,\\
\sum^{2}_{i=0}Q\big(r_{i},n+2-i-1,-i-\frac{1}{2}n-\frac{1}{2},-\frac{1}{2}n-\frac{1}{2}\big)  \, \,  &if \, \,\, m=-1,  \, \, \, -1 \leq n \leq 0,\\
0 \, \, & otherwise.
\end{cases}
\end{align*}
\end{prop}
\begin{proof}
This result follows directly from Lemmas \ref{lemmi 4.3 4.4 4.5}, \ref{4.6} and the decomposition $G_{X^{\circ}}=\oplus_{a,b}G_{X^{\circ}}(a,b)$. We show explicitly the thesis for $X=A$; the proof for $X=C,D$ is analogous. By the decomposition $G_{A^{\circ}}=\oplus_{a,b}G_{A^{\circ}}(a,b)$ we obtain that:
\begin{align}
\label{apphomalfa}
H^{m,n}(G_{A^{\circ}})&=\sum_{a, b} H^{m,n}(G_{A^{\circ}}(a,b)).
\end{align}
By the definition of the $G_{A^{\circ}}(a,b)$'s, $y_{1}\partial_{y_{1}}-y_{2}\partial_{y_{2}}$ acts on the elements of $H^{m,n}(G_{A^{\circ}}(a,b))$ as multiplication by $a-b$. By Lemmas \ref{lemmi 4.3 4.4 4.5} and \ref{4.6} we obtain that the RHS of \eqref{apphomalfa} is 0 for $m>1$.\\
For $m=0$, Equation \eqref{apphomalfa} reduces to:
\begin{align}
\label{appoggiopesi}
H^{0,n}(G_{A^{\circ}})=&H^{0,n}(G_{A^{\circ}}(0,n))+H^{0,n}(G_{A^{\circ}}(1,n-1))+...+H^{0,n}(G_{A^{\circ}}(n-1,1))+H^{0,n}(G_{A^{\circ}}(n,0))\\ \nonumber
+&H^{0,n}(G_{A^{\circ}}(1,n))+H^{0,n}(G_{A^{\circ}}(2,n-1))+...+H^{0,n}(G_{A^{\circ}}(n-1,2))+H^{0,n}(G_{A^{\circ}}(n,1))\\ \nonumber
+&H^{0,n}(G_{A^{\circ}}(2,n))+H^{0,n}(G_{A^{\circ}}(3,n-1))+...+H^{0,n}(G_{A^{\circ}}(n-1,3))+H^{0,n}(G_{A^{\circ}}(n,2)).
\end{align}
We observe that the RHS of \eqref{appoggiopesi} is the sum of three irreducible $\g_{0}-$modules that we call $M_{0}$, $M_{1}$ and $M_{2}$, that are defined as follows.
As vector spaces:
\begin{align*}
M_{0}:=H^{0,n}(G_{A^{\circ}}(0,n))+H^{0,n}(G_{A^{\circ}}(1,n-1))+...+H^{0,n}(G_{A^{\circ}}(n,0) )\cong \displaywedge^{0}  \otimes P\Big(n,-\frac{1}{2}n,-\frac{1}{2}n \Big).
\end{align*}
Indeed by Lemmas \ref{lemmi 4.3 4.4 4.5} and \ref{4.6}:
\begin{center}
\begin{tikzpicture}
  \node at (0,0) (n) {$H^{0,n}(G_{A^{\circ}}(n,0))=\inlinewedge^{0}\otimes y_{1}^{n} $};
  \node at (0,-1.2) (n-1) {$H^{0,n}(G_{A^{\circ}}(n-1,1))=\inlinewedge^{0}\otimes y_{1}^{n-1}y_{2} $};
	\node at (0,-1.95)  {$\vdots$};
	\node at (0,-2.7) (1) {$H^{0,n}(G_{A^{\circ}}(1,n-1))=\inlinewedge^{0}\otimes  y_{1}y_{2}^{n-1} $};
	\node at (0,-3.9) (0) {$H^{0,n}(G_{A^{\circ}}(0,n))=\inlinewedge^{0} \otimes y_{2}^{n} .$};
	\draw[->] (n) to node[midway, right]{$y_{2}\partial_{y_{1}}$} (n-1); 
	\draw[->] (1) to node[midway, right]{$y_{2}\partial_{y_{1}}$} (0); 
\end{tikzpicture}
\end{center}
Therefore, as a $\g_{0}-$module, $M_{0}\cong  Q\Big(0,n,-\frac{1}{2}n,-\frac{1}{2}n \Big)$.
Now let us show that as vector spaces:
\begin{align*}
M_{1}:=H^{0,n}(G_{A^{\circ}}(1,n))+H^{0,n}(G_{A^{\circ}}(2,n-1))+...+H^{0,n}(G_{A^{\circ}}(n,1) )\cong \displaywedge^{1}  \otimes P\Big(n,-\frac{1}{2}n,-\frac{1}{2}n \Big).
\end{align*}
By Lemmas \ref{lemmi 4.3 4.4 4.5} and \ref{4.6}:
\begin{center}
\begin{tikzpicture}
  \node at (0,0) (n) {$H^{0,n}(G_{A^{\circ}}(n,1))\cong\inlinewedge_{-}^{1} \otimes y_{1}^{n} $};
  \node at (0,-1.2) (n-1) {$H^{0,n}(G_{A^{\circ}}(n-1,2))\cong\inlinewedge_{-}^{1} \otimes y_{1}^{n-1}y_{2} $};
	\node at (0,-1.95)  {$\vdots$};
	\node at (0,-2.7) (2) { $H^{0,n}(G_{A^{\circ}}(3,n-2))\cong\inlinewedge_{-}^{1} \otimes y_{1}^{3}y_{2}^{n-3}$};
	\node at (0,-3.9) (1) { $H^{0,n}(G_{A^{\circ}}(2,n-1))\cong\inlinewedge_{-}^{1} \otimes y_{1}^{2}y_{2}^{n-2}$};
	\node at (0,-5.1) (0) {$H^{0,n}(G_{A^{\circ}}(1,n))\cong\inlinewedge_{-}^{1} \otimes y_{1}y_{2}^{n-1}. $};
	\draw[->] (n) to node[midway, right]{$y_{2}\partial_{y_{1}}$} (n-1); 
	\draw[->] (2) to node[midway, right]{$y_{2}\partial_{y_{1}}$} (1); 
	\draw[->] (1) to node[midway, right]{$y_{2}\partial_{y_{1}}$} (0); 
\end{tikzpicture}
\end{center}
Indeed by Lemma \ref{lemmi 4.3 4.4 4.5}, $H^{0,n}(G_{A^{\circ}}(n,1))\cong \inlinewedge_{-}^{1} \otimes y_{1}^{n}=\langle w_{12}\otimes y_{1}^{n},  w_{22}\otimes y_{1}^{n}\rangle$. We point out that:
\begin{align*}
(y_{1}\partial_{y_{1}}-y_{2}\partial_{y_{2}}).(w_{12}\otimes y_{1}^{n})&=(n-1)w_{12}\otimes y_{1}^{n},\\
y_{2}\partial_{y_{1}}.(w_{12}\otimes y_{1}^{n}) &=w_{12}\otimes ny_{1}^{n-1}y_{2} \in H^{0,n}(G_{A^{\circ}}(n-1,2)),\\
y_{1}\partial_{y_{2}}.(w_{12}\otimes y_{1}^{n}) &=w_{11}\otimes y_{1}^{n}=\nabla \Big(\frac{x_{1}y_{1}^{n+1}}{n+1}\Big)=0 \quad \text{in} \quad H^{0,n}(G_{A^{\circ}}(n+1,0)),
\end{align*}
and analogously for $w_{22}\otimes y_{1}^{n}$.
Let us show explicitly that $$H^{0,n}(G_{A^{\circ}}(3,n-2))  \xrightarrow{y_{2}\partial_{y_{1}}} H^{0,n}(G_{A^{\circ}}(2,n-1)).$$ By Lemmas \ref{lemmi 4.3 4.4 4.5} and \ref{4.6}, $H^{0,n}(G_{A^{\circ}}(3,n-2))\cong\inlinewedge_{-}^{1} \otimes y_{1}^{3}y_{2}^{n-3}$ and $ H^{0,n}(G_{A^{\circ}}(2,n-1))\cong \inlinewedge_{+}^{1} \otimes y_{1}y_{2}^{n-1}$. We have that:
\begin{align*}
y_{2}\partial_{y_{1}}.(w_{12}\otimes y_{1}^{3}y_{2}^{n-3})=w_{12}\otimes 3y_{1}^{2}y_{2}^{n-2},
\end{align*}
but $w_{12}\otimes y_{1}^{2}y_{2}^{n-2}=-w_{11}\otimes \frac{2y_{1}y_{2}^{n-1}}{n-1}$ in $H^{0,n}(G_{A^{\circ}}(2,n-1))$ since:
\begin{align*}
\nabla \Big(\frac{x_{1}y_{1}^{2}y_{2}^{n-1}}{n-1} \Big)=w_{11}\otimes \frac{2y_{1}y_{2}^{n-1}}{n-1}+w_{12}\otimes y_{1}^{2}y_{2}^{n-2}.
\end{align*}
An analogous argument holds for $w_{22}\otimes y_{1}^{3}y_{2}^{n-3}$.
Finally, by Lemmas \ref{lemmi 4.3 4.4 4.5} and \ref{4.6}, $H^{0,n}(G_{A^{\circ}}(1,n))\cong\inlinewedge_{+}^{1} \otimes y_{2}^{n} \cong\inlinewedge_{-}^{1} \otimes y_{1} y_{2}^{n-1}$ and:
\begin{align*}
y_{2}\partial_{y_{1}}.(w_{12}\otimes y_{1}y_{2}^{n-1})=w_{12}\otimes y_{2}^{n}=\nabla \Big(\frac{x_{1}y_{2}^{n+1}}{n+1}\Big)=0 \quad \text{in} \quad H^{0,n}(G_{A^{\circ}}(0,n+1)),\\
y_{2}\partial_{y_{1}}.(w_{22}\otimes y_{1}y_{2}^{n-1})=w_{22}\otimes y_{2}^{n}=\nabla \Big(\frac{x_{2}y_{2}^{n+1}}{n+1}\Big)=0 \quad \text{in} \quad H^{0,n}(G_{A^{\circ}}(0,n+1)).
\end{align*}
Hence, as a $\g_{0}-$module, $M_{1} \cong Q\Big(1,n-1,-1-\frac{1}{2}n,-\frac{1}{2}n \Big)$. Finally as vector spaces:
\begin{align*}
M_{2}:=H^{0,n}(G_{A^{\circ}}(2,n))+H^{0,n}(G_{A^{\circ}}(3,n-1))+...+H^{0,n}(G_{A^{\circ}}(n,2) )\cong \displaywedge^{2}  \otimes P\Big(n,-\frac{1}{2}n,-\frac{1}{2}n \Big).
\end{align*}
Indeed using an analogous reasoning, by Lemmas \ref{lemmi 4.3 4.4 4.5} and \ref{4.6} if follows that:
\begin{center}
\begin{tikzpicture}
  \node at (0,0) (n) {$H^{0,n}(G_{A^{\circ}}(n,2))\cong\inlinewedge^{2}_{-} \otimes y_{1}^{n} $};
  \node at (0,-1.2) (n-1) {$H^{0,n}(G_{A^{\circ}}(n-1,3))\cong\inlinewedge^{2}_{-} \otimes y_{1}^{n-1}y_{2} $};
	\node at (0,-1.95)  {$\vdots$};
	\node at (0,-2.7) (1) {$H^{0,n}(G_{A^{\circ}}(3,n-1))\cong\inlinewedge^{2}_{-} \otimes y_{1}^{3}y_{2}^{n-3} $};
	\node at (0,-3.9) (0) {$H^{0,n}(G_{A^{\circ}}(2,n))\cong\inlinewedge^{2}_{-} \otimes y_{1}^{2}y_{2}^{n-2} .$};
	\draw[->] (n) to node[midway, right]{$y_{2}\partial_{y_{1}}$} (n-1); 
	\draw[->] (1) to node[midway, right]{$y_{2}\partial_{y_{1}}$} (0); 
\end{tikzpicture}
\end{center}
Therefore, as a $\g_{0}-$module, $M_{2} \cong  Q\Big(0,n-2,-2-\frac{1}{2}n,-\frac{1}{2}n \Big)$.\\
Now let us focus on $m=1$. By Lemma \ref{4.6}, $H^{1,n}(G_{A^{\circ}}(a,b))\cong \inlinewedge^{a+b-n+1} $ for $0\leq a\leq b\leq 2$ and $0\leq n\leq a$ or $0\leq b \leq a\leq 2$ and $0\leq n\leq b$ and it is 0 otherwise. Therefore $H^{1,n}(G_{A^{\circ}}(a,b))=0$ if $n\geq 2$. Indeed if $n= 2$, then $a=b=2$ and $\inlinewedge^{a+b-n+1} \cong \inlinewedge^{3}=0$. The case $n>2$ is ruled out by conditions $0\leq a\leq b\leq 2$ and $0\leq n\leq a$ or $0\leq b \leq a\leq 2$ and $0\leq n\leq b$. Hence we focus on $n=0$ and $n=1$.\\
Let $n=0$. Equation \eqref{apphomalfa} reduces to:
\begin{align}
\label{m=1n=0appoggiopesi}
H^{1,0}(G_{A^{\circ}})=H^{1,0}(G_{A^{\circ}}(0,0))+H^{1,0}(G_{A^{\circ}}(1,0))+H^{1,0}(G_{A^{\circ}}(0,1)).
\end{align}
The RHS of \eqref{m=1n=0appoggiopesi} is the sum of two irreducible $\g_{0}-$modules $M_{1}$ and $M_{2}$ that are defined as follows. We define:
\begin{align*}
M_{1}:=H^{1,0}(G_{A^{\circ}}(0,0)).
\end{align*}
By relation \eqref{degree1} in the proof of Lemma \ref{4.6},  as a $\g_{0}-$module:
\begin{align*}
H^{1,0}(G_{A^{\circ}}(0,0))\cong  Q\Big(1,0,-\frac{1}{2},\frac{1}{2}\Big)
\end{align*}
Moreover:
\begin{align*}
M_{2}:=H^{1,0}(G_{A^{\circ}}(1,0))+H^{1,0}(G_{A^{\circ}}(0,1)).
\end{align*}
By relation \eqref{degree1} in the proof of Lemma \ref{4.6}, as a $\g_{0}-$module:
\begin{align*}
H^{1,0}(G_{A^{\circ}}(1,0))+H^{1,0}(G_{A^{\circ}}(0,1))\cong Q\Big(0,1,-\frac{3}{2},\frac{1}{2}\Big)
\end{align*}
Finally, let $n=1$. Equation \eqref{apphomalfa} reduces to:
\begin{align}
\label{m=1n=0appoggiopesi2}
H^{1,1}(G_{A^{\circ}})=H^{1,1}(G_{A^{\circ}}(1,1)).
\end{align}
By relation \eqref{degree1} in the proof of Lemma \ref{4.6}, as a $\g_{0}-$module:
\begin{align*}
H^{1,1}(G_{A^{\circ}}(1,1))\cong Q\big(0,0,-2,0\big).
\end{align*}
\end{proof}
\subsection{Homology of complexes $M_{X}$}
\label{subsec5.2}
We are now able to compute the homology of the complexes  $M_{X}$'s.
\begin{prop}
\label{omologiafinale}
\begin{align*}
H^{m,n}(M_{A})&=0 \quad for \,\, all \,\,(m,n):\, m>1 \,\, \text{or} \,\, (m=1 \,\,\text{and} \,\,n\neq 1) \,\, \text{or} \,\, (m,n)=(0,1),\\
H^{m,n}(M_{C})&=0 \quad for \,\, all \,\, (m,n) \neq (0,0), (-1,-1),\\
H^{m,n}(M_{D})&=0 \quad for \,\, all \,\, (m,n).
\end{align*}
\end{prop}
\begin{proof}
By Remarks \ref{rem1Gcirc}, \ref{rem2Gcirc} and Proposition \ref{Proppesi} we know that:
\begin{align*}
H^{m,n}(G_{A})&=H^{m,n}(G_{A^{\circ}})=0 \quad \text{if} \,\, m>1 \,\, \text{or} \,\, (m=1 \,\,\text{and} \,\,n\geq 2),\\
H^{m,n}(G_{C})&=H^{m,n}(G_{C^{\circ}})=0 \quad \text{if} \,\, m<-1 \,\, \text{or} \,\, (m=-1 \,\, \text{and} \,\,n\leq -2),\\
H^{m,n}(G_{D})&=H^{m,n}(G_{D^{\circ}})=0 \quad \text{if} \,\, m>0 \,\, \text{and} \,\, n\leq 0.
\end{align*}
Therefore we obtain, by Proposition \ref{keyhomologyspectral}, that:
\begin{align*}
H^{m,n}(M_{A})&=0 \quad \text{if} \,\, m>1 \,\, \text{or} \,\, (m=1 \,\, \text{and} \,\,n\geq 2),\\
H^{m,n}(M_{C})&=0 \quad \text{if} \,\, m<-1 \,\, \text{or} \,\, (m=-1 \,\, \text{and} \,\,n\leq -2),\\
H^{m,n}(M_{D})&=0 \quad \text{if} \,\, m>0 \,\, \text{and} \,\, n\leq 0.
\end{align*}
Let us compute $H^{0,n}(G_{X^{\circ}})$ for $X=C,D$. 
By Proposition \ref{Proppesi} it follows that $H^{0,n}(G_{D^{\circ}}) \cong H^{0,n-2}(G_{C^{\circ}})$ as $\g_{0}-$modules for $n\leq 0$, indeed:
\begin{align*}
H^{0,n}(G_{D^{\circ}})  &\cong \sum^{2}_{i=0}Q\Big(r_{i},-n+i,-i-\frac{1}{2}n+1,-\frac{1}{2}n+1\Big),\\
H^{0,n-2}(G_{C^{\circ}}) &\cong \sum^{2}_{i=0} Q\Big(r_{i},-n+i,-i-\frac{1}{2}n+1,-\frac{1}{2}n+1\Big).
\end{align*}
By Remark \ref{rem1Gcirc}, we know that 
\begin{align*}
H^{0,n}(G_{D^{\circ}})&=\frac{G^{0,n}_{D}}{\Ima (\nabla : G^{1,n+1}_{D}\rightarrow G^{0,n}_{D})}  \quad \text{for} \,\,\, n\leq - 1,\\
H^{0,0}(G_{D^{\circ}})&= G^{0,0}_{D^{\circ}} ,\\
H^{0,n-2}(G_{C^{\circ}})&=\Ker(\nabla: G^{0,n-2}_{C} \rightarrow G^{-1,n-3}_{C}) .
\end{align*}
We want to show that the map induced by $\widetilde{\nabla}_{2}$ between $H^{0,n}(G_{D^{\circ}})$ and $H^{0,n-2}(G_{C^{\circ}})$, for $n\leq 0$, is an isomorphism. Indeed the kernel of the map induced by $\widetilde{\nabla}_{2}$ between $H^{0,n}(G_{D^{\circ}})$ and $H^{0,n-2}(G_{C^{\circ}})$, for $n\leq 0$, is isomorphic to 
$$\frac{\Ker (\widetilde{\nabla}_{2}:  G^{0,n}_{D} \rightarrow  G^{0,n-2}_{C})}{\Ima (\nabla : G^{1,n+1}_{D}\rightarrow G^{0,n}_{D})}=H^{0,n}(G_{D}) \quad \text{for} \,\,\, n\leq -1,$$ $$\Ker (\widetilde{\nabla}_{2}:  G^{0,0}_{D^{\circ}} \rightarrow  G^{0,-2}_{C})=H^{0,0}(G_{D}).$$
Moreover the image of the map induced by $\widetilde{\nabla}_{2}$ between $H^{0,n}(G_{D^{\circ}})$ and $H^{0,n-2}(G_{C^{\circ}})$, for $n\leq 0$, is
 $$\Ima(\widetilde{\nabla}_{2}:  G^{0,n}_{D^{\circ}} \rightarrow  G^{0,n-2}_{C^{\circ}})=\Ima(\widetilde{\nabla}_{2}:  G^{0,n}_{D} \rightarrow  G^{0,n-2}_{C}).$$
Therefore, in order to show that $H^{0,n}(M_{D})=H^{0,n-2}(M_{C})=0$ for $n\leq 0$, it is sufficient to show that the map induced by $\widetilde{\nabla}_{2}$ between $H^{0,n}(G_{D^{\circ}})$ and  $H^{0,n-2}(G_{C^{\circ}})$ is an isomorphism for $n\leq 0$. In order to do that, since $H^{0,n}(G_{D^{\circ}}) \cong H^{0,n-2}(G_{C^{\circ}})$ as $\g_{0}-$modules for $n\leq 0$, it is sufficient to show that the map induced by $\widetilde{\nabla}_{2}$ is different from 0 on the highest weight vectors of $H^{0,n}(G_{D^{\circ}})$. By Proposition \ref{Proppesi}, we know that the highest weight vectors in $H^{0,n}(G_{D^{\circ}})$ are $\partial_{y_{2}}^{-n}, w_{11}\otimes \partial_{y_{2}}^{-n},w_{11}w_{21}\otimes \partial_{y_{2}}^{-n}$; we obtain that:
\begin{align*}
&\widetilde{\nabla}_{2}(\partial_{y_{2}}^{-n})=w_{11}w_{21}\otimes \partial_{y_{1}}^{2}\partial_{y_{2}}^{-n} +w_{11}w_{22}\otimes \partial_{y_{1}}\partial_{y_{2}}^{-n+1}+w_{12}w_{21}\otimes \partial_{y_{1}}\partial_{y_{2}}^{-n+1} +w_{12}w_{22}\otimes \partial_{y_{2}}^{-n+2},\\
&\widetilde{\nabla}_{2}(w_{11}\otimes \partial_{y_{2}}^{-n})  =w_{11}w_{12}w_{21}\otimes \partial_{y_{1}}\partial_{y_{2}}^{-n+1} +w_{11}w_{12}w_{22}\otimes \partial_{y_{2}}^{-n+2},\\
&\widetilde{\nabla}_{2}(w_{11}w_{21}\otimes \partial_{y_{2}}^{-n})=w_{11}w_{21}w_{12}w_{22}\otimes \partial_{y_{2}}^{-n+2}.
\end{align*}
By Proposition \ref{Proppesi}, we have that $H^{0,1}(G_{A^{\circ}}) \cong H^{-1,0}(G_{C^{\circ}})$ as $\g_{0}-$modules, indeed:
\begin{align*}
H^{0,1}(G_{A^{\circ}})  &\cong Q\Big(0,1,-\frac{1}{2},-\frac{1}{2}\Big)+Q\Big(1,0,-\frac{3}{2},-\frac{1}{2}\Big),  \\
H^{-1,0}(G_{C^{\circ}}) &\cong Q\Big(0,1,-\frac{1}{2},-\frac{1}{2}\Big) + Q\Big(1,0,-\frac{3}{2},-\frac{1}{2}\Big).
\end{align*}
With an analogous argument, in order to obtain that $H^{0,1}(M_{A})=H^{-1,0}(M_{C})=0$, it is sufficient to show that the map induced by $\nabla_{3}$ between $H^{0,1}(G_{A^{\circ}})$ and $H^{-1,0}(G_{C^{\circ}})$ is an isomorphism.\\
We show that the map induced by $\nabla_{3}$  is different from 0 on the highest weight vectors of $H^{0,1}(G_{A^{\circ}})$. By Proposition \ref{Proppesi} we know that the highest weight vectors in $H^{0,1}(G_{A^{\circ}})$ are $y_{1}$ and $ w_{12}\otimes y_{1}$; we obtain that:
\begin{align*}
&\nabla_{3}(y_{1})=w_{11}w_{21}w_{12} \partial_{x_{1}}+w_{11}w_{21}w_{22} \partial_{x_{2}},\\
&\nabla_{3}( w_{12}\otimes  y_{1})= w_{12}w_{11}w_{21}w_{12} \partial_{x_{1}}+w_{12}w_{11}w_{21}w_{22} \partial_{x_{2}}=w_{12}w_{11}w_{21}w_{22} \partial_{x_{2}}.
\end{align*}
Finally, by Proposition \ref{Proppesi} we have that $H^{1,0}(G_{A^{\circ}}) \cong H^{0,-1}(G_{C^{\circ}})$ as $\g_{0}-$modules, indeed:
\begin{align*}
H^{1,0}(G_{A^{\circ}})  &\cong Q\Big(1,0,-\frac{1}{2},\frac{1}{2}\Big)+Q\Big(0,1,-\frac{3}{2},\frac{1}{2}\Big), \\
H^{0,-1}(G_{C^{\circ}}) &\cong Q\Big(1,0,-\frac{1}{2},\frac{1}{2}\Big)+Q\Big(0,1,-\frac{3}{2},\frac{1}{2}\Big).
\end{align*}
With an analogous argument, in order to obtain that $H^{1,0}(M_{A})=H^{0,-1}(M_{C})=0$, it is sufficient to show that the map induced by $\widetilde{\nabla}_{3}$ between $H^{1,0}(G_{A^{\circ}})$ and $H^{0,-1}(G_{C^{\circ}})$ is an isomorphism.\\
We show that the map induced by $\widetilde{\nabla}_{3}$ is different from 0 on the highest weight vectors of $H^{1,0}(G_{A^{\circ}})$. By Proposition \ref{Proppesi} we know that the highest weight vectors in $H^{1,0}(G_{A^{\circ}})$ are $x_{1}, w_{11}\otimes x_{2}-w_{22}\otimes x_{1}$; we obtain that:
\begin{align*}
&\widetilde{\nabla}_{3}(x_{1})= w_{11}w_{12}w_{21}\partial_{y_{1}}+w_{11}w_{12}w_{22}\partial_{y_{2}},\\
&\widetilde{\nabla}_{3}(w_{11}\otimes x_{2}-w_{22}\otimes x_{1})= w_{11}w_{21}w_{12}w_{22}\partial_{y_{2}}     -w_{22}w_{11}w_{12}w_{21}\partial_{y_{1}}.
\end{align*}
\end{proof}
We are now able to compute the homology for $M_{B}$ using Remark \ref{cantacasellikacremarksize}, together with Proposition \ref{esattezzafuntoreduale} and \ref{omologiafinale}.
\begin{prop}
\label{omologiafinaleduale}
\begin{align*}
H^{m,n}(M_{A})&=0 \quad for \,\, all \,\,(0,n):\, n>1; \\
H^{m,n}(M_{B})&=0 \quad for \,\, all \,\, (m,n).
\end{align*}
\end{prop}
\begin{proof}
We first compute $H^{m,n}(M_{B})$. We consider the following sequence for $m<0$ and $n>0$, represented in quadrant \textbf{B} of Figure \ref{figura}:
\begin{align*}
M_{B}^{m+1,n+1} \xrightarrow{\nabla} M_{B}^{m,n}  \xrightarrow{\nabla} M_{B}^{m-1,n-1}.
\end{align*}
By Remark \ref{cantacasellikacremarksize}, this sequence is the dual of 
\begin{align*}
M_{D}^{-m+1,-n+1} \xrightarrow{\nabla} M_{D}^{-m,-n}  \xrightarrow{\nabla} M_{D}^{-m-1,-n-1}.
\end{align*}
Indeed we recall that $M_{B}^{m,n} \cong M(-m,n,1-\frac{m+n}{2},-1+\frac{m-n}{2})$ and $ M_{D}^{-m,-n} \cong M(-m,n,1+\frac{m+n}{2},1+\frac{n-m}{2})$. By Proposition \ref{omologiafinale}, the previous sequence is exact in $M_{D}^{-m,-n}$ and $M_{D}^{-m-1,-n-1}$. Therefore $\frac{M_{D}^{-m,-n}}{\Ima \nabla} \cong \frac{M_{D}^{-m,-n}}{\Ker \nabla}$ is isomorphic to a submodule of the free module $M_{D}^{-m-1,-n-1}$. Therefore $\frac{M_{D}^{-m,-n}}{\Ima \nabla}$ is a finitely generated torsion free $\C[\Theta]-$module. The same holds for $\frac{M_{D}^{-m-1,-n-1}}{\Ima \nabla} $. Hence, by Remark \ref{cantacasellikacremarksize} and Proposition \ref{esattezzafuntoreduale}, we obtain exactness in $M_{B}^{m,n}$ for $m<0$ and $n>0$.
Now, let us consider the following sequence for $n>0$, represented from quadrant \textbf{A} to quadrant \textbf{B} in Figure \ref{figura}:
\begin{align*}
M_{A}^{0,n+2} \xrightarrow{\widetilde{\nabla}_{2}} M_{B}^{0,n}  \xrightarrow{\nabla} M_{B}^{-1,n-1}.
\end{align*}
By Remark \ref{cantacasellikacremarksize}, this sequence is the dual of 
\begin{align*}
M_{D}^{1,-n+1} \xrightarrow{\nabla} M_{D}^{0,-n}  \xrightarrow{\widetilde{\nabla}_{2}} M_{C}^{0,-n-2}.
\end{align*}
Analogously, by Proposition \ref{omologiafinale}, the previous sequence is exact in $M_{D}^{0,-n}$ and $M_{C}^{0,-n-2}$. Therefore $\frac{ M_{D}^{0,-n}}{\Ima \nabla}$ and $\frac{M_{C}^{0,-n-2}}{\Ima \widetilde{\nabla}_{2}} $ are finitely generated torsion free $\C[\Theta]-$modules. Hence, by Remark \ref{cantacasellikacremarksize} and Proposition \ref{esattezzafuntoreduale}, we obtain exactness in $M_{B}^{0,n}$ for $n>0$.\\
Using the same reasoning, we obtain exactness in $M_{B}^{0,0}$ and $M_{B}^{m,0}$, for $m<0$, using respectively the sequences $M_{A}^{0,2} \xrightarrow{\widetilde{\nabla}_{2}} M_{B}^{0,0}  \xrightarrow{\nabla_{2}} M_{C}^{-2,0}$ and $M_{B}^{m+1,1} \xrightarrow{\nabla} M_{B}^{m,0}  \xrightarrow{\nabla_{2}} M_{C}^{m-2,0}$.\\
Finally, using the same argument, we obtain exactness in $M_{A}^{0,n}$, for $n>1$, using the sequence $M_{A}^{1,n+1} \xrightarrow{\nabla} M_{A}^{0,n}  \xrightarrow{\widetilde{\nabla}_{2}} M_{B}^{0,n-2}$.
\end{proof}
Let us now focus on the remaining four cases, that are $H^{0,0}(M_{C})$, $H^{-1,-1}(M_{C})$, $H^{0,0}(M_{A})$ and $ H^{1,1}(M_{A})$.
\begin{prop}
\label{omologiaC}
\begin{align*}
H^{0,0}(M_{C})&\cong 0 ,\\
H^{-1,-1}(M_{C})&\cong \C. \\
\end{align*}
\end{prop}
In order to prove Proposition \ref{omologiaC}, we need the following results and the theory of spectral sequences. So far we have shown that $E^{0}(M_{C})^{0,0}=H^{0,0}(\Gr M_{C})=S(\g_{-2}) \otimes H^{0,0}(G_{C}) $ and $E^{0}(M_{C})^{-1,-1}=H^{-1,-1}(\Gr M_{C})=S(\g_{-2}) \otimes H^{-1,-1}(G_{C}) $ as $\mathcal{W}-$modules.
\begin{lem}
\label{lemmaxi}
Let
\begin{align*}
z=iw_{11}w_{21}\Delta^{-}\partial_{y_{1}}+(iw_{12}w_{21}+iw_{11}w_{22})\Delta^{-}\partial_{y_{2}}.
\end{align*}
be an element in $M_{C}^{-1,-1}=M(1,1,3,0)$. The following hold:
\begin{enumerate}
	\item $\nabla z=0$, 
	\item $\g_{0}.z=0$,
	\item $(t \xi_{1}+it \xi_{2}) .z \in \Ima \nabla$, $(\xi_{1}\xi_{3}\xi_{4}+i\xi_{2}\xi_{3}\xi_{4}). z \in \Ima \nabla$,
	\item $z \notin \Ima \nabla$,
\item $[z] $ is a basis for the $\g_{0}-$module $H^{-1,-1}(G_{C}) $.
\end{enumerate}
\end{lem}
\begin{proof}
	(1),(2),(3) These properties follow from direct computations. In particular $(t \xi_{1}+it \xi_{2}) .z=\nabla(2w_{22}\otimes 1)$ and $(\xi_{1}\xi_{3}\xi_{4}+i\xi_{2}\xi_{3}\xi_{4}).z=\nabla(2iw_{22} \otimes 1)$.\\
					(4) Let us show that $z \notin \Ima \nabla$. Let us consider $\nabla: M(0,0,2,0)\longrightarrow M(1,1,3,0)$. Since $M(0,0,2,0)$ is irreducible by Proposition \ref{M(0,0,2,0)} and $\nabla $ is not identically zero, then $\nabla$ is injective. Let us suppose that there exists $v \in M(0,0,2,0)$ such that $z=\nabla(v)$. Due to injectivity and (2) we obtain that $\g_{0}.v=0$. In particular $v$ has weight 0 with respect to $t$. Then:
					\begin{align*}
					v=&a_{1}\Theta \otimes 1+a_{2}w_{11}w_{22} \otimes 1 +a_{3}w_{11}w_{12} \otimes 1+a_{4}w_{11}w_{21} \otimes 1\\
					&+a_{5}w_{22}w_{12} \otimes 1+a_{6}w_{22}w_{21} \otimes 1+a_{7}w_{12}w_{21} \otimes 1.
					\end{align*}
					Imposing that $\g^{ss}_{0}.v=0$ we obtain that $a_{2}=a_{3}=a_{4}=a_{5}=a_{6}=a_{7}=0$. Hence $v=a_{1}\Theta \otimes 1.$
					We compute $\nabla v$:
					\begin{align*}
					\nabla v=&a_{1}\Theta w_{11} \otimes \partial_{x_{1}} \partial_{y_{1}}+a_{1}\Theta w_{21} \otimes \partial_{x_{2}} \partial_{y_{1}}+a_{1}\Theta w_{12} \otimes \partial_{x_{1}} \partial_{y_{2}}+a_{1}\Theta w_{22} \otimes \partial_{x_{2}} \partial_{y_{2}}.
					\end{align*}
					We focus on the coefficient of $\partial_{x_{1}} \partial_{y_{1}}$ in $\nabla v$ and $z$, that should be the same. We get $a_{1}\Theta w_{11} =i w_{11}w_{21}w_{12}$ that is impossible.\\
				(5) Let us show also that $[z] \neq 0 $ in $H^{-1,-1}(G_{C})$ because $z$ does not lie in the image of $\nabla:G_{C}^{0,0}\longrightarrow G_{C}^{-1,-1}$. Indeed if $z$ lies in the image of $\nabla:G_{C}^{0,0}\longrightarrow G_{C}^{-1,-1}$, therefore $ [z] = 0 $ in $H^{-1,-1}(M_{C})$ since $H^{-1,-1}(\Gr_{C})=S(\g_{-2})\otimes H^{-1,-1}(G_{C})$ is the first step of the spectral sequence; this holds a contradiction using (4). By Proposition \ref{Proppesi}, we know that $H^{-1,-1}(G_{C}) $ is one$-$dimensional. By the previous properties $0 \neq[z] \in H^{-1,-1}(G_{C})$; hence $[z] $ is a basis for the $\g_{0}-$module $H^{-1,-1}(G_{C})$.
\end{proof}
\begin{cor}
The vector $z$ is a highest weight singular vector in the quotient $M(1,1,3,0)/ \Ima \nabla$.
\end{cor}
\begin{lem}
\label{lemlambda}
Let
\begin{align*}
k=\frac{1}{2}iw_{11}w_{21}w_{12}w_{22}\otimes 1+i\Theta w_{12}w_{21}\otimes 1 +i \Theta w_{11}w_{22} \otimes 1
\end{align*}
be an element in $M_{C}^{0,0}=M(0,0,2,0)$. The following hold:
\begin{enumerate}
	\item $x_{1}\partial_{x_{2}}.k=0$ and $y_{1} \partial_{y_{2}}.k=0$,
\item $[k ]$ is a basis for the $\g_{0}-$module $H^{0,0}(G_{C}) $. 
\end{enumerate}
\end{lem}
\begin{proof}
It is a straightforward computation that $x_{1}\partial_{x_{2}}.k=y_{1} \partial_{y_{2}}.k=0$.
				We point out that $\nabla k$ is a cycle in $\Gr M_{C}$ since $k \in F_{4} M_{C}$ and $\nabla k \in F_{4}M_{C}$. Indeed in $M_{C}$, by direct computations, $\nabla k=\Theta z$. Moreover $[k]$ lies in $H^{0,0}(G_{C})$ since the terms of $k$ that include $\Theta$ are in $F_{3}M$, the other is in $F_{4}M_{C}$.
				By Proposition \ref{Proppesi}, we know that $H^{0,0}(G_{C}) $ is one$-$dimensional. By the previous computations $0 \neq[k] \in H^{0,0}(G_{C})$; hence $[k] $ is a basis for the $\g_{0}-$module $H^{0,0}(G_{C})$.
				We have also that $\nabla[k]=\Theta [z]$.

\end{proof}
\begin{proof}[Proof of Proposition \ref{omologiaC}]
By \eqref{keyhomology} and Lemmas \ref{lemmaxi}, \ref{lemlambda} we know that as $\mathcal{W}-$modules
\begin{gather*}
 E^{0}(M_{C})^{0,0}=H^{0,0}(\Gr M_{C})  \cong S(\g_{-2}) \otimes H^{0,0}(G_{C}) =S(\g_{-2}) \otimes \langle  [k] \rangle,\\
E^{0}(M_{C})^{-1,-1}=H^{-1,-1}(\Gr M_{C})  \cong S(\g_{-2}) \otimes H^{-1,-1}(G_{C}) =S(\g_{-2}) \otimes \langle  [z] \rangle.
\end{gather*}
By Lemma \ref{lemlambda}, the morphism $\nabla^{(0)}:E^{0}(M_{C})^{0,0}\longrightarrow E^{0}(M_{C})^{-1,-1}$ maps $[k]$ to $\Theta [z]$. Therefore $\nabla^{(0)}$ is injective and $E^{1}(M_{C})^{0,0} \cong 0$, $ E^{1}(M_{C})^{-1,-1} \cong \C$. \\
Thus $E^{\infty}(M_{C})^{0,0} \cong E^{1}(M_{C})^{0,0}=0$ and $ E^{\infty}(M_{C})^{-1,-1} \cong E^{1}(M_{C})^{-1,-1} \cong \C$ as $\mathcal{W}-$modules, and hence as $\g-$modules.
\end{proof}
Finally we focus on the two remaining cases for $M_{A}$.
\begin{prop}
\label{omologiaA}
\begin{align*}
H^{0,0}(M_{A})&\cong \C, \\
H^{1,1}(M_{A})&\cong 0. 
\end{align*}
\end{prop}
\begin{rem}
\label{apphomA}
By straightforward computation we show that $H^{0,0}(M_{A})\cong M_{A}^{0,0}/ \Ima \nabla \cong \C$. Indeed $\Ima \nabla$ is the $\g-$module generated by the singular vector $w_{11}\otimes 1$ and we have that:
\begin{gather*}
x_{2}\partial_{x_{1}}.(w_{11} \otimes 1)=w_{21}\otimes 1, \,\,y_{2}\partial_{y_{1}}.(w_{11}\otimes 1)=w_{12}\otimes 1,\\
 y_{2}\partial_{y_{1}}.(x_{2}\partial_{x_{1}}.(w_{11} \otimes 1))=w_{22}\otimes 1, \,\, w_{12}.(w_{21} \otimes 1)+w_{21}.(w_{12}\otimes 1)=-4\Theta \otimes 1.
\end{gather*}
Therefore all the elements of positive degree of $M(0,0,0,0)$ lie in the image of $\nabla$ and $M_{A}^{0,0}/ \Ima \nabla \cong F(0,0,0,0)$.
\end{rem}
In order to prove Proposition \ref{omologiaA} we need the following lemma.
\begin{lem} 
\label{lemmaesse}
Let
\begin{align*}
s=(w_{11} \otimes x_{2}-w_{21} \otimes x_{1})y_{2}-(w_{12}\otimes x_{2}-w_{22}\otimes x_{1})y_{1}
\end{align*}
be an element in $M_{A}^{1,1}=M(1,1,-1,0)$. The following hold:
\begin{enumerate}
	\item $s$ is a highest weight vector of weight (0,0,-2,0),
\item $[s]$ is a basis for the $\g_{0}-$module $H^{1,1}(G_{A}) $. 
\end{enumerate}
\end{lem}
\begin{proof}
\begin{enumerate}
	\item  This property follows by direct computations.
				\item  Let us show that $s$ is a cycle in $\Gr M_{A}$. By direct computations, using \eqref{bracketwcaselli}, $\nabla s= 8\Theta \otimes 1 \in F_{1}M_{A}.$
				Since $s\in F_{1}M_{A}$, then $\nabla[s]=0$ in $\Gr M_{A}$. By its definition, $s$ lies in $G^{1,1}_{A}$. Moreover $s$ does not lie in $\Ima (\nabla: G_{A}^{2,2}\longrightarrow G_{A}^{1,1})$. Indeed let us suppose that there exists $v \in G_{A}^{2,2}$ such that $\nabla v=s$. Since $t.s=-2s$, then $t.v=-2v$. Therefore $v \in F(2,2,-2,0)$. Imposing that $\nabla v=s$ we get a contradiction. Hence $0 \neq [s] \in H^{1,1}(G_{A})$. By Proposition \ref{Proppesi} $H^{1,1}(G_{A}) $ is one$-$dimensional; therefore $[s]$ is a basis for the $\g_{0}-$module $H^{1,1}(G_{A}) $.
	\end{enumerate}
	\end{proof}
\begin{proof}[Proof of Proposition \ref{omologiaA}]
By \eqref{keyhomology}, Remark \ref{apphomA} and Lemma \ref{lemmaesse}, we know that as $\mathcal{W}-$modules
\begin{gather*}
 E^{0}(M_{A})^{0,0}=H^{0,0}(\Gr M_{A})  \cong S(\g_{-2}) \otimes H^{0,0}(G_{A}) =S(\g_{-2}) \otimes \langle  1 \rangle,\\
E^{0}(M_{A})^{1,1}=H^{1,1}(\Gr M_{A})  \cong S(\g_{-2}) \otimes H^{1,1}(G_{A}) =S(\g_{-2}) \otimes \langle  [s] \rangle.
\end{gather*}
By Lemma \ref{lemmaesse}, the morphism $\nabla^{(0)}:E^{0}(M_{A})^{1,1}\longrightarrow E^{0}(M_{A})^{0,0}$ maps $[s]$ to $8\Theta \otimes[1]$. Therefore $\nabla^{(0)}$ is injective and $E^{1}(M_{A})^{1,1} \cong 0$. Thus $E^{\infty}(M_{A})^{1,1} \cong E^{1}(M_{A})^{1,1}=0$ as $\mathcal{W}-$modules, and hence as $\g-$modules.
\end{proof}
\begin{rem}
We point out that for $C=0$, the study of finite irreducible modules over $K'_{4}$ reduces to the study of finite irreducible modules over $K_{4}$, already studied in \cite{kac1}. In particular, for $C=0$, the diagram of maps between finite Verma modules reduces to the diagonal $m=n$ in the quadrants \textbf{A} and \textbf{C} of Figure \ref{figura}. For $K_{4}$ the homology had been already computed in \cite[Propositions 6.2, 6.4]{kac1} using de Rham complexes. Propositions \ref{omologiaC} and \ref{omologiaA} are coherent with the results of \cite[Propositions 6.2, 6.4]{kac1} for $K_{4}$.
\end{rem}
\begin{proof}[Proof of Theorem \ref{thmomologia}]
The proof follows combining the results of Propositions \ref{omologiafinale}, \ref{omologiafinaleduale}, \ref{omologiaC} and \ref{omologiaA}.
\end{proof}
\section{Size}
The aim of this section is to compute the size of the irreducible quotients $I(m,n,\mu_{t},\mu_{C})$, where $(m,n,\mu_{t},\mu_{C})$ occurs among the weights in Theorems \ref{sing1}, \ref{sing2}, \ref{sing3}. This computation is an application of Theorem \ref{thmomologia}.
For a $S(\g_{-2})-$module $V$, we define its size as (see \cite{kacrudakovE36}):
\begin{align*}
\size(V)=\frac{1}{4}\rk_{S(g_{-2})}V.
\end{align*}
\begin{prop}
\label{taglia}
   \begin{itemize} 
 \item[A)] $\size(I(m,n,-\frac{m+n}{2},\frac{m-n}{2}))=2mn+m+n$,\\ 
\item[B)] $\size(I(m,n,1+\frac{m-n}{2},-1-\frac{m+n}{2}))=2(m+1)(n-1)+n-1+3m+3+2=2mn+m+3n+2$,\\
\item[C)] $\size(I(m,n,\frac{m+n}{2}+2,\frac{n-m}{2}))=2(m+1)(n+1)+m+n+2=2mn+3m+3n+4$,\\
\item[D)] $\size(I(m,n,1+\frac{n-m}{2},1+\frac{m+n}{2}))=2mn+n+3m+2$.
\end{itemize}
\end{prop}
In order to prove Proposition \ref{taglia} we need some preliminary results.
We will say that $I(m,n,\mu_{t},\mu_{C})$ is of type X if $M(m,n,\mu_{t},\mu_{C})$ is represented in quadrant \textbf{X} in Figure \ref{figura}. In the following we will also use the notation $\nabla^{m,n}_{X}$ for the morphism $\nabla_{|M_{X}^{m,n}}: M_{X}^{m,n}\rightarrow M_{X}^{m-1,n-1}$ defined as in \eqref{nabla}, in order to make explicit the dependence on the domain.
\begin{rem}
\label{shifted}
We point out that it is sufficient to compute the size for modules $I(m,n,-\frac{m+n}{2},\frac{m-n}{2})$ of type A and $I(m,n,1+\frac{n-m}{2},1+\frac{m+n}{2})$ of type D and use conformal duality, since conformal dual modules have the same size.\\
 Let us show that the module $I(m,n,\frac{m+n}{2}+2,\frac{n-m}{2})$ of type C is the conformal dual of $I(m+1,n+1,-\frac{m+n+2}{2},\frac{m-n}{2})$ of type A, , when $(m,n) \neq (0,0)$. Indeed, by Remarks \ref{costruzionemorfismi} and \ref{cantacasellikacremarksize}, we have the following dual maps, for $m,n\geq 0$:
\begin{align*}
\nabla^{m+1,n+1}_{A}&:M^{m+1,n+1}_{A}\longrightarrow M^{m,n}_{A},\\
\nabla^{-m,-n}_{C}&:M^{-m,-n}_{C}\longrightarrow M^{-m-1,-n-1}_{C}.
\end{align*}
We use Remark \ref{cantacasellikacremarksize} and Theorem \ref{dualeconforme} with $T:=\nabla^{-m,-n}_{C}$, $M:=M^{-m,-n}_{C}$ and $N:=M^{-m-1,-n-1}_{C}$. 
We point out that we can apply Theorem \ref{dualeconforme}, because we know that $M^{-m-1,-n-1}_{C}/  \Ima (\nabla^{-m,-n}_{C})$ is a finitely generated torsion$-$free $\C[\Theta]-$module. Indeed, by Theorem \ref{thmomologia}, the complex of type C is exact in $M^{-m-1,-n-1}_{C}$ if and only if $(-m-1,-n-1) \neq (-1,-1)$. Therefore:
\begin{gather*}
\frac{M^{-m-1,-n-1}_{C}}{\Ima (\nabla_{C}^{-m,-n})}=\frac{M^{-m-1,-n-1}_{C}}{\Ker(\nabla_{C}^{-m-1,-n-1})} \cong \Ima (\nabla_{C}^{-m-1,-n-1})
\end{gather*}
 and $\Ima (\nabla_{C}^{-m-1,-n-1})$ is a submodule of the free module $M^{-m-2,-n-2}_{C}$, thus it is torsion$-$free as a $\C[\Theta]-$module. We have that $M/ \Ker T=M^{-m,-n}_{C}/ \Ker(\nabla^{-m,-n}_{C}) \cong I(m,n,\frac{m+n}{2}+2,\frac{n-m}{2})$ is the dual of $N^{*} /\Ker T^{*}\cong \Ima T^{*} \cong I(m+1,n+1,-\frac{m+n+2}{2},\frac{m-n}{2})$.\\
Using the same argument, it is possible to show that the module $I(m,n,1+\frac{m-n}{2},-1-\frac{m+n}{2})$ of type B is the conformal dual of $I(m+1,n-1,1+\frac{n-m-2}{2},1+\frac{m+n}{2})$ of type D.
\end{rem}
\subsection{The character}
We now introduce the notion of character, that will be used for the computation of the size. Let $s$ be an indeterminate. We define the character of a $\g-$module $V$, following \cite{kacrudakovE36}, as:
\begin{align*} 
\ch V=\tr_{V}s^{-t}.
\end{align*}
The character is a Laurent series in the indeterminate $s$; the coefficient of $s^{k}$ is the dimension of the eigenspace of $V$ of eigenvalue $k$ with respect to the action of $-t \in \g_{0}$.
\begin{rem} 
\label{charquoz}
Let $V$ be a $\g-$module and $W$ a $\g-$submodule of $V$. It is straightforward that $\ch V/W=\ch V-\ch W$.
\end{rem}
We now compute directly the character of a Verma module $M(m,n,\mu_{t},\mu_{C}) \cong U(\g_{<0}) \otimes F(m,n,\mu_{t},\mu_{C})$. Using that $-t$ acts on elements of $\g_{-2}$ as the multiplication by 2 and on elements of $\g_{-1}$ as the multiplication by 1 and that if $|s|^{2}<1$:
\begin{align*} 
\sum^{\infty}_{k=0} s^{2k}=\frac{1}{1-s^{2}},
\end{align*} 
 we obtain that, if $-1<s<1$:
 \begin{align} 
\label{carattereverma}
\ch M(m,n,\mu_{t},\mu_{C})=s^{-\mu_{t}}\dime F(m,n,\mu_{t},\mu_{C}) \cdot \frac{(1+s)^{4}}{1-s^{2}}.
\end{align} 
For the computation of the size of a $\g-$module $V$ we use that (see \cite{kacrudakovE36}):
\begin{align}
\label{keycaratteri}
\size(V)=\frac{1}{4} \lim_{s \rightarrow 1} (1-s^{2})\ch V.
\end{align}
\begin{prop}
\label{caratteri}
The character of $I(m,n,-\frac{m+n}{2},-\frac{n-m}{2})$ of type A is, if $(m,n) \neq (0,0)$:
\begin{align*}
\ch I\Big(m,n,-\frac{m+n}{2},\frac{m-n}{2}\Big)=s^{\frac{m+n}{2}}\frac{(1+s)^{4}}{1-s^{2}} \Big[\frac{2}{(1+s)^{3}}+\frac{m+n-1}{(1+s)^{2}}+\frac{mn}{1+s}\Big].
\end{align*}
The character of $I\Big(m,n,1+\frac{n-m}{2},1+\frac{m+n}{2}\Big)$ of type D is:
\begin{align*}
\ch I\Big(m,n,1+\frac{n-m}{2},1+\frac{m+n}{2}\Big)=&s^{-1-\frac{n-m}{2}}\frac{(1+s)^{4}}{1-s^{2}} \Big[\frac{-2}{(1+s)^{3}}+\frac{3+n-m}{(1+s)^{2}}+\frac{mn+2m}{1+s}\\
& -(-1)^{n+1}s^{n+1}\Big(\frac{-2}{(1+s)^{3}}+\frac{-m-n+1}{(1+s)^{2}}+\frac{m+n+1}{1+s}\Big)\Big]\\
&+s^{\frac{m+n+2}{2}}\frac{(1+s)^{4}}{1-s^{2}} (-1)^{n+1}\Big(\frac{2}{(1+s)^{3}}+\frac{m+n+1}{(1+s)^{2}}\Big).
\end{align*}
\end{prop}
\begin{proof}
We compute the character of modules $I(m,n,\mu_{t},\mu_{C})$ using the character of $M(m,n,\mu_{t},\mu_{C})$. Let us now focus on the case $I(m,n,-\frac{m+n}{2},-\frac{n-m}{2})$ of type A. By Theorem \ref{thmomologia}, the following is an exact sequence, if $(m,n) \neq (0,0)$:
\begin{align*}
&\dots \longrightarrow M\Big(m+j,n+j,-\frac{m+n+2j}{2},\frac{m-n}{2}\Big)\longrightarrow \dots \longrightarrow M\Big(m+1,n+1,-\frac{m+n+2}{2},\frac{m-n}{2}\Big)\\
&\longrightarrow M\Big(m,n,-\frac{m+n}{2},\frac{m-n}{2}\Big)\longrightarrow I\Big(m,n,-\frac{m+n}{2},\frac{m-n}{2}\Big)\longrightarrow 0.
\end{align*}
Hence, using the exactness of the previous sequence, Remark \ref{charquoz} and \eqref{carattereverma}, we obtain that:
\begin{align*}
\ch I\Big(m,n,-\frac{m+n}{2},\frac{m-n}{2}\Big)=s^{\frac{m+n}{2}}\frac{(1+s)^{4}}{1-s^{2}} \sum^{\infty}_{j=0}(-1)^{j}s^{j}(j+m+1)(j+n+1).
\end{align*}
We need the following identity, that holds if $|s|<1$ and is a consequence of the binomial series:
\begin{align}
\label{appoggiobinmiale}
\sum^{\infty}_{j=0}(-1)^{j}s^{j} \binom{j+m}{m}=\frac{1}{(1+s)^{m+1}}.
\end{align}
By the fact that $(j+m+1)(j+n+1)=(j+2)(j+1)+(j+1)(m+n-1)+mn$ and \eqref{appoggiobinmiale}, we get:
\begin{align*}
\ch I\Big(m,n,-\frac{m+n}{2},\frac{m-n}{2}\Big)=s^{\frac{m+n}{2}}\frac{(1+s)^{4}}{1-s^{2}} \Big(\frac{2}{(1+s)^{3}}+\frac{m+n-1}{(1+s)^{2}}+\frac{mn}{1+s}\Big).
\end{align*}
Now we compute the character for modules $I(m,n,1+\frac{n-m}{2},1+\frac{m+n}{2})$ of type D. By Theorem \ref{thmomologia} the following is an exact sequence:
\begin{gather*}
 \rightarrow M\Big(m+n+2+j,j,-\frac{m+n+2+2j}{2},\frac{m+n+2}{2}\Big)\rightarrow \dots \rightarrow M\Big(m+n+2,0,-\frac{m+n+2}{2},\frac{m+n+2}{2}\Big) \\
 \rightarrow M\Big(m+n,0,1+\frac{-m-n}{2},1+\frac{m+n}{2}\Big)\rightarrow M\Big(m+n-1,1,1+\frac{-m-n+2}{2},1+\frac{m+n}{2}\Big) \rightarrow \dots  \\
\rightarrow M\Big(m,n,1+\frac{n-m}{2},1+\frac{m+n}{2}\Big) \rightarrow I\Big(m,n,1+\frac{n-m}{2},1+\frac{m+n}{2}\Big)\rightarrow 0,
\end{gather*}
where the first row is composed of modules of type A and the following terms are of type D.
Hence, by the exactness of the previous sequence, Remark \ref{charquoz} and \eqref{carattereverma}, we obtain that:
\begin{align*}
\ch I\Big(m,n,1+\frac{n-m}{2},1+\frac{m+n}{2}\Big)=&s^{-1-\frac{n-m}{2}}\frac{(1+s)^{4}}{1-s^{2}} \sum^{n}_{j=0}(-1)^{j}s^{j}(j+m+1)(n-j+1)\\ 
&+s^{\frac{m+n+2}{2}}\frac{(1+s)^{4}}{1-s^{2}} \sum^{\infty}_{i=0}(-1)^{n+1+i}s^{i}(i+m+n+2+1)(i+1).
\end{align*}
We use the identity:
\begin{align}
\label{appcharD}
&\sum^{n}_{j=0}(-1)^{j}s^{j}(j+m+1)(n-j+1)=\\ \nonumber
&=\sum^{\infty}_{j=0}(-1)^{j}s^{j}(j+m+1)(n-j+1)-\sum^{\infty}_{j=n+1}(-1)^{j}s^{j}(j+m+1)(n-j+1).
\end{align}
Let us compute the first series in the RHS of \eqref{appcharD}; by the fact that $(j+m+1)(n-j+1)=-(j+1)(j+2)+(3+n-m)(j+1)+2m+mn$ and \eqref{appoggiobinmiale}, we obtain that if $|s|<1$:
\begin{align*}
\sum^{\infty}_{j=0}(-1)^{j}s^{j}(j+m+1)(n-j+1)=\frac{-2}{(1+s)^{3}}+\frac{3+n-m}{(1+s)^{2}}+\frac{mn+2m}{1+s}.
\end{align*}
Let us compute the second series in the RHS of \eqref{appcharD}; we have that:
\begin{align*}
-\sum^{\infty}_{j=n+1}(-1)^{j}s^{j}(j+m+1)(n-j+1)=&-\sum^{\infty}_{k=0}(-1)^{k+n+1}s^{k+n+1}(k+n+1+m+1)(n-k-n-1+1).
\end{align*}
By the fact that $(k+n+1+m+1)(n-k-n-1+1)=-(k+2)(k+1)+(k+1)(-m-n+1)+m+n+1$ and \eqref{appoggiobinmiale}, we obtain that if $|s|<1$:
\begin{align*}
-\sum^{\infty}_{j=n+1}(-1)^{j}s^{j}(j+m+1)(n-j+1)=&-\sum^{\infty}_{k=0}(-1)^{k+n+1}s^{k+n+1}(k+n+m+2)(-k)=\\
&-(-1)^{n+1}s^{n+1}\Big(\frac{-2}{(1+s)^{3}}+\frac{-m-n+1}{(1+s)^{2}}+\frac{m+n+1}{1+s}\Big).
\end{align*}
Finally, we sum up the previous computations and get:
\begin{align*}
\ch I\Big(m,n,1+\frac{n-m}{2},1+\frac{m+n}{2}\Big)=&s^{-1-\frac{n-m}{2}}\frac{(1+s)^{4}}{1-s^{2}} \Big[\frac{-2}{(1+s)^{3}}+\frac{3+n-m}{(1+s)^{2}}+\frac{mn+2m}{1+s}\\
& -(-1)^{n+1}s^{n+1}\Big(\frac{-2}{(1+s)^{3}}+\frac{-m-n+1}{(1+s)^{2}}+\frac{m+n+1}{1+s}\Big)\Big]\\
&+s^{\frac{m+n+2}{2}}\frac{(1+s)^{4}}{1-s^{2}} (-1)^{n+1}\Big(\frac{2}{(1+s)^{3}}+\frac{m+n+1}{(1+s)^{2}}\Big).
\end{align*}
\end{proof}
\begin{proof}[Proof of Proposition \ref{taglia}]
We first focus on $I(0,0,0,0)$ of type A. We have that $\size(I(0,0,0,0))=0$. Indeed by Theorem \ref{thmomologia} the following is an exact sequence:
\begin{align*}
& \rightarrow M(j,j,-j,0)\rightarrow \dots \rightarrow M(1,1,-1,0)\xrightarrow{\nabla} M(0,0,0,0)\xrightarrow{\phi} I(0,0,0,0)\rightarrow 0,
\end{align*}
where $\phi$ is the projection to the quotient $I(0,0,0,0) \cong \frac{M(0,0,0,0)}{\Ima \nabla}$. Therefore $\size(I(0,0,0,0))=0$ follows by the same computations used in Proposition \ref{caratteri} for case A.\\ Now let us compute the size of $I(0,0,2,0)$ of type C. Since $M(0,0,2,0)$ is irreducible by Proposition \ref{M(0,0,2,0)}, we obtain that $\size(I(0,0,2,0))=\size(M(0,0,2,0))=4$ by \eqref{carattereverma}. Finally the sizes of $I(m,n,\mu_{t},\mu_{C})$ of type A for $(m,n)\neq (0,0)$ and of type D follow directly from Proposition \ref{caratteri} and \eqref{keycaratteri}. The sizes of $I(m,n,\mu_{t},\mu_{C})$ of type C for $(m,n)\neq (0,0)$ and of type B follow from Remark \ref{shifted}.
\end{proof}
\begin{ack}
	The author would like to thank Nicoletta Cantarini, Fabrizio Caselli and Victor Kac for useful comments and suggestions.
\end{ack}

\end{document}